%% file: Homology_Inclusion_Line_Arrangements.tex
\documentclass[10pt,british,oneside,a4paper]{amsart}

\usepackage[utf8]{inputenc}
\usepackage[T1]{fontenc}
\usepackage{fix-cm}
\usepackage{floatrow}
\usepackage{caption}
\usepackage{microtype}
\usepackage[a4paper]{geometry}
\usepackage[british]{babel}
\usepackage{xspace}
\usepackage{amsmath}
\usepackage{amsfonts}
\usepackage{amssymb}
\usepackage{amsthm}
\usepackage{thmtools}
\usepackage[xy]{qsymbols}
\usepackage{bm}
\usepackage{xpatch}
\xapptocmd\bfseries{\boldmath}{}{}
\usepackage[inline,shortlabels]{enumitem}
\usepackage{graphicx}
\usepackage{mathtools}
\usepackage{wasysym}
\usepackage{extarrows}
\usepackage{faktor}
\usepackage{pgf,tikz}
\usepackage{tikz-cd}
\usepackage{array}
\usepackage{multirow}
\usepackage{multicol}
\usepackage{tabularx}
\usepackage{mathrsfs}
\usepackage{yhmath}
\usepackage{subcaption}
\usepackage{afterpage}
\usepackage[section]{placeins}
\usepackage[backend=biber,autolang=hyphen,citestyle=alphabetic,bibstyle=alphabetic,language=british,sorting=nyvt,maxnames=5,isbn=false,doi=false,url=true,date=year]{biblatex}
\usepackage[autostyle,english=british]{csquotes}
\usepackage[colorlinks,linktoc=all,pdfdisplaydoctitle,bookmarksnumbered,
linkcolor={red!80!black},
citecolor={blue!50!black},
urlcolor={magenta!80!black}]{hyperref}
\usepackage[nameinlink,capitalise]{cleveref}
\numberwithin{equation}{section}
\hypersetup{hypertexnames=false} %
\usepackage{widows-and-orphans}

\usetikzlibrary{positioning,topaths,decorations.pathreplacing,decorations.pathmorphing,decorations.markings,decorations.text,quotes,babel,arrows.meta,graphs,matrix,shapes.misc,bending,braids,knots,celtic,calc,fadings,3d,perspective,shapes.multipart,backgrounds}

\declaretheoremstyle[bodyfont=\itshape,qed=$\vartriangleleft$]{thm-qeded}
\declaretheoremstyle[bodyfont=\normalfont,qed=$\Diamond$]{def-qeded}
\declaretheoremstyle[bodyfont=\normalfont,qed={\hexagon}]{ex-qeded}
\declaretheoremstyle[headfont=\bfseries\itshape,bodyfont=\normalfont,qed={\wasylozenge}]{rem-qeded}

\declaretheorem[style=thm-qeded,numberwithin=section]{theorem}
\declaretheorem[style=thm-qeded,sibling=theorem]{corollary,lemma,proposition,conjecture}
\declaretheorem[style=thm-qeded,numbered=no]{claim,question}
\declaretheorem[style=def-qeded,sibling=theorem]{definition,construction}
\declaretheorem[style=ex-qeded,sibling=theorem]{example}
\declaretheorem[style=rem-qeded,sibling=theorem]{remark}

\newcommand{\A}{\mathcal{A}}
\newcommand{\Id}{\mathrm{Id}}
\newcommand{\Hom}[1]{\mathrm{Homeo}\left(#1\right)}
\newcommand{\HGpp}[1]{\mathrm{Homeo}_{\Gamma}^{++}\left(#1\right)}
\newcommand{\LS}[1]{\mathrm{LS}({#1})}
\newcommand{\CS}[1]{\mathrm{CS}({#1})}
\newcommand{\OLS}[2]{\mathrm{LS}_{#2}(#1)}
\newcommand{\OCS}[2]{\mathrm{CS}_{#2}(#1)}
\newcommand{\RG}[1]{\Gamma\left({#1}\right)}
\newcommand{\FG}[1]{\widehat{\Gamma}\left({#1}\right)}
\newcommand{\MCG}[1]{\mathcal{M}\left(#1\right)}
\newcommand{\PMCG}[1]{\mathcal{P}\left(#1\right)}
\newcommand{\Lin}[1]{\mathfrak{L}\left(#1\right)}
\newcommand{\Circ}[1]{\mathfrak{C}\left(#1\right)}

\newcommand{\OGE}[2]{\mathrm{E}_{#1}(#2)}
\newcommand{\Pres}[3]{P_{#1}\left(#2,#3\right)}
\newcommand{\MH}[1]{M(#1)}
\newcommand{\PI}{\mathcal{G}}
\newcommand{\DiffE}{\vec{\nabla}}
\newcommand{\DiffH}{\nabla}
\newcommand{\DiffS}[1]{\nabla_{#1}}
\newcommand{\TLG}[2]{\mathrm{TLG}(#1,#2)}
\newcommand{\hooklongrightarrow}{\lhook\joinrel\longrightarrow}
\renewcommand{\Re}{\operatorname{Re}}
\renewcommand{\hom}{\operatorname{Hom}}
\newcommand{\Ab}{\operatorname{Ab}}
\renewcommand{\epsilon}{\varepsilon}
\title{Homology inclusion of complex line arrangements}
\author{Adrien Rodau}
\address{Institut de Mathématiques de Marseille, Aix–Marseille Université, Site de Saint Charles, 3 place Victor Hugo, Case 19, 13331 Marseille Cedex 3, France}
\email{adrien.rodau@univ-amu.fr}
\date{}
\begin{document}
\begin{abstract}
	We introduce a new topological invariant of line arrangements in the complex projective plane, derived from the interaction between their complement and the boundary of a regular neighbourhood. The motivation is to identify Zariski pairs which have the same combinatorics but different embeddings. Building on ideas developed by B.~Guerville-Ballé and W.~Cadegan-Schlieper, we consider the inclusion map of the boundary manifold to the exterior and its effect on homology classes. A careful study of the graph Waldhausen structure of the boundary manifold allows to identify specific generators of the homology. Their potential images are encoded in a group, the graph stabiliser, with a nice combinatorial presentation. The invariant related to the inclusion map is an element of this group. Using a computer implementation in \texttt{Sage}, we compute the invariant for some examples and exhibit new Zariski~pairs.
\end{abstract}
\maketitle
\section{Introduction}
\label{sec=intro}

The~study of the topology of plane algebraic curves was initiated by O.~Zariski. The~\emph{combinatorics} of a curve is the topological type of the pair formed by a tubular neighbourhood of the curve and the curve itself. These data are equivalent to the topological type of the singularities and the incidence relations between the components. O.~Zariski \cite{Zariski1931,Zariski1937} and E.~R.~van~Kampen \cite{Kampen1933} have shown that there exists pairs of curves with the same combinatorics but different embeddings in $\mathbb{CP}^2$, which were dubbed \emph{Zariski pairs} by E.~Artal in~\cite{ArtalBartolo1994}.

\emph{Line arrangements} are finite collections of complex lines in $\mathbb{CP}^2$, that is, plane algebraic curves whose irreducible components have degree~one. The~components are non-singular and the singularities all belong to a same class of simple type. The combinatorics of a line arrangement depend only on the incidence relations, which can be encoded in a graph $`G$ called the \emph{incidence~graph}. The study of line arrangements provides a favourable setting to create topological methods and invariants that could then be extended to algebraic curves in~general. It~also offers some interesting questions in~itself. The~first Zariski pair of line arrangements was constructed by G.~Rybnikov~\cite{Rybnikov1998} and studied in detail by E.~Artal, J.~Carmona, J.~I.~Cogolludo and M.~Á.~Marco in~\cite{ArtalBartolo2006} using the fundamental groups of the complements. The existence of this pair showed that the combinatorics does not determine the topological type of a curve even in the simplest case of line~arrangements. The~search for more Zariski pairs and a finer comprehension of the relationship between combinatorics and topology of curves and line arrangements has been a very active topic since the~2000s. Notably,~S.~Nazir, M.~Yoshinaga \cite{Nazir2012} and F.~Ye \cite{Ye2013} have completely determined the isotopy classes of line arrangements up to $9$ lines, and in particular that no Zariski pairs exist for these~values. We~also mention the works of B.~Guerville-Ballé \cite{GuervilleBalle2016} and J.~Viu~Sos \cite{GuervilleBalle2019} who discovered several new Zariski pairs with more than $11$~lines. Readers~can refer to \cite{ArtalBartolo2008,GuervilleBalle2022} for a more detailed review of the~subject.

A wide variety of common invariants from algebraic topology have been applied to the study of line arrangements and Zariski~pairs. Consider the exterior $E_{\A}$ of the arrangement $\A$ in $\mathbb{CP}^2$. A~direct comparison of the fundamental groups of the exteriors using the \emph{Zariski-van~Kampen~method}~\cite{Kampen1933} can sometimes give a Zariski pair, as for the original example of~G.~Rybnikov. However,~there are known examples of Zariski pairs where the fundamental groups of the exteriors are isomorphic (see \cite{Shimada2009,shirane19,triangular,GuervilleBalle2020}). It has since been shown that not even the characteristic varieties determine the topology of line arrangements in general, including for the subcategory of arrangements with real~equations.

A possible approach to build invariants of line arrangements is to consider the \emph{boundary manifold} $B_{\A} \vcentcolon= \partial E_{\A}$ of the arrangement. T.~Jiang, S.~S.-T.~Yau~\cite{Jiang} and E.~R.~Westlund \cite{Westlund1997} have shown that this boundary manifold has the structure of a \emph{graph manifold} as defined by F.~Waldhausen \cite{Waldhausen1967a,Waldhausen1967} and W.~D.~Neumann \cite{Neumann1981}. The incidence graph $`G$ provides a \enquote{blueprint} to reconstruct the boundary manifold by gluing together circle bundles (Seifert pieces) corresponding to the boundary of local neighbourhoods around each line component and each singularity of the~arrangement.

Our own interest lies in the study of the inclusion
\[i_{\A} : B_{\A} \hooklongrightarrow E_{\A}\]
of the boundary manifold inside the exterior of a line arrangement~$\A$. E.~Hironaka \cite{Hironaka2001} studied the morphism induced by the inclusion on the fundamental groups for the case of real line arrangements, using methods to compute the presentations of the groups due to W.~Arvola~\cite{Arvola1992} for the exterior and E.~R.~Westlund \cite{Westlund1997} for the boundary. This study was continued and generalised to complex line arrangements by V.~Florens, B.~Guerville-Ballé and M.~Á.~Marco in~\cite{Florens2013}. Their~results exposed that the study of the fundamental group \enquote{inclusion} is made difficult by the algebraic complexity of the morphisms~involved. Another~approach is to consider the morphism induced by the inclusion $i_{\A}$ on the \emph{first homology groups}~instead:
\[i_{\A}^* : H_1(B_{\A},\mathbb{Z}) \longrightarrow H_1(E_{\A},\mathbb{Z})\]
Unlike~what one might think, this is not a trivial~matter. The~study of that morphism was first considered by E.~Artal, V.~Florens and B.~Guerville-Ballé in \cite{Bartolo2014} in the form of the \emph{$\mathcal{I}$\nobreakdash-invariant}, which they used to obtain new Zariski pairs. This invariant was later generalised as the \emph{loop-linking number} of~W.~Cadegan-Schlieper~\cite{Cadegan2018} and further developed by B.~Guerville-Ballé in \cite{GuervilleBalle2022}.

We propose a new invariant that extends these constructions and fully exploits the homology \enquote{inclusion} of the boundary manifold of any line arrangement inside its~exterior. The~main principle of the construction is as~follows. The~first homology groups of the boundary manifold and the exterior are both combinatorially determined. However, the induced morphism $i_{\A}^*$ still contains significant topological~information. The~main difficulty of describing $i_{\A}^*$ lies in the ambiguity of defining a set of generators on the boundary. Such a set can be obtained from a specific class of embeddings of the graph inside the boundary~manifold. These \emph{graphed embeddings} depend on a \emph{graph ordering} $`W$. We then define the \emph{graph stabiliser} as a combinatorial group which computes the homological differences between all ordered graphed~embeddings. This allows to remove the ambiguity on the generators and is the main novelty of our construction. The~morphism~$i_{\A}^*$ induces an element of the graph~stabiliser called \emph{homology inclusion} which is a topological invariant of ordered oriented line arrangements. Using a computer program written in \texttt{Sage}~\cite{sagemath} in collaboration with B.~Guerville-Ballé and E.~Artal, we have computed homology inclusion values and obtained a new Zariski quadruplet of oriented line arrangements with $11$ lines. In particular, four pairs within the quadruplet are Zariski pairs.

\Cref{sec=comb} gives a presentation of line arrangements, with a specific focus on combinatorics and related concepts, in particular the orderings which are key in the construction. In \Cref{sec=bm} we recall the structure of the boundary manifold as a graph manifold. In~\cref{sec=star} we present the \emph{flat stars} which are the \enquote{elementary bricks} used to build the graphed embeddings, which are themselves defined in \cref{sec=oge}. In \Cref{sec=hom-incl}, we define the graph stabiliser group and the homology inclusion invariant. \cref{sec=gs} is dedicated to determining of a presentation of the graph stabiliser using results from \cref{sec=star,sec=oge}. This presentation is then used in \Cref{sec=lln} to establish that the homology inclusion extends the loop-linking number. Finally \cref{sec=calc} summarises the results of the computations made with the invariant and presents the new oriented Zariski quadruplet.

The computation of the homology inclusion makes use of specific results and algorithms which are not immediately necessary for its definition. We therefore delegate the presentation of the details of this computation to another publication.

In addition, we mention that the homology inclusion invariant represents a preliminary step in the study of the \emph{twisted} homology inclusion of complex line arrangements. Consider the morphism induced by the inclusion $i_{\A}$ between the \emph{twisted} homology groups of the boundary and exterior manifolds. To compute the twisted homology group of $B_{\A}$ one needs to restrict abelian representations of the fundamental group of the exterior to the boundary. The homology inclusion precisely allows to compute these values. The twisted homology inclusion invariant could potentially be used to discriminate equality cases of the homology inclusion invariant and thus detect more Zariski pairs. This new invariant will be the object of a future publication.

\subsection*{Acknowledgments}
The author would like to thank his former PhD advisors Enrique Artal and Vincent Florens for their support and reviewing of this research. Many thanks also to Benoît Guerville-Ballé for his advising and for his invaluable contribution to the computations of the invariant.
\section{Combinatorics of complex line arrangements}
\label{sec=comb}

\subsection{Generalities}
\label{ssec=la}

A \emph{complex line arrangement} $\A$ is a union of complex lines on the projective complex plane $\mathbb{CP}^2$.
The intersections are called the \emph{singular points}.

\begin{definition}
	\label{def=la-equiv}
	A \emph{topological equivalence} between two complex line arrangements $\mathcal{A}$ and $\mathcal{A}'$ in $\mathbb{CP}^2$ is an homeomorphism
	\[\begin{tikzcd}[cramped]
		`F : (\mathbb{CP}^2, \mathcal{A}) \rar["\sim"] & (\mathbb{CP}^2, \mathcal{A}')
	\end{tikzcd}\qedhere\]
\end{definition}
If $`F$ respects the orientation of all the components of $\A$ then the equivalence is \emph{positive}, or \emph{oriented}. A topological equivalence between two arrangements induces a bijection between the sets of lines and between the sets of singular points. This fact is made precise using the concept of \emph{combinatorics}.
\begin{definition}
	\label{def=comb}
	A \emph{line combinatorics} is a triple $C = (\mathcal{L},\mathcal{Q},`:)$ where $\mathcal{L}$ is a finite set, $\mathcal{Q}$ is a finite subset of $\mathcal{P}(\mathcal{L})$ and $`:$ is a relation from $\mathcal{Q}$ to $\mathcal{L}$ such that:
	\begin{enumerate}[(i)]
		\item For every element $P `: \mathcal{Q}$, there exist at least two distinct elements $L, L' `: \mathcal{L}$ such that $P `: L$ and $P `: L'$.
		\item For every pair $L,L' `: \mathcal{L}$ with $L \neq L'$ there exists a unique $P `: \mathcal{Q}$ such that $P `: L$ and~$P `: L'$. \qedhere
	\end{enumerate}
\end{definition}
For a line arrangement $\mathcal{A}$, there is a natural line combinatorics $C_{\mathcal{A}}$ associated with $\mathcal{A}$, given by the incidence of the lines. Combinatorics can also be encoded with graphs. Let $\mathcal{Q}^{>2}$ be the subset of singular points with multiplicity $> 2$.
\begin{definition}
	\label{def=graph-incid-red}
	The \emph{incidence graph} $\RG{C}$ of the combinatorics $C$ is defined by the following description:
	\begin{description}[font={\normalfont\itshape},labelindent=2em,labelsep=.4em,itemsep=3pt]
		\item[Vertices] one vertex for each element of $\mathcal{L} `U \mathcal{Q}^{>2}$.
		\item[Edges]\
		\begin{enumerate}[(i)]
			\item $L `: \mathcal{L}$ and $P `: \mathcal{Q}^{>2}$ are linked by an edge $e_{P,L}$ if and only if $P `: L$. \label{edge-norm}
			\item $L, L' `: \mathcal{L}$ are linked by an edge $e_{L,L'}$ if and only if $L \cap L'$ has multiplicity~$2$. \label{edge-exc}\qedhere
		\end{enumerate}
	\end{description}
\end{definition}
The combinatorics $C$ is equivalent to the data of the graph $\RG{C}$ up to every automorphism that respects the subsets $\mathcal{L}$ and $\mathcal{Q}^{>2}$.
\begin{remark}
	\label{rem:graph-full}
	We sometimes also consider the \emph{full} incidence graph $\FG{C}$, with vertex set $\mathcal{L} `U \mathcal{Q}$, and edge set given by the description: $L `: \mathcal{L}$ and $P `: \mathcal{Q}$ are linked by an edge $e_{P,L}$ if and only if~$P `: L$.
\end{remark}

Any topological equivalence $`F : (\mathbb{CP}^2, \mathcal{A}) "->" (\mathbb{CP}^2, \mathcal{A}')$ induces an automorphism
\begin{equation}
	\label{graph-aut-init}
	\begin{tikzcd}[cramped]
		G(`F) : \RG{C_{\A}} \rar["\sim"] & \RG{C_{\A'}}
	\end{tikzcd}
\end{equation}
see \cite{Jiang}. In particular, the combinatorics is a topological invariant of line arrangements. An~arrangement equivalence such that $G(`F) = \Id$ is called an \emph{ordered} equivalence.

If two arrangements with the same combinatorics are such that there exists no equivalence between them, they form a \emph{Zariski pair}. If there is no \emph{ordered} (resp. \emph{oriented}) equivalence, they form an \emph{ordered} (resp. \emph{oriented}) Zariski pair.

As an example, consider the two MacLane arrangements $\mathcal{M}^+$ and $\mathcal{M}^-$ with $8$ lines given by the equations:
\begin{align*}
	L_0 & : 0=z    & L_3 & : 0=y                 & L_6 & : 0=-x-`w^2 y+ z \\
	L_1 & : 0=-x+z & L_4 & : 0=`w^2 x + `w y + z & L_7 & : 0=`w y + z     \\
	L_2 & : 0=x    & L_5 & : 0=-x+y              &     &
\end{align*}
where $`w = e^{\frac{2i`p}{3}}$ for $\mathcal{M}^+$ and $`w = e^{-\frac{2i`p}{3}}$ for $\mathcal{M}^-$. They share a common combinatorics whose incidence graph is given in \Cref{fig=maclane-gr}. The two MacLane arrangements form the smallest possible \emph{ordered oriented} Zariski pair of line arrangements. However, there still exists an oriented equivalence between them that does not induce the identity on the graph, and in addition the complex conjugation in $\mathbb{CP}^2$ preserves the graph but not the orientation of the lines, see \cite{MacLane1936}. Therefore the MacLane arrangements do not form neither an oriented nor an ordered Zariski pair, and in particular not a Zariski pair.
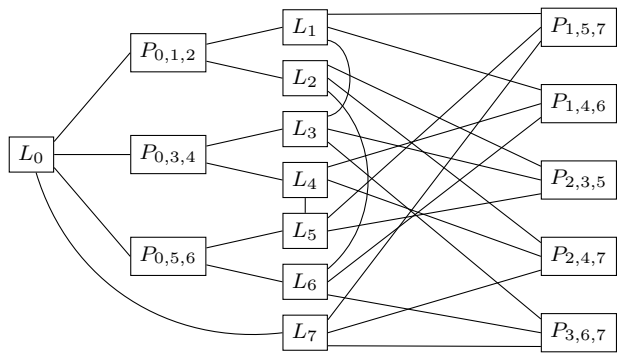
\begin{figure}[htb]
	\input{MacLane_graph.pgf}
	\caption{Incidence graph of the MacLane combinatorics}
	\label{fig=maclane-gr}
\end{figure}

\subsection{Orderings}
\label{ssec=ord}

Let $W$ be a finite set with cardinal $m \geq 2$.  A \emph{linear order} on $W$ is a bijection $`q : W "->" \{1,\dots,m\}$. The set of all linear orders on $W$ is denoted by $\Lin{W}$. Consider the action of $\mathfrak{S}_m$ on the set~$\{1,\dots,m\}$. The cyclic subgroup generated by the circular permutation $(1 \cdots m)$ induces a left free action on the same set. This in turn induces a left free action of $\left<(1 \cdots m)\right>$ on $\Lin{W}$ given by, for all $`q `: \Lin{W}$ and $r `: \mathbb{Z}$:
\[{(1 \cdots m)}^r `. `q = {(1 \cdots m)}^r `o `q\]
The quotient set of $\Lin{W}$ by this action is denoted by $\Circ{W}$ and its elements are called \emph{circular orders} on~$W$. We denote the quotient map by
\begin{equation}
	\label{ord-circ}
	\begin{tikzcd}[cramped]
		c : \Lin{W} \rar[two heads] & \Circ{W}
	\end{tikzcd}
\end{equation}
For every $`w `: \Circ{W}$, the $m$ elements of $c^{-1}(`w)$ are called the \emph{linearisations} of the circular order~$`w$. The set $\Circ{W}$ has cardinal ${(m-1)!}$.

Now fix an element $`q_0 `: \Lin{W}$. For all $`q `: \Lin{W}$ there is a unique element $`s$ of the permutation group $\mathfrak{S}_m$ such that $`q = `s `o `q_0$. There is a free transitive right action of $\mathfrak{S}_m$ on $\Lin{W}$ given by, for all $`t `: \mathfrak{S}_m$:
\[`q `. `t = `s `o `t `o `q_0\]
\vspace{-\baselineskip}
\begin{lemma}
	\label{lem=circ-act}
	The right action of $\mathfrak{S}_m$ on $\Lin{W}$ respects the left quotient $c : \Lin{W} "->" \Circ{W}$.
\end{lemma}
\begin{proof}
	Let $`q, `q' `: \mathfrak{S}_m$ such that $`q' = {(1 \cdots m)}^r `o `q$ for some $r `: \mathbb{Z}$.
	Then for every $`t `: \mathfrak{S}_m$ one has again:
	\[`q' `. `t = `s' `o `t `o `q_0 = {(1 \cdots m)}^r `o `s `o `t `o `q_0 =  {(1 \cdots m)}^r `o  `q `. `t\qedhere\]
\end{proof}
There is therefore a transitive right action of $\mathfrak{S}_m$ on $\Circ{W}$. The stabiliser of an element $`w `: \Circ{W}$ is the subgroup $\mathfrak{S}_{`w}$ generated by the permutation $(`s(1) \cdots `s(m))$ where $`s `o `q_0$ lies in~$c^{-1}(`w)$.

Consider now a graph $`G$ with vertex set $V(`G)$. For any vertex $v `: V(`G)$, the \emph{neighbour set} $W_v$ is the subset of vertices that are connected to $v$ by an edge. The cardinal $m_v$ of $W_v$ is called the \emph{multiplicity} of the vertex.
\begin{definition}
	\label{def=graph-ord}
	A \emph{graph ordering} of $`G$ is a collection $`W = {{(`w_v)}_{v `: V(`G)}}$ where for every vertex $v$, $`w_v `: \Circ{W_v}$ is a circular order on the neighbour set $W_v$.
\end{definition}
\begin{lemma}
	\label{lem=ord-rest}
	Any graph ordering $\widehat{`W}$ of the full incidence graph $\FG{C_{\A}}$ can naturally be restricted to a graph ordering $`W$ of the incidence graph~$\RG{C_{A}}$.
\end{lemma}
\begin{proof}
	The set of vertices of $\RG{C_{A}}$ is $\mathcal{L} \cup \mathcal{Q}^{>2}$. For every point-vertex $P `: \mathcal{Q}^{>2}$ one has directly $W_P = \widehat{W}_P$. For every line-vertex $L$, there is a natural bijection $\widehat{W}_L "->" {W}_L$ that replaces any vertex $P = L \cap L' `: \mathcal{Q}^{>2}$ with $L'$ itself. The local circular order of $\widehat{`W}$ on $\widehat{W}_L$ thus extends to~${W}_L$.
\end{proof}

\section{Boundary manifold}
\label{sec=bm}

Let $\A$ be an arrangement in $\mathbb{CP}^2$. Denote by $N_{\A}$ a closed tubular neighbourhood of $\A$.
In~\cite{Cohen2006}, D.~Cohen and A.~Suciu give a geometrical construction of $N_{\mathcal{A}}$: write $\A = \{P_{\A} = 0\}$ where $P_{\A}$ is a defining homogenous polynomial of $\A$ in $\mathbb{CP}^2$. Let $`f : \mathbb{CP}^2 \rightarrow \mathbb{R}$ defined~by
\[`f(\mathbf{x}) = {\left[P_{\mathcal{A}}(\mathbf{x})\right]}^2 / {\left\|\mathbf{x}\right\|}^{2(n+1)}\]
Then $N_{\mathcal{A}} \vcentcolon= `f^{-1}([0, `d])$ does not depend on $`d > 0$ for $d$ sufficiently small.

The compact closed 3-manifold $B_{\A} \vcentcolon= \partial N_{\A}$ is called the \emph{boundary manifold} of~$\A$ and has the structure of a graph manifold, see \cite{Jiang}. The aim of this section is to describe this structure.

\subsection{Circle bundles}
\label{ssec=seif}

Let $`S$ be an oriented surface of genus $0$ and let $W$ be a set of $m \geq 0$ marked points on $`S$. Denote by $`S_W$ the surface obtained by removing from $`S$ an open neighbourhood around each point of $W$. A \emph{circle bundle} $(S, `S_W, p)$ is a fibre bundle $p: S \rightarrow `S_W$ with basis $`S_W$, such that the fibres $p^{-1}(*)$ are homeomorphic to $S^1$. For short, we often denote the circle bundle as simply $S$. An \emph{orientation} of $(S, `S_W, p)$ is given by the data of two of the following three:
\begin{enumerate}[(i)]
	\item an orientation on $S$.
	\item an orientation on all fibres of $p$.
	\item an orientation of $`S_W$. \qedhere
\end{enumerate}
Note that the choice of two of these fixes the third, while the choice of one leaves only two possibilities for the other two.

A homeomorphism $`f : S \rightarrow S'$ is \emph{fibrewise} if it commutes with $p$, and \emph{fibre-positive} if it preserves the orientation of the fibres. A circle bundle is a special case of a \emph{Seifert manifold}, namely one without singular fibres. The following results are standard, see for example~\cite{Jaco1979,Fomenko1997}.
\begin{theorem}
	\label{thm=cb-fibre}
	Let $S,S'$ be two Seifert manifolds with more than $3$ boundary components. Then~any homeomorphism $S "->" S'$ is isotopic to a fibrewise one.
\end{theorem}

\begin{example}[Standard circle bundle]
	\label{ex=cb-stand}
	Let $`e `: \mathbb{Z}$ and let $D_W \subset `S_W$ be a disc with $m$ holes centred on the points of $W$. We denote by $\partial^{\infty} D_W$ the $(m+1)$-th boundary component of $D_W$. Consider the oriented product bundle
	$T_W \vcentcolon= D_W `* S^1$
	and an oriented solid torus
	$T_{\infty} \vcentcolon= D' `* S^1$.

	Fix a section $\check{s} : D_W "`->" T_W$, which intersects the boundary $\partial T_W$ on a collection $\check{`m}$ of simple closed curves, with~$\check{`m}^w \vcentcolon= \check{s}(\partial^w D_W)$. Consider a fibre $\check{`l}^{\infty}$ in $T_W$ transverse to~$\check{`m}^{\infty}$. Similarly, let $\check{s}'$ be a section of $T_{\infty}$ and consider a fibre $\check{`l}^{\infty'}$ transverse to $\check{`m}^{\infty'} \vcentcolon= \check{s}'(\partial D')$. Glue $\partial T_{\infty}$ to the toric boundary component $\partial^{\infty} D_W `* S^1$ of $T_W$ using the gluing map:
	\begin{equation*}
		g_{`e} : \left\{
		\begin{aligned}
			\check{`m}^{\infty} &\longmapsto - \check{`m}^{\infty'} - `e `. \check{`l}^{\infty'}\\
			\check{`l}^{\infty} &\longmapsto \check{`l}^{\infty'}
		\end{aligned}
		\right.
	\end{equation*}
	where the gluing respects the orientations of the fibres of $T_W$ and $T_{\infty}$.
	Then $S(W,\check{s},`e) \vcentcolon= T_W `U_{g_{`e}} T_{\infty}$ is a circle bundle over $`S_W$, and $`e$ is called its \emph{Euler number}. The definition obviously does not depend on the choice of $\check{s}'$.
\end{example}

Now let $S$ be any circle bundle and fix a collection of closed curves ${\left(`m^w\right)}_{w`: W}$ on $\partial S$ such that $`m^w \subset \partial^w S$ is transverse to the fibres of $S$. Such curves are called \emph{horizontal}.
\begin{theorem}
	\label{thm=cb-model}
	Let $(S,`S_W,p)$ be a circle bundle with $m$ boundary components and let ${\left(`m^w\right)}_{w`: W}$ be a collection of horizontal curves on the boundary of $S$. Then there exists a unique number $`e `: \mathbb{Z}$ and a positive fibrewise homeomorphism $S(W,\check{s},`e) \rightarrow S$ sending $\check{`m}^w$ to $`m^w$ for every~$w `: W$.
\end{theorem}
By an abuse of language, we call the integer $`e `: \mathbb{Z}$ the \emph{Euler number} of the oriented circle bundle $S$, since the usual Euler number vanishes for bundles with non-empty boundary.

\subsection{Graph structure}
\label{ssec=gm}

Graph manifolds were first considered by F.~Waldhausen \cite{Waldhausen1967a,Waldhausen1967} and were further developed by W.~Neumann \cite{Neumann1981}. The boundary manifold of a line arrangement is a graph manifold, see \cite{Jiang}. Detailed proofs can also be found in \cite{Westlund1997,Cohen2006}.

A \emph{graph structure} on a closed oriented $3$-manifold $M$ is a set $`Q$ of pairwise disjoint \emph{joining} tori such that $M \smallsetminus `Q$ is a disjoint union of Seifert manifolds. A \emph{graph manifold} is a closed oriented $3$-manifold who admits a graph structure.
\begin{theorem}[\cite{Waldhausen1967}]
	\label{thm=mgs}
	Any graph manifold admits a unique graph structure with a minimal number of tori, except for those listed in \textnormal{\cite[Satz 8.1]{Waldhausen1967}}.
\end{theorem}
Exceptional classes of graph manifolds arise as boundary manifolds of only two specific families of complex line arrangements, listed in \cite{Jiang}. We leave these two families outside the scope of our study. From now on we only consider non-exceptional line arrangements.

One can naturally associate to a graph structure $`Q$ of $M$ a graph $`G$ as follows: each vertex $v$ is associated to a Seifert manifold $S_v$ and is decorated by its Euler number and singular fibres. Each edge $e_{v,w}$ corresponds to a gluing torus~$T_{v,w} `: `Q$ with $T_{v,w} = S_v \cap S_w$ and is decorated by the gluing map. The graph manifold $M$ can be reconstructed from the graph by gluing together the Seifert pieces along the gluing tori. If $M$ is a boundary manifold, then the graph has no dead-ends (i.e. every vertex has at least two neighbours), and the Seifert manifolds $S_v$ are actually circle bundles of the form $S(W_v,\check{s}_v,`e_v)$ as described in \cref{ex=cb-stand}. In particular they have no singular fibres. For every edge $e_{v,w}$, the Seifert pieces $S_v$ and $S_w$ are glued along the torus $T_{v,w}$ with a gluing map sending the horizontal curve $\check{`m}^w_v \subset S_v$ to a fibre in $S_w$ and the horizontal curve $\check{`m}^v_w \subset S_w$ to a fibre in $S_v$. This gluing map is the same for every edge, therefore the edges of the graph do not need to be decorated. By \cref{thm=mgs}, $M$ admits a unique such \emph{minimal graph}, and we only consider this one from now on.

\begin{theorem}[\cite{Jiang}]
	\label{thm=bm-struct}
	The boundary manifold $B_{\A}$ of a non-exceptional line arrangement $\A$ is a graph manifold whose minimal graph coincides with the incidence graph $\RG{C_{\A}}$ decorated with the following Euler numbers:
	\begin{enumerate}[(i), font=\normalfont]
		\item $`e_{L} = 1-b(L)$ for every line $L `: \mathcal{L}$.
		\item $`e_P = -1$ for every singular point $P `: \mathcal{Q}^{>2}$.\qedhere
	\end{enumerate}
	where $b(L)$ is the number of singular points of $\mathcal{Q}^{>2}$ meeting $L$.
\end{theorem}

Up to isotopy, homeomorphisms of a graph manifold always preserve the minimal graph structure.
\begin{theorem}[\cite{Waldhausen1967a}]
	\label{thm=wald}
	Let $M$ and $N$ be two graph manifolds with respective minimal graph structures ${`Q}_{M}$ and ${`Q}_{N}$. Let $`F : M \rightarrow N$ be a homeomorphism. Then $`F$ is isotopic to a homeomorphism $`F' : M \rightarrow N$ such that $`F'\left({`Q}_{M}\right) = `F'\left({`Q}_{N}\right)$.
\end{theorem}
According to \cref{thm=wald}, any homeomorphism of a graph manifold $M$ acts on the graph structure $`G$  (eventually after isotopy). Therefore, there is a surjective morphism
\begin{equation}
	\label{graph-aut}
	\begin{tikzcd}[cramped]
		G: \mathrm{Homeo}(M) \rar[two heads] & \mathrm{Aut}(`G)
	\end{tikzcd}
\end{equation}
The kernel of $G$ is called the group of \emph{graphed homeomorphisms} of the graph manifold $M$ and is denoted $\mathrm{Homeo}_{`G}(M)$. In particular a graphed homeomorphism $`J$ restricts to an homeomorphism $`J_v `: \Hom{S_v}$ for every bundle piece $S_v$ of $M$, and $`J_v$ fixes the boundary of $S_v$ component-wise. Moreover, if every vertex of the incidence graph $`G$ has more than $3$ neighbours (which is the case for boundary manifolds of non-exceptional line arrangements), then \cref{thm=cb-fibre} applies and $`J_v$ is fibrewise up to isotopy.
\begin{definition}
	\label{def=pos-homeo}
	A graphed homeomorphism $\mathrm{Homeo}_{`G}(M)$ is \emph{positive} if it respects the global orientation of~$M$. It is \emph{strongly positive} if its restriction on every piece $S_v$ is fibre-positive.
\end{definition}
We denote by $\HGpp{M}$ the group of isotopy classes of graphed strongly positive homeomorphisms of $M$.

\begin{proposition}
	\label{thm=ord-zar-pair-bm}
	Let $`F : (\mathbb{CP}^2,\A) "->" (\mathbb{CP}^2,\A')$ be a positive ordered equivalence of line arrangements. Then $`F$ induces an homeomorphism $B_{\A} \xlongrightarrow{\sim} B_{\A'}$ which is an element of $\HGpp{B_{\A}}$, i.e. graphed and strongly-positive.
\end{proposition}
\begin{proof}
	We write $B \vcentcolon= B_{\A} `~ B_{\A'}$. By \eqref{graph-aut}, $`F$~induces a transformation $G(`F)$ on the graph~$`G$, which coincides with the transformation of \eqref{graph-aut-init}. In particular since $`F$ is ordered then its restriction to the graph manifold $B$ is graphed and positive. The gluing map between the bundle pieces of $B$ respects their orientation. Therefore on every $S_v$ the restriction $`F_v$ of $`F$ is also positive. However, $`F_v$ is not necessarily fibre-positive. If $`F_v$ changes the orientation of the fibres of some $S_v$, then $`F_v$ must also change the orientation of the basis $`S_{W_v}$ since it preserves the global orientation of $S_v$. These changes propagate through the gluings between the pieces, so much that if $`F$ is fibre-positive on one $S_v$ then it must be on all of them, and vice versa. Let~$`n$ be the graphed positive homeomorphism of $B$ that reverse the orientation of the fibres (and therefore of the bases) of all~$S_v$'s. Then if $`F$ is not strongly-positive, $`n `o `F$ is.
\end{proof}

\section{Flat stars}
\label{sec=star}

A \emph{flat ordered star} is an ordered set of non-intersecting arcs drawn on a disc with boundary and starting at a common point. They are used in \cref{sec=oge} to define a way to embed all half-edges of the graph $`G$ with a common starting vertex $v `: V(`G)$ inside the corresponding circle bundle component~$S_v$ of the boundary manifold~$B_{\A}$. Depending on the situation, we consider flat stars ordered with a linear or circular~order, which are defined in \cref{ssec=star-def}. \cref{ssec=mcg,ssec=model} establish technical results which are then used in \cref{ssec=mcg-star} to prove the main goals of this section, namely \cref{prop=mcg-lstar,prop=mcg-cstar} that state that the pure mapping class group of surfaces acts transitively on flat stars.

\subsection{Linear and circular stars}
\label{ssec=star-def}

\begin{definition}
	\label{def=star}
	Let $M$ be a compact $3$-dimensional manifold and let $W$ be a subset of components of $\partial M$. A \emph{star} on $M$ is a collection of properly embedded simple arcs $`a = {\left(`a^w\right)}_{w `: W}$ in $M$ such that:
	\begin{enumerate}[(i)]
		\item for every $w `: W$, $`a^w(]0,1[) \subset \ring{M}$, and $`a^w(1) `: w$.
		\item there is a common point $b `: M$ such that for every $w `: W$, $`a^w(0) = b$.
		\item for every $w \neq w' `: W$, $`a^w \cap `a^{w'} = \{b\}$.
	\end{enumerate}
	The point $b$ is the \emph{centre} of the star $`a$ and the arcs $`a^w$ are the \emph{branches}.
\end{definition}
Reusing notations from \cref{ssec=seif}, let $D_W \subset `S_W$ be a disc (resp. sphere) with $m$ holes centred on the points of $W$. For each $w `: W$, the corresponding boundary component is denoted by $\partial^w D_W$ (resp. $\partial^w `S_W$). The last boundary component of $D_W$ is denoted by $\partial^{\infty} D_W$.
\begin{definition}
	\label{def=star-lin}
	Let $`q `: \Lin{W}$ be a linear order on $W$. A \emph{linear $`q$-star} drawn on $D_W$ is a star $\bar{`a} = {\left(\bar{`a}^w\right)}_{w `: W}$ in $D_W$ with centre $b$ such that:
	\begin{enumerate}[(i),start=4]
		\item $b `: \partial^{\infty} D_W$, and $\bar{`a}^w(1) `: \partial^w D_W$ for every $w `: W$.
		\item there exists a neighbourhood $U$ of $b$ and a positive local chart $`f : D_W \rightarrow \mathbb{R}^2$ that sends the pair ${(D_W \cap U, \bar{`a} \cap U)}$ to the model shown in~\Cref{fig=star-ord-lin} ordered by $`q$.
	\end{enumerate}
	The set of all isotopy classes of $`q$-linear stars on $D_W$ is denoted by $\OLS{`D_W}{`q}$.
\end{definition}
\begin{definition}
	\label{def=star-circ}
	Let $`w `: \Circ{W}$ be a circular order on $W$. A \emph{circular $`w$-star} drawn on $`S_W$ is a star $\tilde{`a} = {\left(\tilde{`a}^w\right)}_{w `: W}$ in $`S_W$ with centre $b$ such that:
	\begin{enumerate}[(i'),start=4]
		\item $b `: \ring{`S}_m$, and $\tilde{`a}^w(1) `: \partial^w `S_W$ for every $w `: W$.
		\item there exists a neighbourhood $U$ of $b$ and a positive local chart $`f : `S_W \rightarrow \mathbb{R}^2$ that sends the pair ${(`S_W \cap U, \tilde{`a} \cap U)}$ to the model shown in~\Cref{fig=star-ord-circ} ordered by $`w$.
	\end{enumerate}
	The set of all isotopy classes of $`w$-circular stars on $`S_W$ is denoted by $\OCS{`S_W}{`w}$.
\end{definition}
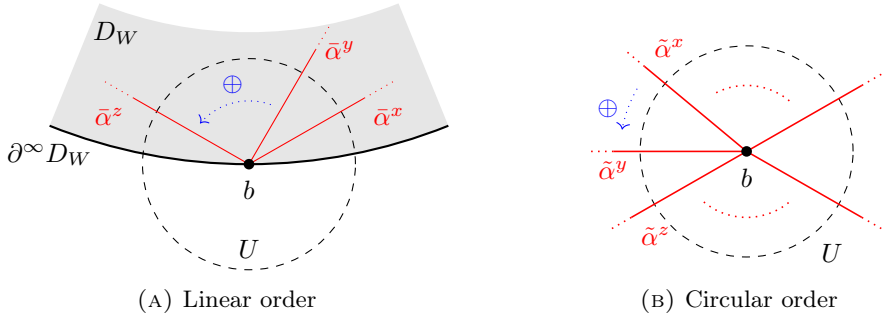
\begin{figure}[bht]
	\centering
	\begin{subfigure}{.45\linewidth}
		\centering
		\input{star_linear_order.pgf}
		\caption{Linear order}
		\label{fig=star-ord-lin}
	\end{subfigure}
	\begin{subfigure}{.45\linewidth}
		\centering
		\input{star_circular_order.pgf}
		\caption{Circular order}
		\label{fig=star-ord-circ}
	\end{subfigure}
	\caption{Star orderings}
	\label{fig=star-ord}
\end{figure}
\begin{figure}[bht]
	\centering
	\begin{subfigure}{.45\linewidth}
		\centering
		\input{star_linear.pgf}
		\caption{Linear flat star}
		\label{fig=star-ex-lin}
	\end{subfigure}
	\begin{subfigure}{.45\linewidth}
		\centering
		\input{star_circular.pgf}
		\caption{Circular flat star}
		\label{fig=star-ex-circ}
	\end{subfigure}
	\caption{Examples of flat stars}
	\label{fig=star-ex}
\end{figure}
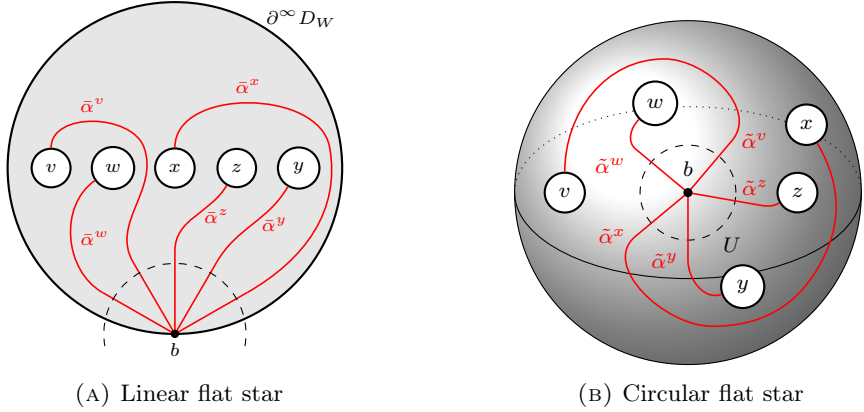
Examples of a linear star and of a circular star are shown on \Cref{fig=star-ex-lin,fig=star-ex-circ} respectively, with $W = \{x,y,z,v,w\}$ in this linear (and circular) order. We use the convention of denoting unordered stars as $`a$, linear stars as $\bar{`a}$ and circular stars as~$\tilde{`a}$.

We also make use of the set $\CS{`S_W}$ (resp. $\LS{`D_W}$) of all isotopy classes of circular stars on $`S_W$ (resp. linear stars on $D_W$).

\subsection{Mapping class groups}
\label{ssec=mcg}

We recall some results about the mapping class groups of surfaces, see \cite{Birman1975,Farb2011} for details.

Let $`S$ be a compact planar oriented surface. The mapping class group $\MCG{`S}$ is the group of isotopy classes of homeomorphisms of $`S$. The pure mapping class group $\PMCG{`S}$ is the subgroup of homeomorphisms that restrict to the identity on $\partial `S$.

Let $`S_W'$ be the surface obtained from $D_W$ by filling in the boundary component $\partial^{\infty} D_W$ with a disk punctured with one point $b_{\infty}$, and let $C : \PMCG{D_W} "->" \PMCG{`S_W'}$ be the map induced by the inclusion $f : D_W "`->" `S_W'$.
\begin{lemma}[Capping]
	\label{lem=cap}
	There is a short exact sequence:
	\[\begin{tikzcd}[cramped]
		1 \rar & \left< `D^2 \right> \rar & \PMCG{D_W} \rar["C"] & \PMCG{`S_W'} \rar & 1
	\end{tikzcd}\]
	where $`D^2$ is the full Dehn twist parallel to $\partial^{\infty} D_W$.
\end{lemma}
Let $F : \PMCG{`S_W'} "->" \PMCG{`S_W}$ be the map induced by the inclusion $g : `S_W' "`->" `S_W$, where we fill in the puncture $b_{\infty}$. There is a natural \enquote{pushing} map $P : `p_1(`S_W,b_{\infty}) "->" \PMCG{`S_W'}$ defined as such: for any simple closed curve $`d$ in $`S_W$ based on $b_{\infty}$, let $`d_+$ and $`d_-$ be two curves parallel to $`d$, which enclose an annulus centred on $`d$. Then $P(`d)$ is defined as the composition $T_{`d_+} `o T_{`d_-}^{-1}$ of the Dehn twists along $`d_+$ and $`d_-$ (see \cite[Section 4.2.2]{Farb2011} and \Cref{fig=mcg-push-init}).
\begin{lemma}[Birman exact sequence]
	\label{lem=birman}
	There is a short exact sequence:
	\[\begin{tikzcd}[cramped]
		1 \rar & `p_1(`S_W,b_{\infty}) \rar["P"] & \PMCG{`S_W'} \rar["F"] & \PMCG{`S_W} \rar & 1
	\end{tikzcd}\qedhere\]
\end{lemma}
\cref{lem=cap,lem=birman} also stand for non-pure mapping class groups.

Let $`D_W$ be a disc with $m$ \emph{punctures} placed at the points of the set $W$. Suppose that~$m \geq 2$. The braid group $\mathbb{B}_m$ on $m$ strands is generated by the braids $`s_{j,l}$ that performs a half twist on the strands $j$ and $l$. The pure braid group $\mathbb{P}_m$ is the subgroup generated by the braids $a_{j,l} \vcentcolon= `s_{j,l}^2$ that performs a \emph{full} twist on the strands $j$ and $l$. Any element $`q `: \Lin{W}$ induces isomorphisms:
\[\begin{tikzcd}[cramped]
	\mathbb{P}_m \rar["\sim"] & \PMCG{`D_W} & \mathbb{B}_m \rar["\sim"] & \MCG{`D_W}
\end{tikzcd}\]
The disc $`D_W$ can be obtained by capping every boundary components of $D_W$ except $\partial^{\infty} D_W$. Applying \cref{lem=cap} for each capping eventually gives the following result:
\begin{proposition}
	\label{thm=mcg-disc}
	Suppose that $m \geq 2$. Any element $`q `: \Lin{W}$ induces group isomorphisms
	\[\begin{tikzcd}[cramped]
		`r_{`q} : \mathbb{P}_m `* \mathbb{Z}^m \rar["\sim"] & \PMCG{D_W} & `r_{`q} : \mathbb{B}_m `* \mathbb{Z}^m \rar["\sim"] & \MCG{D_W}
	\end{tikzcd}\qedhere\]
\end{proposition}
The generator $a_{`q(v),`q(w)} `: \mathbb{P}_m$ is sent by $`r_{` q}$ to a full Dehn twist along a curve $`d_{v,w}$ going around $\partial^v D_W$ and $\partial^w D_W$, as shown on \textnormal{\Cref{fig=mcg-init}}. Similarly, $`s_{`q(v),`q(w)} `: \mathbb{B}_m$ is sent to a half Dehn twist along the same curve. The generator $d_{`q(v)} = (0, \dots, 1, \dots,0) `: \mathbb{Z}^m$ is sent to the full Dehn twist around a curve $`d_v$ parallel to $\partial^v D_W$, as shown on \textnormal{\Cref{fig=mcg-dt}}. Note that the subgroup $\mathbb{Z}^m$ is isomorphic to the kernel of the capping map $C_W : \PMCG{D_W} "->" \PMCG{`D_W}$.

To simplify notations, we often write ${(`b,d)}_{`q} \vcentcolon= `r_{`q}(`b,d)$. However, when using this notation the \emph{left} action of the mapping class group becomes a \emph{right} action of the braid group.

\subsection{Model disc}
\label{ssec=model}

Let $`D_m$ be the unit disc in the complex plane $\mathbb{C}$ with $m \geq 2$ punctures $x^1,\dots, x^m$ aligned from right to left on the real axis, and let $D_m$ be the disc obtained by removing a small open disc $D^j$ around each puncture $x^j$ in $`D_m$. We define linear stars on $D_m$ and on $`D_m$ similarly to \cref{def=star-lin}, replacing the boundary components with the punctures in the latter case. The base point is always $-i$ and the linear order is always assumed to be the identity. The~objective of this section is to establish \cref{prop=mcg-model-fth} which states that the action of the pure mapping class group of $D_m$ on these flat stars is transitive. This result on $D_m$ will be used in \cref{ssec=mcg-star} to generalise it to linear and circular stars on $D_W$ and $`S_W$ respectively.

The proof of \cref{prop=mcg-model-fth} uses a series of reductions to smaller structures: first we establish a similar result in the disc $`D_m$ in \cref{prop=mcg-model}. The proof of this second result consists in turn in building a connection between linear stars on $`D_m$ and a construction due to \cite{Fenn1999} called \emph{curve diagrams}. We prove with \cref{prop=curve-star} that curve diagrams are in bijection with linear stars on $`D_m$, and then we use \cref{thm=curve-braid} that establishes that $\PMCG{`D_m}$ acts transitively on curve diagrams, and thus on linear stars on $`D_m$. We then extend the result on $`D_m$ to $D_m$. Because of this entanglement of proofs, we present these results and their proofs by going from the smallest structure (curve diagrams) to the largest (linear stars on~$D_m$).

A \emph{curve diagram} is a curve embedded in the interior of $`D_m$, starting from $1$ and ending at $-1$, and going through all $x^j$'s in order. Let $\mathcal{C}_m$ be the set of isotopy classes of curve diagrams in $`D_m$. We denote by $`g_0$ the \emph{standard curve diagram} shown on \Cref{fig=star-stand}.

\begin{theorem}[\cite{Fenn1999}]
	\label{thm=curve-braid}
	The natural action of $\PMCG{`D_m}$ on $\mathcal{C}_m$ is free and transitive. In~particular there is a bijection
	\[\begin{tikzcd}[row sep=0pt,/tikz/column 1/.append style={anchor=base east},/tikz/column 3/.append style={anchor=base west}]
		X : & [-2.5em] \mathcal{C}_m \rar["\sim"] & \PMCG{`D_m}                           \\
		    & `g \rar[maps to]                    & `f \text{ such that } `f `. `g_0 = `g
	\end{tikzcd}\qedhere\]
\end{theorem}
Denote by $\LS{`D_m}$ the set of isotopy classes of linear stars on $`D_m$ (remember that they are all ordered by the identity).
\begin{proposition}
	\label{prop=mcg-model-def}
	The natural action of $\PMCG{`D_m}$ on $`D_m$ induces an action on $\LS{`D_m}$ given by
	\[\begin{tikzcd}[row sep=2pt]
		\PMCG{`D_m} `* \LS{`D_m} \rar & \LS{`D_m}\\
		\left(`f, {(\bar{`a}^j)}_{j}\right) \rar[maps to] & {(`f `. \bar{`a}^j)}_{j}
	\end{tikzcd} \qedhere\]
\end{proposition}
\begin{proof}
	Remember from \cref{ssec=mcg} that there is a group isomorphism $\mathbb{P}_m "->" \PMCG{`D_m}$.
	Take~$\bar{`a} `: \LS{D_m}$ and $`b `: \mathbb{P}_m$. The associated isomorphism is denoted by $`b_{\Id}$. It is obvious that $`b_{\Id}$ preserves the properties of \cref{def=star}, so $\bar{`a} `. `b_{\Id}$ is still a star. The diffeomorphism $`b_{\Id}$ is a product of Dehn twists along curves $`d_{j,l}$ circling $x^j$ and $x^l$. Up to isotopy, one can always suppose that all $`d_{j,l}$ do not intersect with the neighbourhood $U$ where the ordering of $\bar{`a}$ is defined. Therefore $\bar{`a} `. `b_{\Id}$ still respects the \cref{def=star-lin} of a linear star on $`D_m$.
\end{proof}

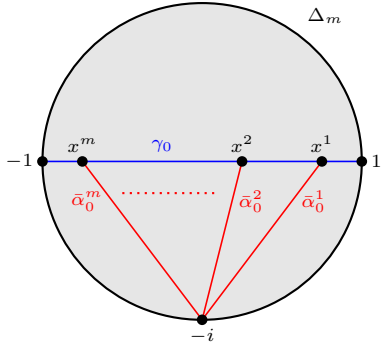
\begin{figure}[htb]
	\centering
	\input{star_linear_standard.pgf}
	\caption{Standard linear star $\bar{`a}_0$ and curve diagram $`g_0$ on $`D_m$}
	\label{fig=star-stand}
\end{figure}
\begin{proposition}
	\label{prop=curve-star}
	There exists a bijection $\LS{`D_m} `~ \mathcal{C}_m$ that respects the actions of~$\PMCG{`D_m}$ on each set.
\end{proposition}
\begin{proof}
	We build a map
	\[S : \mathcal{C}_m \longrightarrow \LS{`D_m}\]
	as follows: Let $`g `: \mathcal{C}_m$ be any curve diagram. By \cref{thm=curve-braid}, there exists a diffeomorphism $X(`g) `: \PMCG{`D_m}$ such that $X(`g) `. `g_0 = `g$. Let $\bar{`a}_0$ be the \emph{standard linear star} shown on \Cref{fig=star-stand}. We~then define
	\begin{equation}
		\label{def=curve-star}
		S(`g) \vcentcolon= X(`g) `. \bar{`a}_0 `: \LS{`D_m}
	\end{equation}
	Conversely, we build another map
	\[V : \LS{`D_m} \longrightarrow \mathcal{C}_m\]
	as follows: let $\bar{`a} `: \LS{`D_m}$ be a linear star and consider a regular neighbourhood $N$ of $\bar{`a}$ inside the complex plane, see \Cref{fig=star-ngbh-emb} for an example. Define $N_{\bar{`a}} \vcentcolon= N \cap `D_m \smallsetminus \bar{`a}$ as the \emph{stripe} of~$\bar{`a}$. It~can be obtained by cutting $N \cap `D_m$ along each branch of $`a$. The boundary $\partial N_{\bar{`a}}$ can be separated in two parts: one corresponds to $\partial N \cap `D_m$, and the other corresponds to the cutting of the star $\bar{`a}$ with the $m$ points $x^j$ marked in order. Each time we cut along a branch of $\bar{`a}$ the base point $-i$ is \enquote{split}. There are thus $m+1$ \enquote{copies} of $-i$ on $\partial N_{\bar{a}}$ intertwined with the $x^j$'s, as shown on \Cref{fig=star-ngbh-cut}. A \emph{wedge collection} $(V^j)_j$ is a set of curves embedded in $N_{\bar{`a}}$ such that each $V^j$ goes from $\partial N \cap `D_m$ to $x^j$ and back to $\partial N \cap `D_m$, and such that $V^j \cap V^k = \varnothing$ if~$j \neq k$. Then~$V(\bar{`a})$ is defined as the blue curve shown on \Cref{fig=star-ngbh-cut}, embedded in $`D_m$.
	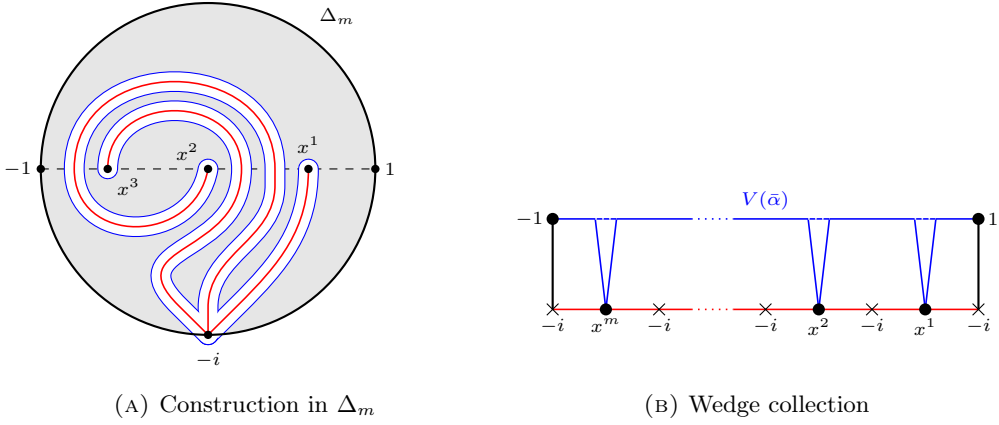
\begin{figure}[t]
		\centering
		\begin{subfigure}{.45\linewidth}
			\input{star_linear_neigh.pgf}
			\caption{Construction in $`D_m$}
			\label{fig=star-ngbh-emb}
		\end{subfigure}
		\begin{subfigure}{.45\linewidth}
			\input{star_neighbourhood.pgf}
			\caption{Wedge collection}
			\label{fig=star-ngbh-cut}
		\end{subfigure}
		\caption{Stripe $N_{\bar{`a}}$ of a star $\bar{`a}$}
		\label{fig=star-ngbh}
	\end{figure}

	Let us prove that the map $V$ is well-defined.

	Let $\bar{`a}, \bar{`a}'$ be two linear stars in $`D_m$ such that there exists an isotopy $I_t$ of $`D_m$ sending $\bar{`a}$ to $\bar{`a}'$. Then $I_t$ also sends the pair $(N_{\bar{`a}},\partial N_{\bar{`a}})$ to $(N_{\bar{`a}'},\partial N_{\bar{`a}'})$. Since $I_t$ fixes $\partial `D_m$, it also preserves the separation of $\partial N_{\bar{`a}}$ shown on \Cref{fig=star-ngbh-cut}, and in particular $I_t$ fixes all the $x^j$'s. It is clear that one can always build an isotopy $J_t$ of $N_{\bar{`a}}$ sending any wedge collection to any other. Extend $J_t$ by the identity outside of $N_{\bar{`a}}$. Then the composition $J_t `o I_t$ is an isotopy of $`D_m$ sending $V(\bar{`a})$ to~$V(\bar{`a}')$.

	We now prove that the maps $S$ and $V$ are inverse to each other.

	Let $`g `: \mathcal{C}_m$ and let $\bar{`a} \vcentcolon= S(`g)$. We want to show that $V(\bar{`a}) = `g$. It is clear from the definition of $V$ that for the standard linear star we have $V(\bar{`a}_0) = `g_0$. Consider $N_{\bar{`a}_0}$ the stripe of~$\bar{`a}_0$. By~\eqref{def=curve-star} we have $\bar{`a} = X(`g) `. \bar{`a}_0$. The diffeomorphism $X(`g)$ preserves the regular neighbourhoods up to isotopy and therefore sends the pair $(N_{\bar{`a}_0},\partial N_{\bar{`a}_0})$ to $(N_{\bar{`a}}, \partial N_{\bar{`a}})$. In particular:
	\[V(\bar{`a}) = X(`g) `. V(\bar{`a}_0) = X(`g) `. `g_0 = `g\]
	Therefore $V `o S = \Id_{\mathcal{C}_m}$. Reciprocally, let $\bar{`a} `: \LS{`D_m}$ and let $`g \vcentcolon= V(\bar{`a})$. We want to show that~$S(`g) = \bar{`a}$. By~\eqref{def=curve-star}, we have $S(`g) = X(`g) `. \bar{`a}_0$. Let $\bar{`a}' = {X(`g)}^{-1} `. \bar{`a}$. We show that $\bar{`a}' = \bar{`a}_0$. The curve $`g$ separates $`D_m$ in two halves and by construction $\bar{`a}$ is contained entirely within one~half. The diffeomorphism ${X(`g)}^{-1}$ preserves that separation, which implies that $\bar{`a}'$ is contained entirely within the lower half $`D_m \cap \{\Re(z) \leq 0\}$, just like $\bar{`a}_0$. By \cref{def=star-lin}, the stars $\bar{`a}'$ and $\bar{`a}_0$ are isotopic inside a neighbourhood $U$ of their common base point $-i$. We are now reduced to examining the \enquote{rectangle} $R \vcentcolon= `D_m \cap \{\Re(z) \leq 0\} \smallsetminus U$. The boundary $\partial R$ can be divided in four~parts. The top part corresponds to $`g_0$ with the $x^j$'s marked. The bottom part corresponds to $\partial U \cap `D_m$ and has also $m$ points marked, namely the $y^j \vcentcolon= \partial U \cap \bar{`a}^{\prime j} = \partial U \cap \bar{`a}_0^j$. The points $(x^j)$ and $(y^j)$ are ordered oppositely on their respective parts of $\partial R$ with respect to the the orientation of~$`D_m$. Inside $R$, the star $\bar{`a}'$ (and $\bar{`a}_0$) is a set of simple non-intersecting arcs joining each $y^j$ to the corresponding $x^j$. It is clear that there is only one isotopy type of this construction. The~stars $\bar{`a}'$ and $\bar{`a}_0$ are therefore isotopic inside both $R$ and $U$, so $\bar{`a}' = \bar{`a}_0$. Therefore $S `o V = \Id_{\LS{`D_m}}$, which completes the proof.
\end{proof}

\begin{proposition}
	\label{prop=mcg-model}
	The action of \cref{prop=mcg-model-def} is free and transitive.
\end{proposition}
\begin{proof}
	This is immediate from \cref{thm=curve-braid} and the fact that the bijection of \cref{prop=curve-star} commutes with the natural action of $\PMCG{`D_m}$ on $`D_m$ by construction.
\end{proof}

We can now go back to $D_m$.

\begin{proposition}
	\label{prop=mcg-model-fth}
	There is a natural action of $\PMCG{D_m}$ on $\LS{D_m}$. This action is transitive.
\end{proposition}
\begin{proof}
	The disc $`D_m$ can be obtained from $D_m$ by \enquote{capping} each boundary component of $D_m$ (save for $\partial^{\infty} D_m$) with a disc $D^j$ with one puncture at $x^j$. The capping operation induces a map
	\[J : \LS{D_m} \longrightarrow \LS{`D_m}\]
	described as follows: consider $\bar{`a}$ in $D_m \subset `D_m$. Extend each branch $\bar{`a}^j$ inside $D^j$ by an arc joining the end point on $\partial D^j$ to the puncture $x^j$. There is only one isotopy type of such an arc inside~$D^j$. This extended star is~$J(\bar{`a})$. An isotopy between two stars in $D_m$ can be extended to each~$D^j$, giving isotopic extended stars in~$`D_j$. Therefore $J$ is well-defined.

	Define an equivalence relation on $\LS{`D_m}$ by
	\begin{equation*}
		`A\: \bar{`a}, {\bar{`a}'} `:  \LS{`D_m} : \quad \bar{`a} `~ {\bar{`a}'} \Longleftrightarrow J(\bar{`a}) = J({\bar{`a}'})
	\end{equation*}
	We first prove that the subgroup $`r_{\Id}(\mathbb{P}_m)$ acts freely and transitively on the set of equivalence classes of $\LS{`D_m}$.
	For every $`f `: \PMCG{D_m}$ and $\bar{`a} `: \LS{`D_m}$, we have:
	\begin{equation}
		\label{eq=star-cap}
		J(`f `. \bar{`a}) = C(`f) `. J(\bar{`a})
	\end{equation}
	where $C : \PMCG{D_m} "->" \PMCG{`D_m}$ is the capping map obtained from applying \cref{lem=cap} on all boundary components $\partial^j D_m$. The map $J$ is clearly surjective, and so is the map $C$ by \cref{lem=cap}. The result then follows from the freeness and transitivity of the action of $\PMCG{`D_m}$ on $\LS{`D_m}$ by \cref{prop=mcg-model}.

	We now prove that the orbits of the action of the subgroup~$`r_{\Id}(\mathbb{Z}^m)$ coincide with the equivalence classes of $\LS{`D_m}$.

	Let $\bar{`a}, \bar{`a}' `:  \LS{`D_m}$. Suppose that there exists $`f `: `r_{\Id}(\mathbb{Z}^m)$ such that $\bar{`a} = `f `. \bar{`a}'$. Then by \cref{lem=cap}, $`f `: \ker C$ so by \eqref{eq=star-cap}, $J(\bar{`a}) = J(\bar{`a}')$. Reciprocally, suppose that $\bar{`a}$ is equivalent to $\bar{`a}'$. There is an isotopy of $`D_m$ sending $J(\bar{`a})$ to $J(\bar{`a}')$. By definition of $J$, this isotopy restricts to an homeomorphism $`f$ of $D_m$ sending $\bar{`a}$ to $\bar{`a}'$. We then have:
	\[J(\bar{`a}')  = C(`f) `. J(\bar{`a}) = J(\bar{`a})\]
	But the action of $\PMCG{`D_m}$ on $\LS{`D_m}$ is free. Therefore, $`f `: \ker C = `r_{\Id}(\mathbb{Z}^m)$.
\end{proof}

\subsection{Action on flat stars}
\label{ssec=mcg-star}
In this section we prove that the action of the pure mapping class group of surfaces on the sets of linear and circular stars is transitive in both cases. The case of linear stars on $D_W$ follows directly from the results of \cref{ssec=model}. The case of circular stars on $`S_W$ can eventually be reduced to the previous one thanks to \cref{lem=mcg-cstar-lin} which explains how to particularise a circular star into a linear one.

\begin{proposition}
	\label{prop=mcg-lstar}
	Let $`q$ be a linear order on $W$. The natural action of $\PMCG{D_W}$ on $D_W$ induces a well-defined action on $\OLS{D_W}{`q}$ given by
	\[\begin{tikzcd}[row sep=2pt]
		\PMCG{D_W} `* \OLS{D_W}{`q} \rar & \OLS{D_W}{`q}\\
		\left({(`b,d)}_{`q}, {(\bar{`a}^w)}_{w `: W}\right) \rar[maps to] & {(\bar{`a}^w `. {(`b,d)}_{`q})}_{w `: W}
	\end{tikzcd}\]
	This action is transitive.
\end{proposition}
\begin{proof}
	The choice of the ordering $`q$ on $W$ induces an homeomorphism ${M_{`q} : D_m "->" D_W}$, which in turn induces a bijection $\LS{D_m} `~ \OLS{D_W}{`q}$. The results of \cref{prop=mcg-model-fth} are thus passed on to $D_W$ and $\OLS{D_W}{`q}$.
\end{proof}
\begin{proposition}
	\label{prop=mcg-cstar}
	Let $`w$ be a circular order on $W$. The natural action of $\PMCG{`S_W}$ on $`S_W$ induces a well-defined action on $\OCS{`S_W}{`w}$. This action is transitive.
\end{proposition}

Recall the circularisation map $c : \Lin{W} "->>" \Circ{W}$ from \eqref{ord-circ}. It is clear from \cref{def=star-lin,def=star-lin} that the inclusion $i : D_W "`->" `S_W$ induces a surjective map
\[\begin{tikzcd}
	i : \LS{D_W} \rar[two heads] & \CS{`S_W}
\end{tikzcd}\]
In particular, for every $`q `: \Lin{W}$, the map restricts to:
\[i : \OLS{D_W}{`q} \longrightarrow \OCS{`S_W}{c(`q)}\]
\vspace{-\baselineskip}
\begin{lemma}
	\label{lem=mcg-cstar-lin}
	Let $`q,`q' `: \Lin{W}$ such that $c(`q)=c(`q')=`w$. Then for every linear star $\bar{`a} `: \OLS{D_W}{`q}$ there exists another linear star $\bar{`a}' `: \OLS{D_W}{`q'}$ such that $i(\bar{`a}) = i(\bar{`a}') `: \OCS{`S_W}{`w}$.
\end{lemma}
\begin{proof}
	It is enough to prove the result when $`q' = {(1\cdots m)} `o `q$. Fix $b_{\infty} `: `S_W \smallsetminus D_W$ and define $`S_W' = `S_W \smallsetminus \{b_{\infty}\}$. We write
	\[\begin{tikzcd}[cramped]
		D_W \rar[hook, "f"] & `S_W' \rar[hook, "g"] & `S_W
	\end{tikzcd}\]
	with $i = g `o f$. Let $\bar{`a} `: \OLS{D_W}{`q}$ drawn on $D_W$ and let $z `: W$ with $`q(z) = m$. Consider a loop $`d \subset `S_W$ based on $b_{\infty}$ that circles the entire star $i(\bar{`a})$ except the last branch (in linear order) $\bar{`a}^z$, which it crosses once. Such a loop can always be constructed as follows: starting from the centre $b$ of the star, take a path parallel to the left of $\bar{`a}^z$, then turn around $\partial^z D_W$ and return to $b$ with a path parallel to the right of $\bar{`a}^z$. Repeat the process for all other branches in reverse linear order. Then inside the neighbourhood $U$ of $b$, shortcut the loop around $\partial^z D_W$. Finally, deform the loop in $`S_W \smallsetminus D_W$ to make it go around $b_{\infty}$. Now recall from \cref{ssec=mcg} that $`d$ induces a \enquote{pushing} transformation $P(`d) `: \PMCG{`S_W'}$. Its action on $f(\bar{`a})$ is shown on \Cref{fig=mcg-push-init}. By construction, only the last strand $\bar{`a}^z$ is affected. Let $\hat{`a}' = P(`g) `. f(\bar{`a})$. By dragging the modified branch $\hat{`a}^{\prime z}$ all across $`S_W' \smallsetminus D_W$, the star $\hat{`a}'$ can be isotoped inside $`S_W'$ to the star shown on \Cref{fig=mcg-push-end}, where $\hat{`a}^{\prime z}$ is now the first branch in linear order inside $D_W$. Thus $\bar{`a}' \vcentcolon= f^{-1}(\hat{`a}') `: \OLS{`D_W}{`q'}$. Finally, by \cref{lem=birman}, we have
	\[i(\bar{`a}') = g(\hat{`a}') = g\left(P(`d) `. f(\bar{`a})\right) = F(P(`d)) `. i(\bar{`a}) = i(\bar{`a})\]
	since $F `o P = \Id$.
\end{proof}
\begin{figure}[tbh]
	\centering
	\begin{subfigure}{.45\linewidth}
		\centering
		\input{mcg_action_push.pgf}
		\caption{Double Dehn twist}
		\label{fig=mcg-push-init}
	\end{subfigure}
	\begin{subfigure}{.45\linewidth}
		\centering
		\input{mcg_action_push_end.pgf}
		\caption{New linear star in $D_W \subset `S_W'$}
		\label{fig=mcg-push-end}
	\end{subfigure}
	\caption{Pushing map action on $`S_W'$}
	\label{fig=mcg-push}
\end{figure}
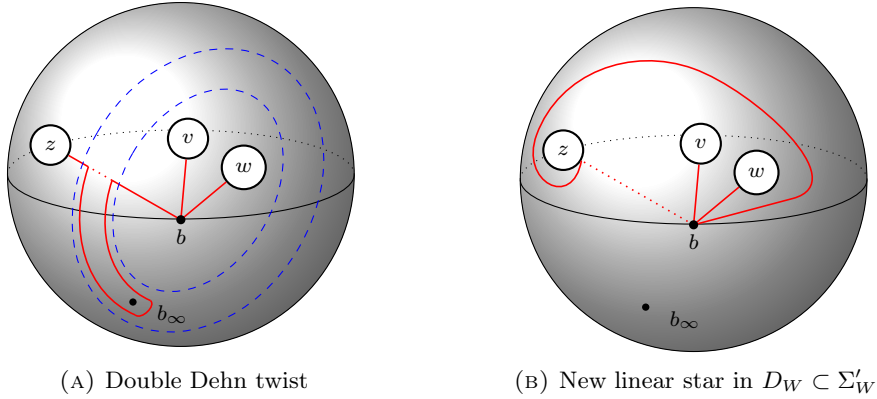

\begin{proof}[Proof of \cref{prop=mcg-cstar}]
	Let $`w `: \Circ{W}$ and $\tilde{`a} `: \OCS{`S_W}{`w}$. Let $D_{\infty}$ be a disc on $`S_W$ that does not meet $\tilde{`a}$ except on its centre $b$. Removing $\ring{D}_{\infty}$ from $`S_W$, we obtain $D_W$. Consider the inclusion $i : D_W "`->" `S_W$. There exists a linear star $\bar{`a} `: \LS{D_W}$ such that $i(\bar{`a}) = \tilde{`a}$. By \cref{lem=cap,lem=birman}, the map $i$ induces a surjective map $F`o C : \PMCG{D_W} "->>" \PMCG{`S_W}$. Let $\bar{`f} `: \PMCG{D_W}$ such that $`f = (F`o C)(\bar{`f})$. Then:
	\[`f `. \tilde{`a} = `f `. i(\bar{`a}) = (F `o C) (\bar{`f}) `. i(\bar{`a}) = i(\bar{`f} `. \bar{`a})\]
	By \cref{thm=mcg-disc}, there exists $(`b,d) `: \mathbb{P}_m `* \mathbb{Z}^m$ such that $\bar{`f} = `r_{`q}(`b,d)$. Then by \cref{prop=mcg-lstar}, $\bar{`f}(\bar{`a})$ is again a $`q$-linear star. Its image by $i$ is thus again a $`w$-circular star on $`S_W$.

	We now prove the transitivity of the action.
	Let $\tilde{`a}, {\tilde{`a}'} `: \OCS{`S_W}{`w}$ be two $`w$-circular stars on~$`S_W$. Up to to isotopy we can suppose that the centre $b$ of $\tilde{`a}$ does not lie on $\tilde{`a}'$ and vice versa. Let $D_{\infty}$ be a disc that does not meet neither $\tilde{`a}$ nor $\tilde{`a}'$. Up to isotopy one can always move the centres $b$ and $b'$ to $\partial D_{\infty}$. Similarly to the proof of \cref{prop=mcg-cstar}, by removing $\ring{D}_{\infty}$ from $`S_W$, we obtain $D_W$ and two linear stars $\bar{`a}, \bar{`a}'$ ordered by two linearisations $`q, `q' `: c^{-1}(`w)$ respectively. By \cref{lem=mcg-cstar-lin}, up to an isotopy of $`S_W$ we can actually suppose that $`q = `q'$. Then~by \cref{prop=mcg-lstar} there exists $\bar{`f} `: \PMCG{D_W}$ such that $\bar{`f} `. \bar{`a} = \bar{`a}'$. By \cref{lem=cap,lem=birman} we~have
	\[\tilde{`a}' = i(\bar{`a}') = i(\bar{`f} `. \bar{`a}) = (F `o C) (\bar{`f}) `. \tilde{`a}\]
	with  $(F `o C) (\bar{`f}) `: \PMCG{`S_W}$.
\end{proof}

\section{Graphed embeddings}
\label{sec=oge}

The first homology and fundamental group of the boundary manifold $B_{\A}$ have two types of generators: one corresponding to the meridians of the generators and the other to the cycles of the graph. To properly define the cycle generators, it is necessary to associate them with actual curves drawn on $B_{\A}$ itself. We therefore define a type of morphism
\[\begin{tikzcd}[cramped]
	`g : \RG{C_{\A}} \rar[hook] & B_{\A}
\end{tikzcd}\]
called \emph{graphed embeddings} that respects the graph structure of $B_{\A}$. Graphed embeddings are built by tying together \emph{ordered stars}, one for each circle bundle component of the graph manifold~$B_{\A}$.

Recall from \cref{thm=bm-struct} that $B_{\A}$ is a graph manifold over the graph $`G = \RG{C_{\A}}$.
\begin{definition}
	\label{def=ge}
	Let $`W = {(`w_v)}_{v `: V(`G)}$ be a graph ordering of $`G$. An \emph{ordered graphed embedding} with respect to $`W$ is a map $`g :  `G \hookrightarrow B_{\A}$ described as such: for every vertex $v `: V(`G)$, $`g \cap S_v = `a_v$ is a star drawn on the circle bundle $p_v : S_v "->" `S_{W_v}$ such that $p_v(`a_v)$ is a circular star on $`S_{W_v}$ ordered by $`w_v$.
\end{definition}
Note that the branches of the star $`a_v$ inside $S_v$ can travel along the fibres but two branches can never meet the same fibre, except over the centre $b_v$.

The set of isotopy classes of ordered graphed embeddings with respect to a graph ordering~$`W$ is denoted by $\OGE{`G}{`W}$.

\begin{proposition}
	\label{prop=act-h++-ord}
	There is a well-defined action of the group of \emph{graphed} strongly positive homeomorphisms $\HGpp{B_{\A}}$ on~$\OGE{`G}{`W}$.
\end{proposition}
\begin{proof}
	Consider $`g `: \OGE{`G}{`W}$. For every $v `: V(`G)$, let $`a_v = `g \cap S_v$ be the corresponding star. Then~$p_v(`a_v) `: \OCS{`S_{W_v}}{`w_v}$. Let ${`J `: \HGpp{M}}$. Since $`J$ is graphed then for every $v `: V(`G)$ it restricts to an homeomorphism $`J_v$ of $S_v$, which is fibre-positive and fixes the boundary of $S_v$ componentwise. In particular $`J_v$ induces an element $\widetilde{`J}_v `: \PMCG{`S_{W_v}}$ on the basis $`S_{W_v}$ of the fibre bundle $S_v$, such that the following diagram commutes:
	\begin{equation*}
		\label{eq=seif-bas-homeo}
		\begin{tikzcd}[cramped]
			S_v \rar["`J_v"] \dar["p_v",two heads] & S_v \dar["p_v",two heads] \\
			`S_{W_v} \rar["\widetilde{`J}_v"] & `S_{W_v}
		\end{tikzcd}
	\end{equation*}
	By~\cref{prop=mcg-cstar}, $\widetilde{`J}_v `. p_v(`a_v)$ is again a circular star on $`S_{W_v}$ for the same order $`w_v$. Since
	\[p_v (`J_v `. `a_v) = \widetilde{`J}_v `. p_v(`a_v)\]
	then $`J_v `. `a_v$ respects the conditions of \cref{def=ge}. Since this is true for all $v `: V(`G)$ then ${`J `. `a `: \OGE{`G}{`W}}$.
\end{proof}

The following fundamental \enquote{flattening} lemma allows to consider a graphed embedding as a union of linear stars rather than circular ones, which makes the action easier to handle.
Recall~from \cref{thm=mgs} that for every circle bundle $S_v$ there is a model $`c_v : S(m_v,\check{s}_v,`e_v) "->" S_v$ with $S(m_v,\check{s}_v,`e_v) = T_{m_{v}} \cup T_{\infty}$, and where $\check{s}_v$ is a fixed section of $T_{m_v}$.
\begin{lemma}
	\label{lem=star-flat}
	Let $`g = {(`a_v)}_{v `: V(`G)} `: \OGE{`G}{`W}$. Let $v `: V(`G)$ be a fixed vertex. There exists a linear star $\bar{`a}_v `: \OLS{D_{W_v}}{`q_v}$, where $`q_v `: \Lin{W_v}$ is a linearisation of $`w_v$, and such that $`a_v$ can be isotoped to $\check{s}_v(\bar{`a}_v)$ inside $B_{\A}$.
\end{lemma}
\begin{proof}
	Up to isotopy, the star $`a_v$ can be moved entirely within the trivially fibred torus $T_{m_v}$. If~no branch of $`a_v$ makes a full loop around a fibre of $T_{m_v}$ then there exist a section $s'$ of $T_{m_v}$ such that $s' \cap `a_v = \varnothing$. Cut $T_{m_v}$ along $s'$ to get a cylinder $S_{m_v} `* I$, then retract the star $`a_v$ on the section $\check{s}_v$ along the fibres to get the result. Now suppose that for some $w `: W_v$ the branch $`a^w_v$ makes at least one full loop around a fibre. Up to isotopy one can always move the loops at the extremity of the branch, i.e. on the boundary torus $\partial^w S_v$. Inside the graph manifold $B_{\A}$, the fibres of $\partial^w S_v$ are isotopic to horizontal loops inside $T_{m_w} \subset S_w$. One can thus \enquote{push} the loops to the other side of the gluing with an isotopy of $B_{\A}$, so that the branch $`a_v^w$ has no more fibre loops. Finally, push the centre $b$ of $`a_v$ to the boundary $\partial^{\infty} \check{s}_v$. This has the effect of particularising the circular order $`w_v$ to give a flat $`q_v$-linear star on $\check{s}_v$.
\end{proof}
\begin{lemma}
	\label{lem=mcg-cstar-lift}
	For any $v `: V(`G)$, the isotopy of $`S_{W_v}$ described in the proof of \cref{lem=mcg-cstar-lin} can be lifted to an isotopy of $S_v$.
\end{lemma}
\begin{proof}
	Let $z = `q_v^{-1}(m_v) `: W_v$ and let $`d$ be a loop on $`S_{W_v}$ that separates $\partial^z `S_W$ from the other boundary components, and let $`d_+$ and $`d_-$ be two curves parallel to $`d$ as shown on \Cref{fig=mcg-push}. Then~we have
	\[`S_{W_v} = E_{`d_-} \cup_{`d_-} D_{`d_-}\]
	with $E_{`d_+} `~ D_{W_v \smallsetminus \{z\}}$ and $D_{`d_-} `~ D_{\{z\}}$. The \enquote{pushing} transformation $P(`d)$ is the identity outside of $D_{`d_-}$. Using a trivialisation $T_{`d}$ of the fibre bundle over $D_{`d_-}$, one can thus lift $P(`d)$ to a partial isotopy $I$ of $S_v$, which can then be extended by the identity to the entirety of $S_v$. By construction, $p_v `o I = P(`d)$ inside $T_{`d}$.
\end{proof}
\begin{proposition}
	\label{prop=act-h++-tr}
	The action of \cref{prop=act-h++-ord} is transitive.
\end{proposition}
\begin{proof}
	Let $`g,`g' `: \OGE{`G}{`W}$ be two ordered graphed embeddings. Fix $v `: W$ and write $`a_v = `g \cap S_v$ and $`a_v' = `g' \cap S_v$. By \cref{lem=star-flat} up to isotopy one can suppose that $`a_v = \check{s}_v(\bar{`a}_v)$ (resp. $`a_v = \check{s}_v(\bar{`a}'_v)$) with $\bar{`a}_v `: \OLS{D_{W_v}}{`q_v}$ (resp. $\bar{`a}_v' `: \OLS{D_{W_v}}{`q_v'}$), where $`q_v, `q_v' `: c^{-1}({`w_v})$. Projecting on the basis $`S_{W_v}$ of $S_v$, we have $p(\check{s}_v) `~ D_{W_v}$. Write
	\[`S_{W_v} = p(\check{s}_v) \cup_{\partial^{\infty}} D_0\]
	where $D_0 `~ D^2$. The situation is similar to \Cref{fig=mcg-push} with $p(\check{s}_v)$ and $D_0$ as the upper and lower hemispheres respectively.
	By \cref{lem=mcg-cstar-lift}, we can suppose that $`q_v' = `q_v$. Then by \cref{prop=mcg-lstar} there exists $`j_v `: \PMCG{D_{W_v}}$ such that $`j_v `. \bar{`a}_v = \bar{`a}_v'$. Using the trivialisation of $S_v$ over $D_{W_v}$ and extending by the identity on the rest of $S_v$ we can lift $`j_v$ to a fibre-positive homeomorphism $`J_v$ of $S_v$. Then $`J_v `. `a_v = `a_v'$. Moreover $`J_v$ restricts to the identity on $\partial S_v$. One can then define an homeomorphism $`J$ of $B_{\A}$ as the reunion of all $`J_v$ acting on their respective $S_v$ for every~$v `: V(`G)$. By construction, $`J `: \HGpp{B_{\A}}$ and $`J `. `g = `g'$.
\end{proof}

\section{Homology inclusion}
\label{sec=hom-incl}

In this section we present the construction of our new ordered oriented line arrangements invariant. Consider the inclusion map
\[i_{\A} : B_{\A} \hooklongrightarrow E_{\A}\]
and the induced map on the first homology groups
\[i_{\A}^* : H_1(B_{\A},\mathbb{Z}) \longrightarrow H_1(E_{\A},\mathbb{Z})\]
In \cref{ssec=bm-hom} we give a presentation of these homology groups. The group $H_1(E_{\A},\mathbb{Z})$ is generated by the meridians of the lines of $\A$ (modulo the sum of all of them), corresponding to the fibres of the associated circle bundle components of $B_{\A}$. Meridians of the singular points can also be added as redundant generators. On the other hand $H_1(B_{\A},\mathbb{Z})$ is generated by the same meridians plus the cycles of the graph. \cref{thm=gm-hom} presents how the choice of an ordered graphed embedding allows to define the cycle generators. With that setting, \cref{ssec=hom-incl-def} gives the definition of the invariant. The restriction of $i_{\A}^*$ on the cycle generators depends on the choice of a graph ordering $`W$ and of an ordered graphed embedding $`g `: \OGE{`G}{`W}$. In order to remove that last dependency, we define the \emph{graph stabiliser} $\PI_{`G}$ as the quotient of all potential images of the cycles as sums of meridians by the differences between every two graphed embeddings. Then the \emph{homology inclusion} $\mathcal{J}_{`W}(\A)$ is given in \cref{def=hom-incl} as the class of the restriction of $i_{\A}^*$ inside the graph stabiliser. \cref{thm=la-istar-inv} establishes that $\mathcal{J}_{`W}$ is an invariant of ordered oriented line arrangements.

\subsection{Presentations of the homology groups}
\label{ssec=bm-hom}

For the boundary manifold the presentation is well known since the works of Westlund \cite{Westlund1997}, see also \cite{Mumford:61, Bartolo2019, kollar-nemethi}. For~the exterior the presentation is obtained by abelianising the fundamental group, which is itself computed using the Zariski-van Kampen method. We simply reformulate these results using our own notations and the new notion of \emph{ordered} graphed embeddings.

For this entire section, $`g `: \OGE{`G}{`W}$ is a fixed ordered graphed embedding. Two types of generators are involved: the \emph{meridians} which correspond to the fibres of each circle bundle component of $B_{\A}$, and the cycles of the graph $`G$.

We start with the exterior manifold $E_{\A}$. Let $C_0(`G)$ be the free abelian group generated by the set $V(`G)$ of vertices of the graph $`G$. Recall also that every vertex $v `: V(`G)$ is decorated with the Euler number $`e_v `: \mathbb{Z}$ of the corresponding $S_v$.
\begin{definition}
	\label{def=hom-mer}
	The \emph{meridian homology} $\MH{`G}$ of a graph $`G$ is the quotient of $C_0(`G)$ by the normally generated subgroup
	\[R(`G) \vcentcolon= \Big<\displaystyle`e_v `. v + \sum_{w `: W_v} w,\  v `: V(`G) \Big>\]
	with the exact sequence:
	\[\begin{tikzcd}[cramped]
		0 \rar & R(`G) \rar & C_0(`G) \rar["`h"] & \MH{`G} \rar & 0
	\end{tikzcd}\qedhere\]
\end{definition}
For every $x `: C_0(`G)$ we note $\bar{x}$ its image in $\MH{`G}$.
\begin{proposition}
	\label{prop=hom-ext}
	There is a natural isomorphism $\MH{`G(C_{\A})} \xlongrightarrow{\sim} H_1(E_{\A})$ that sends the class $\bar{v}$ to the homology class $f_v$ of the fibres of the circle bundle $S_v$ for every vertex $v `: V(`G)$.
\end{proposition}
\begin{proof}
	Is it known that:
	\[H_1(E_{\A}) = \left< f_L,\ L `: \mathcal{L} \ \middle| \ \sum_{L `: \mathcal{L}} f_L \right>\]
	where $f_L$ is the line meridian of $L$, that is the homological class of the fibres of $S_L$. Now using \cref{thm=bm-struct}, the relations of $\MH{`G(C_{\A})}$ become
	\begin{align*}
		`A\: L `: \mathcal{L}: \quad (b(L)-1) `. L &= \sum_{P `: W_L} P + \sum_{L' `: W_L} L' & `A\: P `: \mathcal{Q}^{>2}: \quad P &= \sum_{L `: W_P} L
	\end{align*}
	Replacing every $P$ in the first relation allows to remove them all as generators and yields
	\begin{gather*}
		`A\: L `: \mathcal{L}: \quad (b(L)-1) `. L = \sum_{P `: W_L} \sum_{L'' `: W_P} L'' + \sum_{L' `: W_L} L'
		\intertext{There are exactly $b(L)$ singular points in $W_L$ and each has $L$ once amongst their neighbours. We~can thus simplify the relations to~get:}
		`A\: L `: \mathcal{L}: \quad 0 = L + \sum_{P `: W_L} \sum_{\substack{L'' `: W_P\\L'' \neq L}} L'' + \sum_{L' `: W_L} L'
	\end{gather*}
	where $L''$ (resp. $L'$) are all the lines meeting $L$ at singular points with multiplicity $>2$ (resp.~$2$). But since $L$ meets every other line of $\mathcal{L}$ exactly once, we have in fact
	\[`A\: L `: \mathcal{L}: \quad 0 = L + \sum_{\substack{L' `: \mathcal{L}\\L' \neq L}} L' = \sum_{L' `: \mathcal{L}} L'\]
	Identifying the formal generator $L$ with $f_L$, this is exactly the defining relation of $H_1(E_{\A})$ repeated for every line.
\end{proof}
Note that \cref{prop=hom-ext} implies that $H_1(E_{\A})$ is combinatorially determined. We use the name \enquote{meridian homology} and notation $\MH{`G}$ to designate the corresponding formal combinatorial group. This distinction is necessary because $\MH{`G}$ also appears as an important subgroup of the first homology group of the boundary manifold $H_1(B_{\A})$, as we shall see now.

The fundamental group $`p_1(B_{\A})$ has two types of generators: the meridians, which must be connected to a common base point, and the cycles of the graph. A graphed embedding is required to properly draw both types of generators on the boundary manifold $B_{\A}$. One also needs additional combinatorial data to obtain a presentation of $`p_1(B_{\A})$.

Fix a vertex $r `: V(`G)$. Seeing the graph $`G$ as a topological space, one can define the fundamental group $`p_1(`G,r)$, which is freely generated by the cycles of the graph based in $r$. Note that there is no canonical set of generators. Let $\mathcal{T}$ be a \emph{spanning tree} on $`G$, that is a connected acyclic subgraph of $`G$ containing all vertices. For every pair of vertices $v,w `: V(`G)$ there always exist a unique path from $v$ to $w$ inside $\mathcal{T}$. Then for every edge $e_{v,w} `: `G \smallsetminus \mathcal{T}$, there is a unique cycle $c_{v,w}$ drawn on $`G$ going through $e_{v,w}$ and the root $r$. In other words, the choice of $\mathcal{T}$ gives a cycle basis. We then write~$`g_{v,w} \vcentcolon= `g(c_{v,w})$.

Let $\mathbb{F}_V$ be the group freely generated by the vertex set $V$. For every vertex $v `: V(`G)$, let $f_v$ be the fibre curve of the circle bundle $S_v$ over the point $`g(v)$, and let $c_v$ be the unique path from the root $r$ to $v$ in $\mathcal{T}$. Then define the \emph{meridian} of $v$ as the curve $`m_v \vcentcolon= `g(c_v) `. f_v `. {`g(c_v)}^{-1}$ on $B_{\A}$.

Finally, let also $`d$ be an orientation on the graph $`G$, that is a function assigning a sign $`d_{v,w} = `+ 1$ to each edge $e_{v,w}$, with $`d_{v,w} = - `d_{w,v}$.
\begin{theorem}
	\label{thm=gm-pi1-pres}
	Write $b \vcentcolon= `g(r)$. Then there is an exact sequence
	\[\begin{tikzcd}[cramped]
		0 \rar & \Pres{`G}{\mathcal{T}}{`g} \rar & \mathbb{F}_V * `p_1(`G,r) \rar["{`g}_{\mathcal{T}}"] & `p_1(B_{\A},b) \rar & 0
	\end{tikzcd}\]
	where $\Pres{`G}{\mathcal{T}}{`g}$ is the subgroup of relations normally generated by
	\begin{enumerate}[(i), font=\normalfont, labelsep=.5em, itemsep=1em]
		\item $`A\ v `: V(`G) : \quad {`m_v}^{`e_v} `. \displaystyle\prod_{k=1}^{m_v} {`m_{{`q_v^{-1}}(k)}}^{u(v,`q_v^{-1}(k))}$
		\item $`A\ e_{v,w} `: `G: \quad \left[`m_v, {`m_{w}}^{u(v,w)}\right]$
		\[\text{where } u(v,w) \vcentcolon=
		\begin{dcases*}
			{(`g_{v,w})}^{`d_{v,w}} & if $e_{v,w} `: `G \smallsetminus \mathcal{T}$.\\
			1 & if $e_{v,w} `: \mathcal{T}$.
		\end{dcases*} \qedhere\]
	\end{enumerate}
\end{theorem}
For $x,y `: `p_1(B_{\A},b)$ the notations are $[x,y] = x y x^{-1} y^{-1}$ and $x^y = y^{-1} x y$.
\begin{remark}
	This presentation can often be simplified, in particular when some Euler numbers $`e_v$ are equal to 0, 1 or $-1$.
\end{remark}
The first homology group $H_1(`G)$ is a free abelian group generated by the cycles of the graph $`G$. Note that again there is no canonical basis of cycle generators.

\begin{theorem}
	\label{thm=gm-hom}
	Let $`g `: \OGE{`G}{`W}$ be a graphed embedding. Then $`g$ induces a group isomorphism
	\[\begin{tikzcd}[cramped]
		`g_* : \MH{`G} `(+) H_1(`G) \rar["\sim"] & H_1(B_{\A})
	\end{tikzcd}\qedhere\]
\end{theorem}
\begin{proof}
	Abelianising the exact sequence of \cref{thm=gm-pi1-pres} yields a new exact sequence
	\[\begin{tikzcd}[cramped]
		0 \rar & \Pres{`G}{\mathcal{T}}{`g} \rar \dar["\Ab"] & \mathbb{F}_V * `p_1(`G,r) \rar["{`g}_{\mathcal{T}}"] \dar["\Ab"] & `p_1(M,b) \rar \dar["\Ab"] & 0\\
		0 \rar & R({`G}) \rar & C_0(`G) `(+) H_1(`G) \rar["\Ab(`g)"] & H_1(M,\mathbb{Z}) \rar & 0
	\end{tikzcd}\]
	which is exactly the defining exact sequence of $\MH{`G} `(+) H_1(`G)$. The morphism $`g_*$ is defined as the morphism induced by $\Ab(`g)$ on the quotient, which does not depend on the choice of the parameter $\mathcal{T}$.
\end{proof}
\begin{proposition}
	\label{prop=ge-diff-mer}
	For every pair of graphed embeddings $`g, `g' `: \OGE{`G}{`W}$,	we have
	\[{\left({`g_*}^{-1} `o `g_*'\right)}_{|\MH{`G}} = \Id_{\MH{`G}} \qedhere\]
\end{proposition}
\begin{proof}
	For every $v `: V(`G)$, $`g_*(\bar{v})$ is equal to the homological class of the meridian $`m_v$ in $H_1(B_{\A})$, which by construction is equal to the class of the fibre $f_v$ over $`g(v)$ in $H_1(S_v)$. Since all fibres of $S_v$ have the same homological class, the image $`g_*(\bar{v})$ does not depend on the choice of $`g$.
\end{proof}

\subsection{Definition of the invariant}
\label{ssec=hom-incl-def}

Let $\A$ be a line arrangement and write $`G \vcentcolon= \widetilde{`G}(C_{\A})$ its reduced incidence graph. Consider the inclusion map
\[i_{\A} : B_{\A} \hooklongrightarrow E_{\A}\]
and the induced map on the first homology groups
\[i_{\A}^* : H_1(B_{\A},\mathbb{Z}) \longrightarrow H_1(E_{\A},\mathbb{Z})\]
Let $`W$ be an ordering of the graph $`G$. By \cref{thm=gm-hom}, every graphed embedding $`g `: \OGE{`G}{`W}$ induces an isomorphism
\[\begin{tikzcd}[cramped]
	`g_* : \MH{`G} `(+) H_1(`G) \rar["\sim"] & H_1(B_{\A},\mathbb{Z})
\end{tikzcd}\]
By \cref{prop=hom-ext}, there is a natural group~isomorphism $H_1(E_{\A},\mathbb{Z}) `~ \MH{`G}$. Therefore, we have a map
\[i_{\A}^* `o `g_* :  \MH{`G} `(+) H_1(`G) \longrightarrow \MH{`G}\]
\vspace{-\baselineskip}
\begin{lemma}
	\label{lem=la-incl-V}
	For every graphed embedding $`g `: \OGE{`G}{`W}$, we have
	\[{\left(i_{\A}^* `o `g_*\right)}_{|\MH{`G}} = \Id_{\MH{`G}} \qedhere\]
\end{lemma}
\begin{proof}
	For every line $L$ of $\A$, $i_{\A}^* `o `g_*(\bar{L})$ is equal to the homology class of the fibre $f_L$ over $`g(L)$ inside $S_L$, which is exactly the generator of $H_1(E_{\A})$ corresponding to $\bar{L}$ by \cref{prop=hom-ext}.
\end{proof}
Now consider the restriction
\[{\left(i_{\A}^* `o `g_*\right)}_{|H_1(`G)} `: \hom(H_1(`G),\MH{`G})\]
which we simply write $i_{\A}^* `o `g_*$ for~short. This morphism obviously depends on the choice of the graphed embedding $`g$. The \emph{graph stabiliser} is a quotient designed to express the morphism $i_{\A}^*$ in a way that no longer depends on the choice of $`g `: \OGE{`G}{`W}$ but only on the choice of the graph ordering $`W$ itself.
\begin{definition}
	\label{def=gs}
	The \emph{graph stabiliser} $\PI_{`G}(`W)$ is defined as the quotient of $\hom(H_1(`G),\MH{`G})$ by the subgroup generated by all elements $`f `o {(`g_* - `g_*')}_{|H_1(`G)}$, where
	\begin{enumerate}[(i)]
		\item $`g,`g' `: \OGE{`G}{`W}$.
		\item \label{gs-cond} $`f `: \hom(H_1(B_{\A}),\MH{`G})$ is such that ${(`f `o `g_*)}_{|\MH{`G}} = \Id_{\MH{`G}}$ for all $`g `: \OGE{`G}{`W}$. \qedhere
	\end{enumerate}
\end{definition}
All objects involved in \cref{def=gs}, including $H_1(B_{\A})$ by \cref{thm=gm-hom}, depend only on the the combinatorics $`G$ of $\A$ and the choice of a graph ordering $`W$. However, this last dependency can actually be lifted.
\begin{theorem}
	\label{cor=gs-ordered-model}
	The group isomorphism type of the graph stabiliser $\PI_{`G}(`W)$ is the same for all graph orderings~$`W$ on~$`G$.
\end{theorem}
The proof of \cref{cor=gs-ordered-model} is obtained by establishing a presentation of the graph stabiliser, which is the object of \cref{sec=gs}. From now on we simply write $\PI_{`G}$.

We are now ready to define the homology inclusion invariant.
\begin{definition}
	\label{def=hom-incl}
	Let $\A$ be an ordered line arrangement and let $`G \vcentcolon= \RG{C_\A}$ be its incidence graph. Fix an ordering $`W$ of the graph. Then for any graphed embedding $`g `: \OGE{`G}{`W}$ the class
	\[\mathcal{J}_{`W}(\A) \vcentcolon= \left[ i_{\A}^* `o `g_* \right] `: \PI_{`G}\]
	is called the \emph{homology inclusion} of the ordered line arrangement.
\end{definition}
\cref{lem=la-incl-V}~ensures that $i_{\A}^*$ respects condition \ref{gs-cond} of \cref{def=gs}, so by construction of $\PI_{`G}$ the class $\left[ i_{\A}^* `o `g_* \right]$ does not depend on the choice of~$`g `: \OGE{`G}{`W}$. Therefore $\mathcal{J}_{`W}(\A)$ is well-defined. Our~main result states that this class element of the graph stabiliser $\PI_{`G}$ is a topological invariant of ordered oriented line~arrangements.
\begin{theorem}
	\label{thm=la-istar-inv}
	Let $\A, \A' \subset \mathbb{CP}^2$ be two non-exceptional line arrangements with the same combinatorics $C$. If there exists a positive ordered equivalence between $\A$ and $\A'$ then for any graph ordering $`W$ of the incidence graph $`G$ we have $\mathcal{J}_{`W}(\A) = \mathcal{J}_{`W}(\A')$ inside the graph stabiliser~$\PI_{`G}$.
\end{theorem}
\begin{proof}
	The topological equivalence between $\A$ and $\A'$ induces
	homeomorphisms $`J$ and $`F$ such that the following diagram commutes:
	\[\begin{tikzcd}
		B_{\A'} \rar["`J","\sim"'] \dar["i_{\A'}"',hook] & B_{\A} \dar["i_{\A}",hook] \\
		E_{\A'} \rar["`F"',"\sim"]                      & E_{\A}
	\end{tikzcd}\]
	Write $B = B_{\A'} `~ B_{\A}$. By \cref{thm=ord-zar-pair-bm}, we can always suppose that ${`J `: \HGpp{B}}$. Let $`g `: \OGE{`G}{`W}$ be an ordered graph embedding. By~\cref{prop=act-h++-ord}, the image $`J(`g)$ is again an element of~$\OGE{`G}{`W}$. The~map $`J$ induces a group automorphism $`J_* : H_1(B,\mathbb{Z}) \rightarrow H_1(B,\mathbb{Z})$. By~construction we have ${`J (`g)}_* = `J_* `o `g_*$ and ${(i_{\A} `o `J)}_* = i_{\A}^* `o `J_*$. Similarly, the map $`F$ induces $`F_* : H_1(E_{\A'},\mathbb{Z}) \rightarrow H_1(E_{\A},\mathbb{Z})$. By \cref{prop=hom-ext} these last two groups are both naturally isomorphic to $\MH{`G}$ generated by the line meridians. Since $`F$ is an ordered equivalence it induces the identity between the sets of lines of $\A$ and $\A'$, therefore $`F_* = \Id_{\MH{`G}}$. Then~inside $\hom(H_1(`G),\MH{`G})$ we~have
	\begin{equation*}
		i_{\A'}^* `o `g_* = `F_*^{-1} `o i_{\A}^* `o `J_* `o `g_* = i_{\A}^* `o {`J (`g)}_*
	\end{equation*}
	By \cref{def=gs} of the graph stabiliser $\PI_{`G}$, this implies that inside the quotient:
	\[0 = \left[i_{\A}^* `o \left({`J (`g)}_* - `g_*\right) \right] = \left[i_{\A'}^* `o `g_*\right] - \left[i_{\A}^* `o `g_*\right] = \mathcal{J}_{`W}(\A') - \mathcal{J}_{`W}(\A)\qedhere\]
\end{proof}
\begin{remark}
	\label{rem=graph-ord}
	By definition of the graph stabiliser, the class $\left[ i_{\A}^* `o `g_* \right] `: \PI_{`G}$ does not depend on the graphed embedding~$`g_* `: \OGE{`G}{`W}$. However,~it does depend on the graph ordering $`W$ which is preserved only by ordered equivalences. The~homology inclusion is therefore an \emph{ordered} line arrangement~invariant. Still,~some combinatorics have trivial automorphism groups. The~restriction of the homology inclusion to this subclass of line arrangements becomes an unordered topological~invariant.
\end{remark}
\begin{remark}
	\label{rem=la-orient}
	The graph stabiliser does not quotient the homology differences caused by the application of the complex conjugation inside $\mathbb{CP}^2$, since it is not a positive homeomorphism of the boundary manifold. A line arrangement $\A$ and its conjugate $\overline{\A}$ may thus have different homology inclusion values in general.
\end{remark}

\section{Presentation of the graph stabiliser}
\label{sec=gs}

The graph stabiliser $\PI_{`G}$ is the underlying group where the homology inclusion invariant is defined. In order to make this invariant effective, one needs to be able to not only compute the invariant itself, but also the graph stabiliser. The objective of this section is to establish \cref{thm=gs-pres} which gives an explicit finite presentation of the graph stabiliser. To obtain this result we analyse further the objects involved in \cref{def=gs} of the graph stabiliser in order to actually compute the homological difference between every two graphed embeddings. The crucial step to achieve this is \cref{prop=act-h++-tr} which allows to reduce this difference computation to a finite number of cases.

In addition, the presentation of the graph stabiliser is the main tool used in \cref{sec=lln} to establish the connection between the homology inclusion and the pre-existing loop-linking number invariant.

Since the graph stabiliser is a combinatorial object, in this section we fix a graph $`G$. For any line arrangement $\A$ with graph~$`G$, the boundary manifold $B_{\A}$ as a topological object only depends on the choice of $`G$, see \cref{thm=bm-struct}.

We reuse all notations introduced at the beginning of \cref{ssec=bm-hom}. Recall that $`G$ has a natural structure of a CW-complex generated by the vertices and the edges. The first homology group $H_1(`G)$ is the kernel of the boundary map $\partial_1 : C_1(`G) \rightarrow C_0(`G)$. There is an inclusion homomorphism
\[`z : H_1(`G) \longrightarrow C_1(`G)\]
which decomposes every cycle into the sum of its oriented edges.

Let us denote by $\vec{e}_{v,w}$ the \emph{half-edge} of $`G$ starting from $v$ and going towards $w$ (and reciprocally for $\vec{e}_{w,v}$). Let $\vec{C}_1(`G)$ be the free abelian group generated by all half-edges. Fix an orientation $`d$ of~$`G$. There is a natural map $`c_{`d} : C_1(`G) "->" \vec{C}_1(`G)$ defined by $`c_{`d}(e_{v,w}) \vcentcolon= `d_{v,w} \left(\vec{e}_{v,w} - \vec{e}_{w,v}\right)$. There is also a decomposition
\[\vec{C}_1(`G) = \bigoplus_{v `: V(`G)} \left<\vec{W}_v\right>\]
where $\vec{W}_v \vcentcolon= \left\{\vec{e}_{v,w} \mid w `: W_v\right\} `~ W_v$.

Finally, if $A,B$ are two free abelian groups, we denote by $A^\dagger$ the dual $\mathbb{Z}$-module of $A$. We~make the natural identification $\hom(A,B) `~ A^\dagger `(`*) B$, and for every homomorphism $`f : A "->" B$ we denote by $`f^\dagger : B^\dagger "->" A^\dagger$ its dual homomorphism. In addition, if we fix a basis $\mathcal{B}$ of the free module~$A$, then for every element $a `: A$ we denote by $a^\ddagger `: A^\dagger$ its dual element, which decomposes with the same coefficients in the dual basis of $\mathcal{B}$. In particular the free modules $C_0(`G), C_1(`G), \vec{C}_1(`G)$ (and the modules of morphisms between them) all have obvious canonical bases which we use to define their dual elements.

\subsection{Difference maps}
\label{ssec=diff-maps}

The difference maps are used to compute the homological difference between two ordered graphed~embeddings. This~difference lies exclusively on the cycle generators of $H_1(B_{\A})$ since by \cref{prop=ge-diff-mer} the homological values of two ordered graphed embeddings always coincide on the meridian~generators. We build the difference maps step by step, by decomposing both the graph manifold and the graphed embeddings into their elementary parts and then computing the homological difference at every level.

Two stars $`a,`a'$ in a manifold $M$ are called \emph{joined} if their centres coincide, and if for every $w `: W$ the endpoint of their corresponding branches $`a^w$ and ${`a'}^w$ also coincide.
\begin{definition}
	\label{def=diff-map-star}
	Let $\bar{`a}, \bar{`a}' `: \OLS{D_W}{`q}$ be two \emph{joined} linear stars on the disc $D_W$. The \emph{star difference map} is the map
	\[\DiffS{W}(\bar{`a},\bar{`a}') : W \longrightarrow H_1(D_W)\]
	that sends every $w `: W$ to the class of the closed curve
	\[\bar{`l}^w \vcentcolon= \bar{`a}^w `. {\left(\bar{`a}^{\prime w}\right)}^{-1}\qedhere\]
\end{definition}
Consider now two ordered graphed embeddings $`g,`g' `: \OGE{`G}{`W}$ and fix $v `: V(`G)$. Write~$`a_v = `g \cap S_v$ and $`a_v' = `g' \cap S_v$ the corresponding stars on $S_v$. By \cref{lem=star-flat,lem=mcg-cstar-lift}, we have $`a_v = \check{s}_v(\bar{`a}_v)$ and $`a_v' = \check{s}_v(\bar{`a}_v')$ with $\bar{`a}_v, \bar{`a}_v' `: \OLS{D_{W_v}}{`q_v}$.
\begin{lemma}
	\label{lem=diff-map-clos}
	There exists an isotopy of $B_{\A}$ that makes $`a_v$ and $`a_v'$ joined inside $S_v$.
\end{lemma}
\begin{proof}
	Fix $w `: W_v$. Since the gluing map between $S_v$ and $S_w$ permutes meridians and longitudes, $\check{s}_v \cap \check{s}_w \subset B_{\A}$ is reduced to a single point $t_{v,w} `: T_{v,w}$. Applying \cref{lem=star-flat} in $S_v$ gives an isotopy of $B_{\A}$ that projects the two stars $`a_v, `a_v'$ on $\check{s}_v$. Applying \cref{lem=star-flat} now in $S_w$ gives another isotopy of $B_{\A}$ that projects $`g \cap S_w$ and $`g' \cap S_w$ to $\check{s}_w$. This second isotopy acts on $T_{v,w}$ and therefore extends to $S_v$ where it moves $`a_v$ and $`a_v'$ inside $\check{s}_v$ to make them meet $t_{v,w}$. Note that any of these isotopies might introduce full loops around a fibre in $S_v$ or $S_w$. These~loops can always be \enquote{pushed} to the other side of the gluing, where they will become loops in $\check{s}_w$ or $\check{s}_v$. Repeat~the process for all $w `: W_v$. Separately, the centres of $`a_v$ and $`a_w$ both lie on~$\partial^{\infty} D_{W_v}$. Using an isotopy of $S_v$ in the neighbourhood of $\partial^{\infty} D_{W_v}$ we can always bring one centre to the~other.
\end{proof}
\cref{lem=diff-map-clos} also implies that $\bar{`a}_v$ and $\bar{`a}_v'$ are joined inside $\check{s}_v `~ D_{W_v}$.
\begin{definition}
	\label{def=disc-hom-map}
	Let $v `: V(`G)$ be a vertex of the graph $`G$. Let $\left<W_v\right>$ be the submodule of $C_0(`G)$ freely generated by the neighbour set $W_v$. The \emph{homological neighbour~map} is the group isomorphism
	\[h_v : \left<W_v\right> \longrightarrow H_1\left(D_{W_v}\right)\]
	that sends every neighbour vertex $w `: W_v$ to the class of the curve $\partial^{w} D_{W_v}$.
\end{definition}
\begin{definition}
	\label{def=ge-diff-edge}
	The \emph{half-edge difference map} is the morphism
	\[\DiffE : {\OGE{`G}{`W}}^2 \longrightarrow \hom\left(\vec{C}_1(`G),C_0(`G)\right)\]
	defined on every generating subset $\vec{W}_v$ of $\vec{C}_1(`G)$ by:
	\[\DiffE(`g,`g')_{|\vec{W}_v} \vcentcolon= h_{v}^{-1} `o \DiffS{W_v}(\bar{`a}_v,\bar{`a}_v')\qedhere\]
\end{definition}
\begin{proposition}
	\label{prop=ge-diff-edge-props}
	For every $`g, `g', `g'' `: \OGE{`G}{`W}$, the half-edge difference map $\DiffE$ verifies the following properties:
	\begin{enumerate}[(i),font=\normalfont]
		\item $\DiffE(`g,`g)$ is the trivial homomorphism.
		\item $\DiffE(`g,`g') = - \DiffE(`g',`g)$.
		\item $\DiffE(`g,`g'') = \DiffE(`g,`g') + \DiffE(`g',`g'')$. \qedhere
	\end{enumerate}
\end{proposition}
\begin{definition}
	\label{def=ge-diff}
	Let $`d$ be an orientation of $`G$. The \emph{cycle difference map} is the map
	\[\DiffH : {\OGE{`G}{`W}}^2 \longrightarrow \hom\left(H_1(`G),\MH{`G}\right)\]
	defined by
	\begin{equation*}
		\DiffH(`g,`g') \vcentcolon= ((`c_{`d} `o `z)^\dagger `(`*) `h) `o \DiffE(`g,`g')
	\end{equation*}
	where $`h : C_0(`G) \twoheadrightarrow \MH{`G}$ is the projection from \textnormal{\cref{def=hom-mer}}.
\end{definition}
\begin{remark}
	The cycle difference map $\DiffH$ verifies properties similar to \cref{prop=ge-diff-edge-props} which are induced by the properties of the half-edge difference map $\DiffE$.
\end{remark}

As announced, the cycle difference map gives a first reformulation of the definition of the graph stabiliser.
\begin{proposition}
	\label{prop=ge-diff-class}
	Let $`g, `g' `: \OGE{`G}{`W}$. Then
	\[`g_* `o \DiffH(`g,`g') = {(`g_* - `g_*')}_{|H_1(`G)} `: \hom(H_1(`G),H_1(B_{\A})) \qedhere\]
\end{proposition}
\begin{theorem}
	\label{thm=gs-diff}
	There is a natural identification
	\[\PI_{`G}(`W) `~ \faktor{\hom(H_1(`G),\MH{`G})}{\left<\DiffH(`g,`g') \mid `g,`g' `: \OGE{`G}{`W} \right>} \qedhere\]
\end{theorem}
\begin{proof}
	Let $`g, `g' `: \OGE{`G}{`W}$ and $`f `: \hom(H_1(B_{\A}),\MH{`G})$ such that ${(`f `o `g_*)}_{|\MH{`G}} = \Id_{\MH{`G}}$ as in \cref{def=gs} of $\PI_{`G}(`W)$. Then by \cref{prop=ge-diff-class}:
	\[`f `o {(`g_* - `g_*')}_{|H_1(`G)} = `f `o `g_* `o \DiffH(`g,`g') = \DiffH(`g,`g') \qedhere\]
\end{proof}

The remainder of this subsection is dedicated to the proof of \cref{prop=ge-diff-class}. Let $e_{v,w}$ be an edge. Reusing previous notations, we write:
\begin{align*}
	`g(\vec{e}_{v,w}) &= \check{s}_v(\bar{`a}_v^{w}) & `g(\vec{e}_{w,v}) &= \check{s}_w(\bar{`a}_w^{v})
\end{align*}
and similarly for $`g'$.
\begin{lemma}
	\label{lem=hom-vertex}
	Let $i_v : S_v "`->" B_{\A}$ be the inclusion. The following diagram commutes
	\[\begin{tikzcd}
		H_1\left(D_{W_v}\right) \rar["i_v^* `o \check{s}_v^*"] & H_1(B_{\A})\\
		\left<W_v\right> \rar["`h"'] \uar["h_v"] & \MH{`G} \uar["`g_*"']
	\end{tikzcd} \qedhere\]
\end{lemma}
\begin{proof}
	For any $w `: W_v$, the closed curve $\partial^{w} \check{s}_v$ inside $S_v$ is sent by the gluing map along the torus $T_{v,w}$ to a fibre of $S_w$. Its homological class inside $H_1(B_{\A})$ is thus equal to $f_w$. This precisely means that $i_v^* `o \check{s}_v^* `o h_v (w) = f_w$. But by \cref{thm=gm-hom}, we also have $`g_*(\overline{w}) = f_w$ by construction of~$`g_*$.
\end{proof}
\begin{figure}[hbt]
	\centering
	\input{cycle_difference_map.pgf}
	\caption{Difference between two graphed embeddings}
	\label{fig=ge-diff}
\end{figure}
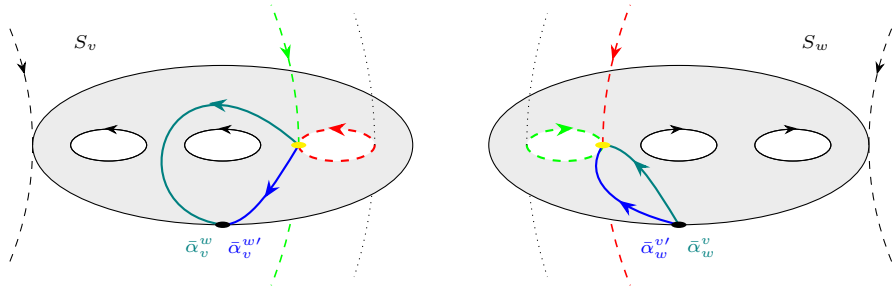
\begin{proof}[Proof of \cref{prop=ge-diff-class}]
	Let $c `: H_1(`G)$ and write $`z(c) = \sum_{k} {e_{v_k,v_{k+1}}}$. By construction of the embedding $`g `: \OGE{`G}{`W}$ we have
	\[`g(`z(c)) = \prod_k {`g(e_{v_k,v_{k+1}})}\]
	Inside $H_1(M)$, $(`g_* - `g_*')(c)$ can be seen as the class of the closed curve
	\[`g(`z(c)) `. {`g'(`z(c))}^{-1}\]
	As shown on \Cref{fig=ge-diff}, a 2-chain bordering that curve inside $M$ can be decomposed into a sum of 2-chains bordering each of the closed curves
	\begin{align*}
		`l_{v_k}^{v_{k+1}} &\vcentcolon= `a_{v_k}^{v_{k+1}} `. {\left(`a_{v_k}^{\prime\, v_{k+1}}\right)}^{-1} & `l^{v_k}_{v_{k+1}} &\vcentcolon= `a^{v_k}_{v_{k+1}} `. {\left(`a_{v_{k+1}}^{\prime\, v_k}\right)}^{-1}
	\end{align*}
	lying in $\check{s}_{v_k}$ and $\check{s}_{v_{k+1}}$ respectively. Given an orientation $`d$, this yields the equation inside~$H_1(B_{\A})$:
	\[(`g_* - `g_*')(c) = \sum_k `d_k \left({[`l_{v_k}^{v_{k+1}}]} - {[`l_{v_{k+1}}^{v_{k}}]}\right)\]
	where $[\,`.\,]$ designates the class of the closed curve inside $H_1(B_{\A})$ and~$`d_k \vcentcolon= `d_{v_k,v_{k+1}}$. Applying successively \cref{lem=hom-vertex} and \cref{def=ge-diff-edge} of the half-edge difference map $\DiffE(`g,`g')$, we get the following equality:
	\begin{align*}
		{[`l_{v_k}^{v_{k+1}}]} - {[`l_{v_{k+1}}^{v_{k}}]} &= i_{v_k}^* `o \check{s}_{v_k}^* `o h_{v_k} `o h_{v_k}^{-1} (\bar{`l}_{v_k}^{v_{k+1}}) - i_{v_{k+1}}^* `o \check{s}_{v_{k+1}}^* `o h_{v_{k+1}} `o h_{v_{k+1}}^{-1} (\bar{`l}_{v_{k+1}}^{v_{k}})\\
			&=`g_* `o `h `o \left(\DiffE(`g,`g')(\vec{e}_{v_k,v_{k+1}}) - \DiffE(`g,`g')(\vec{e}_{v_{k+1},v_k})\right)
	\end{align*}
	Replacing in the sum yields:
	\begin{align*}
		(`g_* - `g_*')(c) &= \sum_k `d_k `. `g_* `o `h `o \left(\DiffE(`g,`g')(\vec{e}_{v_k,v_{k+1}}) - \DiffE(`g,`g')(\vec{e}_{v_{k+1},v_k})\right)\\
		&= `g_* `o \left((`c_{`d} `o `z)^\dagger `(`*) `h\right) `o \DiffE(`g,`g') (c)\\
		&= `g_* `o \DiffH(`g,`g') (c) \qedhere
	\end{align*}
\end{proof}

\subsection{Relations of the graph stabiliser}
\label{ssec=gs-pres}

The results obtained in \cref{ssec=mcg-star} allows us to compute explicitly the image of the cycle difference map $\DiffH$. Combining this with \cref{thm=gs-diff} gives an explicit combinatorial presentation of the graph stabiliser $\PI_{`G}(`W)$.
\begin{theorem}
	\label{thm=gs-pres}
	The group $\PI_{`G}(`W)$ is finitely presented and admits the presentation:
	\begin{description}[font={\normalfont},labelindent=2em,labelsep=2em,itemsep=3pt]
		\item[Generators] $c `(`*) \overline{v}$ for every $c `: H^1(`G)$ and $v `: V(`G)$.
		\item[Relations] the images by $`z^\dagger `(`*) `h : \hom(C_1(`G),C_0(`G)) \to \hom(H_1(`G),\MH{`G})$ of:
		\begin{enumerate}[(GS1),font=\normalfont,labelsep=1em,itemsep=3pt]
			\item \label{gs-gen-1} $e_{v,x}^\ddagger `(`*) v$ and $e_{v,x}^\ddagger `(`*) x$ for every edge $e_{v,x}$.
			\item \label{gs-gen-2} $e_{v,y}^\ddagger `(`*) z + `d_{v,y}`d_{v,z} `. e_{v,z}^\ddagger `(`*) y$ for every pair $(e_{v,y},e_{v,z})$ of adjacent edges in~$`G$.\qedhere
		\end{enumerate}
	\end{description}
\end{theorem}
Note that \cref{thm=gs-pres} directly implies \cref{cor=gs-ordered-model}.
\begin{proof}
	We reuse the notations from \cref{sec=star,ssec=diff-maps}. By \cref{thm=gs-diff} we need to compute
	\[\left<\DiffH(`g,`g') \mid `g,`g' `: \OGE{`G}{`W} \right> \subset H^1(`G) `(`*) \MH{`G}\]
	Applying \cref{def=ge-diff}, this subgroup is isomorphic to the image by $(`c_{`d} `o `z)^\dagger `(`*) `h$ of the subgroup
	\[\left<\DiffE(`g,`g') \mid `g `: \OGE{`G}{`W}, `J `: \HGpp{M} \right> \subset \vec{C}^1(`G) `(`*) C_0(`G)\]
	By \cref{prop=act-h++-tr}, this subgroup is in turn isomorphic to
	\[\vec{G}_{`G} \vcentcolon= \left<\DiffE(`g, `J `. `g) \mid `g `: \OGE{`G}{`W}, `J `: \HGpp{M}\right>\]
	Therefore, we need to determine a generating set of $\vec{G}_{`G}$ as a submodule of $\hom(\vec{C}_1(`G),C_0(`G))$.

	Let $`J `: \HGpp{M}$ and $`g `: \OGE{`G}{`W}$. Fix $v `: V(`G)$ and write $`a_v = `g \cap S_v$. Up to isotopy, $`J$ restricts to a fibre-positive homeomorphism $`J_v$ of $S_v$ and we have $`J `. `g \cap S_v = `J_v `. `a_v$. By~\cref{lem=star-flat}, we can also suppose that $`g \cap S_v = \check{s}_v (\bar{`a}_v)$ where $\bar{`a}_v `: \OLS{D_{W_v}}{`q_v}$. Since $`J_v$ is fibrewise and fixes $\partial S_v$ componentwise, its action on $\check{s}_v$ is completely determined by an element $`j_v `: \PMCG{D_{W_v}}$, and we have
	\[`J `. `g \cap S_v = \check{s}_v \left(`j_v `. \bar{`a}_v\right)\]
	We can then rewrite $\vec{G}_{`G} = {\bigoplus}_{v `: V(`G)} \vec{G}_{`G}^v$ with
	\begin{equation*}
		\vec{G}_{`G}^v \vcentcolon= \left<h_v^{-1} `o \DiffS{W_v}(\bar{`a}_v, `j_v `. \bar{`a}_v) \mid \bar{`a}_v `: \OLS{D_{W_v}}{`q_v}, `j_v `: \PMCG{D_{W_v}}\right>
	\end{equation*}
	where the domain $W_v$ of $\DiffS{W_v}(\bar{`a}_v, `j_v `. \bar{`a}_v)$ is viewed as the subset $\vec{W}_v$ of $\vec{C}_1(`G)$.

	To compute a presentation of $\vec{G}_{`G}^v$ we will reuse the model disc $D_m$ from \cref{ssec=model}. Denote~by $\DiffS{m}$ the star difference map on $D_m$, and let
	\[G_{m} \vcentcolon= \left<h^{-1} `o \DiffS{m}(\bar{`a}, `j `. \bar{`a}) \mid \bar{`a} `: \LS{D_m}, `j `: \PMCG{D_{m}}\right>\]
	where $h: \left<x^1, \dots, x^m\right> "->" H_1(D_m)$ is the natural group isomorphism.
	By \cref{prop=mcg-model-fth}, the action of $\PMCG{D_m}$ on $\LS{D_m}$ is transitive. Using the standard star $\bar{`a}_0$ from \Cref{fig=star-stand}, we can therefore rewrite
	\[G_{m} \vcentcolon= \left<h^{-1} `o \DiffS{m}(\bar{`a}_0, `j `. \bar{`a}_0) \mid `j `: \PMCG{D_{m}}\right>\]
	\vspace{-\baselineskip}
	\begin{lemma}
		\label{lem=hom-diff-model}
		The group $G_m$ is freely generated by the union $R_1 \sqcup R_2$, where
		\begin{align*}
			R_1 &\vcentcolon= \left\{{(x^j)}^{\ddagger} `(`*) x^j \mid 1 \leq j \leq m \right\}\\
			R_2 &\vcentcolon= \left\{{(x^k)}^{\ddagger} `(`*) (x^k + x^l) + {(x^l)}^{\ddagger} `(`*) (x^k + x^l) \mid 1 \leq k < l \leq m \right\} \qedhere
		\end{align*}
	\end{lemma}
	For every $v `: V(`G)$, the bijection $M_{`q_v} : D_m "->" D_{W_v}$ induces a bijection $G_{m_v} `~ \vec{G}_{`G}^v$ given by
	\begin{equation}
		\label{eq=hom-model}
		{(x^k)}^{\ddagger} `(`*) x^l "|->" \vec{e}_{v,y}^{\: \ddagger} `(`*) z, \quad \text{where }`q_v(y) = k \text{ and } `q_v(z) = l
	\end{equation}
	From \cref{lem=hom-diff-model}, we get that
	\[\vec{G}_{`G}^v =
	\left<
		\vec{e}_{v,x}^{\: \ddagger} `(`*) x, \quad
		\vec{e}_{v,y}^{\: \ddagger} `(`*) (y + z) + \vec{e}_{v,z}^{\: \ddagger} `(`*) (y + z) \mid
	x,y,z `: W_v
	\right>\]
	This set of generators can be further simplified by subtracting the first generator from the second when $x = y$ and $x = z$. Therefore
	\[\vec{G}_{`G}^v =
	\left<
	\vec{e}_{v,x}^{\: \ddagger} `(`*) x, \quad
	\vec{e}_{v,y}^{\: \ddagger} `(`*) z + \vec{e}_{v,z}^{\: \ddagger} `(`*) y \mid
	x,y,z `: W_v
	\right>\]
	Repeating this for all vertices $v `: V(`G)$ gives the following generating set of $\vec{G}_{`G}$:
	\begin{equation}
		\label{eq=gs-pres-he}
		\vec{G}_{`G} =
		\left<
		\vec{e}_{v,x}^{\: \ddagger} `(`*) x, \quad
		\vec{e}_{v,y}^{\: \ddagger} `(`*) z + \vec{e}_{v,z}^{\: \ddagger} `(`*) y \mid
		\vec{e}_{v,x}, \vec{e}_{v,y}, \vec{e}_{v,z} `: \vec{C}_1(`G)
		\right>
	\end{equation}
	Finally, for every edge $e_{v,w}$, one has
	\[`c^\dagger_{`d}(\vec{e}_{v,w}^{\: \ddagger}) = - `c^\dagger_{`d}(\vec{e}_{w,v}^{\: \ddagger}) = `d_{v,w} `. e_{v,w}^{\ddagger}\]
	Therefore taking the image of \eqref{eq=gs-pres-he} by $`c^{\dagger}_{`d}$ gives exactly the set of generators of \cref{thm=gs-pres}, now lying in $\hom(C_1(`G),C_0(`G))$ as required.
\end{proof}
\begin{figure}[bh]
	\centering
	\begin{subfigure}{\linewidth}
		\centering
		\input{mcg_action_initial.pgf}
		\caption{Initial setting}
		\label{fig=mcg-init}
	\end{subfigure}
	\begin{subfigure}{\linewidth}
		\centering
		\input{mcg_action_pure_braid.pgf}
		\caption{Action of $a_{k,l}$}
		\label{fig=mcg-pb-gen}
	\end{subfigure}
	\caption{Action of $\mathbb{P}_m$ on a linear star}
	\label{fig=mcg-pb}
\end{figure}
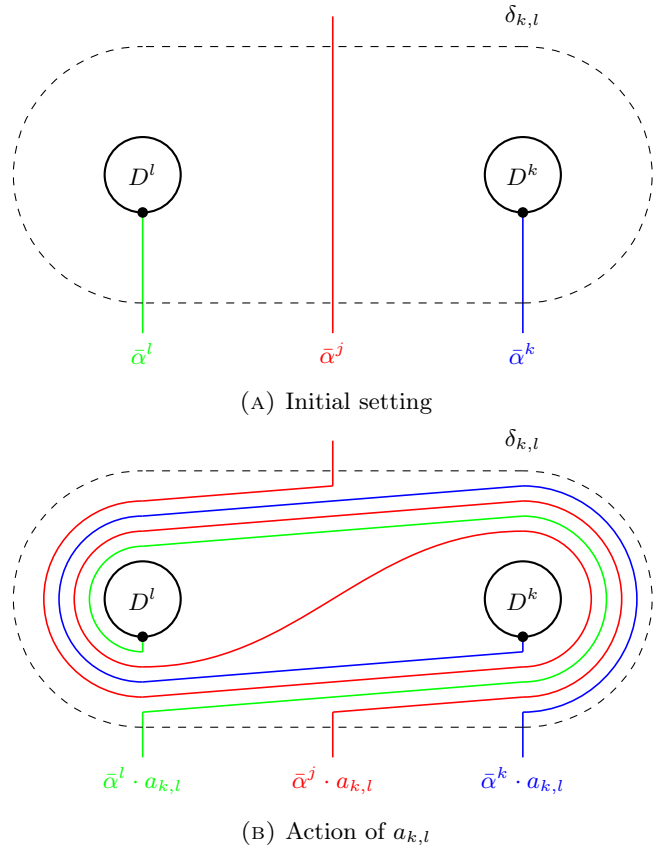
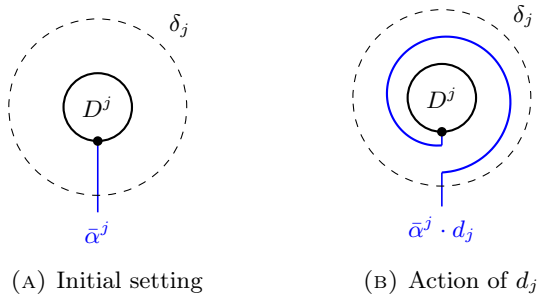
\begin{figure}[bh]
	\centering
	\begin{subfigure}{.3\linewidth}
		\centering
		\input{mcg_action_initial_dt.pgf}
		\caption{Initial setting}
		\label{fig=mcg-init-dt}
	\end{subfigure}
	\begin{subfigure}{.3\linewidth}
		\centering
		\input{mcg_action_dehn_twist.pgf}
		\caption{Action of $d_{j}$}
		\label{fig=mcg-dt-gen}
	\end{subfigure}
	\caption{Action of $\mathbb{Z}^m$ on a linear star}
	\label{fig=mcg-dt}
\end{figure}
\begin{proof}[Proof of \cref{lem=hom-diff-model}]
	By \cref{thm=mcg-disc}, every $`j `: \PMCG{D_{m}}$ is of the form $(`b,d)$, with
	\begin{align*}
		`b &= \prod_{m=1}^{r} a_{k_m,l_m} & d &= \sum_{m=1}^{s} d_{j_m}
	\end{align*}
	Each generator $a_{k,l}$ (resp. $d_j$) correspond to a full Dehn twist inside a curve $`d_{k,l}$ (resp. $`d_j$). We~say that a star $\bar{`a} `: \LS{m}$ is in \emph{standard position} if:
	\begin{enumerate}[(S1)]
		\item for all curves $`d_j$, the pair $(\bar{`a},`d_j)$ is isotopic to \Cref{fig=mcg-init-dt}.
		\item for all curves $`d_{k,l}$, the pair $(\bar{`a},`d_{k,l})$ is isotopic to \Cref{fig=mcg-init} with possibly several (or none) parallel \enquote{middle lines}, all of which are parts of $\bar{`a}$.
	\end{enumerate}
	We will prove by double induction on $r \geq 0$ and $s \geq 0$ that
	\begin{equation}
		\label{eq=diff-gen}
		\tag{$\mathrm{P}_{r,s}$}
		\begin{cases}
			h^{-1} `o \DiffS{m}(\bar{`a}_0, \bar{`a}_0 `. {(`b,d)}) `: R_1 \sqcup R_2\\
			\bar{`a}_0 `. {(`b,d)} \text{ is in standard position.}
		\end{cases}
	\end{equation}
	The standard star $\bar{`a}_0$ is clearly in standard position, with no middle lines. Now fix $r,s \geq 0$ and suppose that $\bar{`a} \vcentcolon= \bar{`a}_0 `. {(`b,d)}$ verifies \eqref{eq=diff-gen}.

	Let $1 \leq j \leq m$. Since $\bar{`a}$ is in standard position, the action of $d_j$ on $\bar{`a}$ is shown on \Cref{fig=mcg-dt}. By superimposing \Cref{fig=mcg-init-dt,fig=mcg-dt-gen}, one sees that
	\begin{equation}
		\label{eq=diff-dt}
		h^{-1} `o \DiffS{m}(\bar{`a}, \bar{`a} `. {(1,d_j)}) = {(x^j)}^{\ddagger} `(`*) x^j
	\end{equation}
	and therefore
	\[h^{-1} `o \DiffS{m}(\bar{`a}_0, \bar{`a} `. {(1,d_j)}) = h^{-1} \left(\DiffS{m}(\bar{`a}_0, \bar{`a}) + \DiffS{m}(\bar{`a}, \bar{`a} `. {(1,d_j)})\right) `: R_1 \sqcup R_2\]
	By retracting the curve $`d_j$ closer to $\partial^j D_m$ on \Cref{fig=mcg-dt-gen}, we can put $d_j `. \bar{`a}$ back on standard position with respect to all $`d_{j}$'s. By retracting the curves $`d_{j,k}$ for every $k \neq j$, we can also put $d_j `. \bar{`a}$ in standard position with respect to them, with one middle line added.

	Let $1 \leq k < l \leq m$. Since $\bar{`a}$ is in standard position, the action of $a_{k,l}$ on $\bar{`a}$ is shown on \Cref{fig=mcg-pb}. By superimposing \Cref{fig=mcg-init,fig=mcg-pb-gen}, one sees that
	\begin{equation}
		\label{eq=diff-pb}
		h^{-1} `o \DiffS{m}(\bar{`a}, \bar{`a} `. {(a_{k,l},0)}) = {(x^k)}^{\ddagger} `(`*) (x^k + x^l) + {(x^l)}^{\ddagger} `(`*) (x^k + x^l)
	\end{equation}
	and therefore
	\[h^{-1} `o \DiffS{m}(\bar{`a}_0, \bar{`a} `. {(a_{k,l},0)}) = h^{-1} `o \left(\DiffS{m}(\bar{`a}_0, \bar{`a}) + \DiffS{m}(\bar{`a}, \bar{`a} `. {(a_{k,l},0)})\right) `: R_1 \sqcup R_2\]
	In particular the middle lines do not contribute to the homological difference. By retracting the curves $`d_k$ and $`d_{l}$ on \Cref{fig=mcg-pb-gen}, we can put $\bar{`a} `. a_{k,l}$ back on standard position with respect to all~$`d_{j}$'s. Now let $1 \leq k' < l' \leq m$ be another pair and suppose that (say) $l = k'$. To put $\bar{`a} `. a_{k,l}$ back on standard position, we can retract $`d_{l,l'}$ closer to $\partial^l D_m$ on \Cref{fig=mcg-init}. This creates some new middle lines for $`d_{l,l'}$. Similarly, $`d_{l,k}$ itself can be retracted closer to $\partial^l D_m$ and $\partial^k D_m$, and keep the same middle lines.

	This achieves the induction, and proves that $G_m \subset \left< R_1 \sqcup R_2 \right>$. Conversely, applying \eqref{eq=diff-dt} and \eqref{eq=diff-pb} to $\bar{`a}_0$ for all $d_j$'s and all $a_{k,l}$'s shows that $R_1 \sqcup R_2 \subset G_m$.
\end{proof}

\subsection{Change of graph ordering}
\label{ssec=gs-trans}

By \cref{cor=gs-ordered-model} the graph stabiliser $\PI_{`G}$ does not depend as a group on the choice of a graph ordering $`W$. However, the original presentation of $\PI_{`G}$ given in \cref{def=gs} does involve a choice of $`W$, and it is this presentation that is used to define (and compute) the homology inclusion invariant $\mathcal{J}_{`W}$. \cref{prop=gs-trans-def} defines an explicit \emph{transition function} $T$ that describes the effect of changing $`W$ in the original presentation of $\PI_{`G}$.

We reuse notations from \cref{ssec=ord}. A graph ordering $`W$ on $`G$ is an element of the set
\[\mathcal{C}(`G) \vcentcolon= \prod_{v `: V(`G)} \Circ{W_v}\]
Consider the group of permutations on graph orderings given by:
\[\mathfrak{S}_{`G} \vcentcolon= \prod_{v `: V(`G)} \mathfrak{S}_{m_v}\]
The group $\mathfrak{S}_{`G}$ has a transitive right action on $\mathcal{C}(`G)$ by \cref{lem=circ-act}.

Fix a vertex $v `: V(`G)$. For $`s `: \mathfrak{S}_{m_v}$ acting on $W_v$, let $`s^v$ be the element of $\mathfrak{S}_{`G}$ defined as $`s$ on $W_v$ and as the identity on every other neighbour set. The group $\mathfrak{S}_{m_v}$ can be generated only by the permutations of type $`t_k \vcentcolon= (k,k+1)$ with the relations:
\begin{align}
	\label{eq=transp-rel}
	{(`t_k)}^2 &= \Id & `t_k `t_j &= `t_j `t_k \text{ if } \left|k-j\right| > 1 & {(`t_k `t_{k+1})}^3 &= \Id
\end{align}
Therefore the group $\mathfrak{S}_{`G}$ can be generated by the $`t_k^v$ for every $v `: V(`G)$ and $1 \leq k \leq m_v$.
\begin{proposition}
	\label{prop=gs-trans-def}
	The group homomorphism
	\[T: \mathfrak{S}_{`G} \longrightarrow \PI_{`G}\]
	given by, for every $v `: V(`G)$ and $1 \leq k \leq m_v$:
	\begin{equation}
		\label{eq=gs-trans}
		T(`t_k^v) = ({`z}^\dagger `(`*) `h) \left(`d_{v,y} `. e_{v,y}^{\ddagger} `(`*) x\right) \quad \text{ where } `q_v(x) = k,\: `q_v(y) = k+1
	\end{equation}
	verifies the following property: if $`W, `W' `: \mathcal{C}(`G)$ are such that $`W' = `W `. `S$ with $`S `: \mathfrak{S}_{`G}$, then for every $`f `: \hom(H_1(B_{\A}),\MH{`G})$ as in \cref{def=gs} and every $`g `: \OGE{`G}{`W}, `g' `: \OGE{`G}{`W'}$, one has:
	\[\left[`f `o `g_*'\right] = \left[`f `o `g_*\right] + T({`S})\]
	inside $\PI_{`G}$.
\end{proposition}
The kernel of $T$ contains the subgroup $\{(1 \cdots m_v)^v \mid v `: V(`G)\}$ generated by the circular permutations that leave any circular order $`W$ on $`G$ unchanged.
\begin{remark}
	From a computing perspective, \cref{prop=gs-trans-def} makes it possible to compute two homology inclusion values $\mathcal{J}_{`W}(\A)$ and $\mathcal{J}_{`W'}(\A')$ with different graph orderings and still compare them. If there is an ordered oriented equivalence between $\A$ and $\A'$, then the difference $\mathcal{J}_{`W'}(\A') - \mathcal{J}_{`W}(\A)$ would be equal to $T({`S})$ with $`W' = `W `. `S$.
\end{remark}

\begin{figure}[b]
	\centering
	\begin{subfigure}{.45\linewidth}
		\centering
		\input{mcg_action_initial_bis.pgf}
		\caption{Initial setting}
		\label{fig=mcg-init-transp}
	\end{subfigure}
	\begin{subfigure}{.45\linewidth}
		\centering
		\input{mcg_action_braid.pgf}
		\caption{Action of $`s_{k,k+1}$}
		\label{fig=mcg-transp-gen}
	\end{subfigure}
	\caption{Action of $\mathbb{B}_m$ on the standard linear star $\bar{`a}_0$.}
	\label{fig=mcg-transp}
\end{figure}
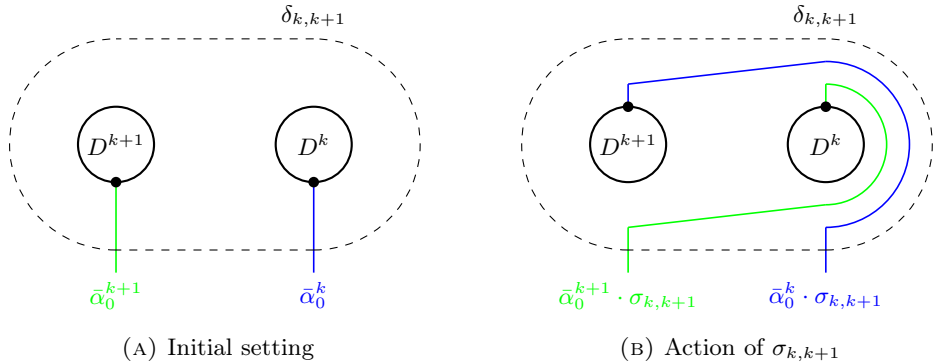

\begin{proof}
	We first check that $T$ respects the relations of the presentation of $\mathfrak{S}_{m_v}$ given in \eqref{eq=transp-rel}. The~second one is obviously respected. Write $`s = {(`t_k)}^2$. Then
	\[T(`s^v) = ({`z}^\dagger `(`*) `h) \left(`d_{v,y} `. e_{v,y}^{\ddagger} `(`*) x + `d_{v,x} `. e_{v,x}^{\ddagger} `(`*) y\right) = 0\]
	by the relation \ref{gs-gen-2} of $\PI_{`G}$. With $`t = {(`t_k `t_{k+1})}^3$, a similar elementary computation shows that $T({`t^v}) = 0$. Therefore the group homomorphism $T$ is well defined.

	It is enough to prove the proposition when $`S = `t_k^v$ for a fixed $v `: V(`G)$ and $1 \leq k \leq m_v$. Let $`w_v$ be the circular order corresponding to $v$ in $`W$ and fix a linearisation $`q_v$. By \cref{lem=star-flat}, we can suppose that $`g \cap S_v = \bar{`a}_v `: \OLS{W_v}{`q_v}$. By \cref{def=gs} of the graph stabiliser, the class $[`f `o `g_*] `: \PI_{`G}$ does not depend on the choice of $`g `: \OGE{`G}{`W}$, and therefore on the choice of $\bar{`a}_v `: \OLS{W_v}{`q_v}$. Let $M_{`q_v} : D_m "->" D_{W_v}$ be the homeomorphism that sends each boundary component $\partial D^j$ of $D_m$ to $\partial^w D_{W_v}$ with $`q_v(w) = j$. Then we can always suppose that $\bar{`a}_v = M_{`q_v}(\bar{`a}_0)$ where $\bar{`a}_0$ is the standard ordered star in $D_m$ shown on \Cref{fig=star-stand}.

	Consider now the braid $`s_{k,k+1} `: \mathbb{B}_m$ that permutes the strands $k$ and $k + 1$. By \cref{thm=mcg-disc}, there is an associated homeomorphism in $\MCG{D_m}$. Its action on $\bar{`a}_0$ is shown on \Cref{fig=mcg-transp}. We~thus obtain a new star $\bar{`a}_1 \vcentcolon= M_{`q_k}\left(\bar{`a}_0 `. `s_{k,k+1}\right)$ in $D_{W_v}$ which is now ordered by $`q_v `. (k,k+1)$, and a new graphed embedding $`g_1$ ordered by $`W_1 = `W `. (k,k+1)^v = `W'$. Therefore $\left[`f `o {(`g_1)}_*\right] =  [`f `o `g_*']$.

	Remember the difference maps from \cref{ssec=diff-maps}. Reusing the same tools from the proof of \cref{thm=gs-pres}, we write:
	\begin{align*}
		[`f `o `g_*'] - [`f `o `g_*] &= [`f `o ({(`g_1)}_* - `g_*)]\\
		&= \left({(`c_{`d} `o `z)}^\dagger `(`*) `h\right) \left(\DiffE(`g,`g_1)\right)\\
		&= \left({(`c_{`d} `o `z)}^\dagger `(`*) `h\right) \left(h_v^{-1} `o \DiffS{W_v}(\bar{`a},\bar{`a}_1)\right)
	\end{align*}
	Using $M_{`q_v}$ to go back to the model disc $D_{m_v}$, $h_v^{-1} `o \DiffS{W_v}(\bar{`a},\bar{`a}_1)$ is the image by \eqref{eq=hom-model} of $h^{-1} `o \DiffS{m}(\bar{`a}_0, \bar{`a}_0 `. `s_{k,k+1})$. By superimposing \Cref{fig=mcg-init-transp,fig=mcg-transp-gen}, we see that
	\[h^{-1} `o \DiffS{m}(\bar{`a}_0, \bar{`a}_0 `. `s_{k,k+1}) = {(x^{k+1})}^{\ddagger} `(`*) x^k\]
	Applying the \eqref{eq=hom-model} isomorphism as well as $`c_{`d}$ we get exactly the value of $T(`t_k^v)$ from~\eqref{eq=gs-trans}.
\end{proof}

\section{Loop-linking number}
\label{sec=lln}

The first homology groups of the exterior and boundary manifold of a line arrangement are determined by its combinatorics. However, the map ${i_{\A}^* : H_1(B_{\A}) "->" H_1(E_{\A})}$ is not necessarily combinatorial and contains topological information. One of the first invariants developed from $i_{\A}^*$ was the \emph{$\mathcal{I}$-invariant} from \cite{Florens2013}, which was then extended into the \emph{loop-linking number} introduced in \cite{Cadegan2018} and further developed in \cite{GuervilleBalle2022}. In this section we prove that the the homology inclusion extends again these two invariants. Since all three are derived from $i_{\A}^*$, the difference lies mainly on the involved algebraic structures which play a role similar to the graph stabiliser. The fact that the loop-linking number generalises the $\mathcal{I}$-invariant is already covered in \cite{GuervilleBalle2022}. We briefly restate the definitions of the loop-linking number reusing our own notations from \cref{sec=gs}.

Let $\A$ be a line arrangement. Recall from \cref{ssec=la} that $\widehat{`G} \vcentcolon= \widehat{`G}(C_{\A})$ is the \emph{full} incidence graph of $\A$. As in \cref{ssec=hom-incl-def}, consider the map induced by the inclusion on the first homology groups
\[i_{\A}^* : H_1(B_{\A},\mathbb{Z}) \longrightarrow H_1(E_{\A},\mathbb{Z})\]
Let $\widehat{`W}$ be a graph ordering on $\widehat{`G}$. The full incidence graph $\widehat{`G}$ gives a non-minimal graph structure of the boundary manifold $B_{\A}$. One can nevertheless define an embedding of the graph $`g : \widehat{`G} "`->" B_{\A}$ again as a union of stars, one for each bundle piece. Then all the results of \cref{ssec=bm-hom} stand for~$\widehat{`G}$, or any other graph structure of $B_{\A}$ for that matter, see \cite{Westlund1997}. \cref{lem=la-incl-V} stands as well, so for every graphed embedding $`g `: \OGE{\widehat{`G}}{\widehat{`W}}$ we have a map
\[i_{\A}^* `o `g_* :  H_1(\widehat{`G},\mathbb{Z}) \longrightarrow \MH{\widehat{`G}}\]
with $\MH{\widehat{`G}} `~ H_1(E_{\A},\mathbb{Z})$.

Let $T_{\widehat{`G}}$ be the subgroup of $C^0(\widehat{`G},\mathbb{Z}) `(`*) C_1(\widehat{`G},\mathbb{Z})$ generated by the following elements:
\begin{enumerate}[(TLG1),font=\normalfont]
	\item \label{tlg-gen-1} ${L'}^\ddagger `(`*) e_{P,L}$ for every edge $e_{P,L}$ and $L' `: W_P$.
	\item \label{tlg-gen-2} ${P'}^\ddagger `(`*) e_{P,L}$ for every edge $e_{P,L}$ and $P' `: W_L$.
\end{enumerate}

Let $G$ be an abelian group. For every free $\mathbb{Z}$-modules $A$ and $B$, we consider the natural pairings
\begin{gather*}
	`p_{A} : A^\dagger\, {`(`*)}_{\mathbb{Z}}\, G\, {`(`*)}_{\mathbb{Z}}\, A \longrightarrow G\\
	`p_{A,B} : \left(A^\dagger\, {`(`*)}_{\mathbb{Z}}\, G\, {`(`*)}_{\mathbb{Z}}\, B\right) `(`*) \left(B^\dagger\, {`(`*)}_{\mathbb{Z}}\, A\right) \longrightarrow G
\end{gather*}
where $`p_{A,B}$ is obtained by composing $`p_{A}$ with the natural pairing $B `(`*) B^\dagger "->" \mathbb{Z}$.
\begin{definition}
	\label{def=tlg}
	The \emph{tensor linking group} $\TLG{\widehat{`G}}{G}$ is the subgroup of $\hom(\MH{\widehat{`G}},G)\, {`(`*)}_{\mathbb{Z}}\, H_1(\widehat{`G}, \mathbb{Z})$ given by
	all elements $`f$ such that for every $`j `: T_{\widehat{`G}}$:
	\[`p_{C_0,C_1} \left((`h^\dagger `(`*) \Id_G `(`*) `z)(`f) `(`*) `j^\ddagger\right) = 0_G\qedhere\]
\end{definition}
Note that $`j^\ddagger$ designates the \emph{dual element} of $`j$ lying in the dual module $C^1(\widehat{`G}) `(`*) C_0(\widehat{`G})$, which is different from the \emph{dual homomorphism} $`j^\dagger$.

Now define
\[`J \vcentcolon= \Id `(`*) \left(i_{\A}^* `o `g_*\right) : \hom(\MH{\widehat{`G}}, G)\, {`(`*)}_{\mathbb{Z}}\, H_1(\widehat{`G}) \longrightarrow \hom(\MH{\widehat{`G}}, G) \, {`(`*)}_{\mathbb{Z}}\, \MH{\widehat{`G}}\]
\vspace{-\baselineskip}
\begin{definition}
	\label{def=lln}
	The \emph{loop-linking number} of $\A$ is the function
	\[\mathcal{L}_G(\A) : \TLG{\widehat{`G}}{G} \longrightarrow G\]
	defined as the restriction of $`p_{\MH{\widehat{`G}}} `o `J$ on $\TLG{\widehat{`G}}{G}$.
\end{definition}
The algebraic properties of $\TLG{\widehat{`G}}{G}$ ensure that $\mathcal{L}_G(\A)$ depends neither on the choice of the graphed embedding $`g$ nor on the choice of the graph ordering $`W$, see \cite[Section 3.3.1]{Cadegan2018}. Just~like the homology inclusion, the loop-linking number is an invariant of \emph{ordered and oriented} line arrangements, although this is not obvious from the definition.

We now establish that the homology inclusion extends the loop-linking number.
\begin{theorem}
	\label{thm=hom-incl-lln}
	Let $\A, \A' \subset \mathbb{CP}^2$ be two non-exceptional line arrangements with the same combinatorics $C$ and same graph $`G$. If~there exists a graph ordering $`W$ of $`G$ such that $\mathcal{J}_{`W}(\A) = \mathcal{J}_{`W}(\A')$ then necessarily $\mathcal{L}_G(\A) = \mathcal{L}_G(\A')$ for any abelian group $G$.
\end{theorem}
The tensor linking group plays a very similar role for the loop-linking number than the graph stabiliser does for the homology inclusion. Indeed, both are used to express the map $i_{\A}^*$ in a way that does not depend on the choice of the graphed embedding. However, the tensor-linking group is defined using every possible graphed embeddings. On the other hand, the graph stabiliser quotients out the differences between the graphed embeddings \emph{with the same graph ordering} only. This gives a \enquote{bigger} quotient with more possible values for the invariant, as evidenced by \cref{lem=gs-tlg}, and is the main reason why the homology inclusion extracts a finer topological information.

The rest of the section is dedicated to the proof of \cref{thm=hom-incl-lln}.

The following lemma can be established through elementary computations on the presentation of the graph stabiliser given in \cref{thm=gs-pres}, see \cite[Section 1.6.4]{Rodau2023}.
\begin{lemma}
	There is a natural group isomorphism $\PI_{\widehat{`G}} `~ \PI_{`G}$. In particular if $\widehat{`W}$ is a graph ordering on $\widehat{`G}$ and $`W$ is its restriction on $`G$ (see \cref{lem=ord-rest}) then for every graphed embeddings $`g `: \OGE{`G}{`W}$ and $\widehat{`g} `: \OGE{\widehat{`G}}{\widehat{`W}}$ one has $[i_{\A}^* `o `g_*] = [i_{\A}^* `o \widehat{`g}_*] `: \PI_{\widehat{`G}}$.
\end{lemma}
Let $A,B$ be two free abelian groups with fixed bases. For any submodule $H$ of $\hom(A,B) `~ A^\dagger `(`*) B$, we write $H^\ddagger = \{`f^\ddagger \mid `f `: H\} \subset {\hom(A,B)}^\dagger `~ B^\dagger `(`*) A$.

Let $G_{\widehat{`G}}$ be the subgroup of $\hom(C_1(\widehat{`G}),C_0(\widehat{`G}))$ generated by the elements of type \ref{gs-gen-1} and \ref{gs-gen-2} from \cref{thm=gs-pres}.
\begin{lemma}
	\label{lem=gs-tlg}
	$G_{\widehat{`G}} \subset T_{\widehat{`G}}^\ddagger$
\end{lemma}
\begin{proof}
	Rewriting the relations in the bases associated to the full incidence graph $\widehat{`G}$, we get that the submodule $G_{\widehat{`G}} \subset C^1(\widehat{`G}) `(`*) C_0(\widehat{`G})$ is generated by the following elements:
	\begin{enumerate}[(GS1'),font=\normalfont]
		\item \label{gs-dd-1} $e_{P,L}^\ddagger `(`*) L$ and $e_{P,L}^\ddagger `(`*) P$ for every edge $e_{P,L}$.
		\item \label{gs-dd-2} $e_{P,L'}^\ddagger `(`*) L + `d_{P,L}\,`d_{P,L'} `. e_{P,L}^{\ddagger} `(`*) {L'}$ for every $P `: \mathcal{Q}$ and $L,L' `: W_P$.\\
		$e_{L,P'}^\ddagger `(`*) P + `d_{L,P}\,`d_{L,P'} `. e_{L,P}^\ddagger `(`*) {P'}$ for every $L `: \mathcal{L}$ and $P,P' `: W_L$.
	\end{enumerate}
	Separately, the submodule $T_{\widehat{`G}}^\ddagger \subset C^1(\widehat{`G}) `(`*) C_0(\widehat{`G})$ is generated by the following elements:
	\begin{enumerate}[(TLG1'),font=\normalfont]
		\item \label{tlg-gen-d-1} $e_{P,L}^\ddagger `(`*) {L'}$ for every edge $e_{P,L}$ and $L' `: W_P$.
		\item \label{tlg-gen-d-2} $e_{P,L}^\ddagger `(`*) {P'}$ for  for every edge $e_{P,L}$ and $P' `: W_L$.
	\end{enumerate}
	Particularising \ref{tlg-gen-d-1} when $L' = L$ gives the first part of \ref{gs-dd-1}, and doing the same with \ref{tlg-gen-d-2} when $P' = P$ gives the second part. It is also clear that the first (resp. second) part of \ref{gs-dd-2} is generated by adding two elements of the first part of \ref{tlg-gen-d-1} (resp. \ref{tlg-gen-d-2}). Therefore~$G_{\widehat{`G}} \subset T_{\widehat{`G}}^\ddagger$.
\end{proof}
\begin{proof}[Proof of \cref{thm=hom-incl-lln}]
	Consider the submodule of $\hom(H_1(\widehat{`G}), \MH{\widehat{`G}})$ defined by
	\[D \vcentcolon= \left< i_{\A}^* `o `g_* - i_{\A'}^* `o `g_*' \mid `g,`g' `: \OGE{\widehat{`G}}{\widehat{`W}} \right>\]
	Now fix $d`: D$ and $`f `: \TLG{\widehat{`G}}{G}$. By \cref{def=lln} of the loop-linking number, we have:
	\begin{align*}
		\left(\mathcal{L}_G(\A) - \mathcal{L}_G(\A')\right)(`f) &= `p_{\MH{\widehat{`G}}} \left(d `o `f\right)\\
		&= `p_{\MH{\widehat{`G}}, H_1(\widehat{`G})} \left(`f `(`*) d\right)
	\end{align*}
	By assumption and by \cref{def=hom-incl} of the homology inclusion, $D$ projects to $\{0\}$ in $\PI_{\widehat{`G}}$. By~\cref{thm=gs-pres,lem=gs-tlg}, this means that there exists $`j `: T_{\widehat{`G}}$ such that $d = (`z^\dagger `(`*) `h)(`j^{\ddagger})$. Then:
	\begin{align*}
		\left(\mathcal{L}_G(\A) - \mathcal{L}_G(\A')\right)(`f) &= `p_{\MH{\widehat{`G}}, H_1(\widehat{`G})} \left(`f `(`*) (`z^\dagger `(`*) `h)(`j^{\ddagger})\right)\\
		&= `p_{C_0, C_1} \left((`h^\dagger `(`*) \Id_G `(`*) `z)(`f) `(`*) `j^\ddagger\right)
	\end{align*}
	which is equal to $0_G$ by \cref{def=tlg} of $T_{\widehat{`G}}$.
\end{proof}
\section{Computations}
\label{sec=calc}
The practical method for computing the homology inclusion of a line arrangement has been developed in collaboration with E.~Artal and B.~Guerville-Ballé and will be detailed in another publication. We give here a quick overview of this method before presenting its main results.

The computation is done in two main phases. The first is to compute an explicit presentation of the graph stabiliser $\PI_{`G}$, which is done in \cref{sec=gs}. The second is to compute the values of the morphism
\[i_{\A}^* `o `g_* : H_1(`G) \longrightarrow \MH{`G}\]
By \cref{def=hom-incl} of the invariant, the values can be computed for any graphed embedding $`g `: \OGE{`G}{`W}$, since the image of the morphism inside $\PI_{`G}$ does not depend on this choice. In practice this requires to explicitly construct the following embeddings all at the same time:
\[\begin{tikzcd}[cramped]
	`G \rar[hook,"`g"] & B_{\A} \rar[hook,"i_{\A}"] & E_{\A} \rar[hook] & \mathbb{CP}^2
\end{tikzcd}\]
The \emph{wiring diagram} of W.~Arvola \cite{Arvola1992} and the \emph{braid monodromy}, which is related to the Zariski-van Kampen method, are two slightly different but equivalent tools that allow to characterise the topology of a complex line arrangement in an abstract form. Both tools do in fact describe the embedding $i_{\A}$ up to isotopy in a way that allows to identify the graph structure of $B_{\A}$. We~devised methods to construct an explicit ordered graphed embedding using either one of these tools. These methods allow to keep control over the actual ordering of the embedding. This point is paramount since two homology inclusion values can be compared only if their respective graph ordering are known. We~also devised an algorithm describing the computation of the value of $i_{\A}^* `o `g_*$ from the raw data of the wiring diagram or braid monodromy. The full computation process is implemented using the \texttt{Sage}~\cite{sagemath} language.

Since the graph stabiliser is an abelian group of finite type, it admits a Smith normal form which can easily be determined with a computer. The values of the homology inclusion are given inside this basis.

We now give several examples of Zariski pairs which are identified by the homology inclusion.

\begin{example}[MacLane arrangements]
	The MacLane arrangements and their common combinatorics are presented at the end of \cref{ssec=la}. It is known (see \cite{Bjoerner1999}) that this constitutes the smallest possible combinatorics that does not admit a realisation in $\mathbb{RP}^2$. The automorphism group of the combinatorics is isomorphic to $\mathrm{GL}_2(\mathbb{F}_3)$. All automorphisms can be effectively realised as projective automorphisms of $\mathbb{CP}^2$, and those lying in the subgroup $\mathrm{SL}_2(\mathbb{F}_3)$ preserve the~order. We~now give the results of the computations of the graph stabiliser and the homology inclusion of $\mathcal{M}^+$~and~$\mathcal{M}^-$.

	The Smith normal form of the graph stabiliser is
	\[\PI_{`G} `~ \faktor{\mathbb{Z}}{3\mathbb{Z}} `* \mathbb{Z}^{35}\]
	The~two values of the homology inclusion in the corresponding basis are given in \Cref{fig=hom-incl-ml}.
	The~difference is a non-zero element of the torsion part of $\PI_{`G}$. The line arrangements $(\mathcal{M}^+,\mathcal{M}^-)$ therefore form an ordered oriented Zariski pair.
	\begin{figure}[hbt]
		\centering
		\setlength{\tabcolsep}{2pt}
		\begin{tabular}{l*{36}{c}}
			$\left[i_{\mathcal{M}^+}^* `o `g_{+}^*\right]$: & $\bar{0}$ & 0 & 1 & 2 & 0 & 0 & 0 & 1 & 0 & 1 & 0 & 1 & 2 & 0 & 1 & 0 & 1 & 1 & 0 & 1 & 1 & 1 & 0 & 0 & -1 & 1 & -1 & -1 & 0 & 1 & 0 & 0 & 0 & 0 & 1 & 0 \\[1em]
			$\left[i_{\mathcal{M}^-}^* `o `g_{-}^*\right]$: & $\bar{2}$ & 0 & 1 & 2 & 0 & 0 & 0 & 1 & 0 & 1 & 0 & 1 & 2 & 0 & 1 & 0 & 1 & 1 & 0 & 1 & 1 & 1 & 0 & 0 & -1 & 1 & -1 & -1 & 0 & 1 & 0 & 0 & 0 & 0 & 1 & 0 \\[1em]
			Difference:                                 & $\bar{2}$ & 0 & 0 & 0 & 0 & 0 & 0 & 0 & 0 & 0 & 0 & 0 & 0 & 0 & 0 & 0 & 0 & 0 & 0 & 0 & 0 & 0 & 0 & 0 & 0  & 0 & 0  & 0  & 0 & 0 & 0 & 0 & 0 & 0 & 0 & 0
		\end{tabular}
		\caption{Homology inclusion values of the MacLane arrangements}
		\label{fig=hom-incl-ml}
	\end{figure}
\end{example}
\begin{example}[Rybnikov quadruplet]
	This is the first Zariski pair of line arrangements identified by G.~Rybnikov in \cite{Rybnikov1998}. It consists of four conjugated line arrangements with $13$ lines $R_1, R_2$ and their complex conjugates $\overline{R_1}, \overline{R_2}$. Their common combinatorics has a non-trivial automorphism group. The Smith normal form of the graph stabiliser is:
	\[\PI_{`G} `~ {\left(\faktor{\mathbb{Z}}{3\mathbb{Z}}\right)}^2 `* \mathbb{Z}^{220}\]
	The homology inclusion values again differ only on the torsion part:
	\begin{align*}
		R_1 &: \left(\bar{2},\bar{2}\right) & \overline{R_1} &: \left(\bar{0},\bar{0}\right) & R_2 &: \left(\bar{0},\bar{2}\right) & \overline{R_2} &: \left(\bar{2},\bar{0}\right)
	\end{align*}
	This means that the four arrangements $(R_1, R_2, \overline{R_1}, \overline{R_2})$ form an ordered oriented Zariski quadruplet. Moreover, the two pairs $(R_1,R_2)$ and $(R_1,\overline{R_2})$ and their respective conjugates all form unoriented ordered Zariski pairs.
\end{example}
Both the MacLane ordered pair and the Rybnikov quadruplet are known examples that were already detected by the loop-linking number \cite{GuervilleBalle2022}. However, the next example is a new Zariski pair that in particular is \emph{not} distinguished by the loop-linking number.
\begin{example}[New Zariski quadruplet]
	Consider the polynomial
	\[P = X^4 + 2X^3 + 4X^2 + 3X + 1\]
	and the equations given by
	\begin{align*}
		L_0 & : 0 = z                                    & L_6    & : 0 = x                  \\
		L_1 & : 0 = `w^2 x - y - `w (`w+1) z             & L_7    & : 0 = x - z              \\
		L_2 & \omit\rlap{:	$0 = (3 `w^2 + 3 `w + 1) x + {(`w+ 1)}^2 y - (`w^3 + 5 `w^2 + 5 `w + 2) z$} & & \\
		L_3 & : 0 = `w(`w^2 + `w +1) x + y + `w(`w +1) z & L_8    & : 0 = y                   \\
		L_4 & : 0 = `w x + y                             & L_9    & : 0 = y - z               \\
		L_5 & : 0 = `w x + y - (`w + 1) z                & L_{10} & : 0 = y + `w(`w^2 + 2`w + 2) z
	\end{align*}
	where $`w = -\frac{1}{2} `+ \frac{1}{2} i \sqrt{5 `+ 2 \sqrt{5}}$ takes the values of the four roots of $P$. This defines four conjugated arrangements with $11$ lines $B_1,B_2,\overline{B_1},\overline{B_2}$ whose common ordered combinatorics is given by
	{%
		\begin{align*}
			&[[0, 1, 2], [0, 3], [0, 4, 5], [0, 6, 7], [0, 8, 9, 10],[1, 3, 6], [1, 4, 7], [1, 5, 8], \\
			&[1, 9], [1, 10],[2, 3, 5], [2, 4], [2, 6, 10], [2, 7], [2, 8], [2, 9],[3, 4, 9], [3, 7, 10], \\
			&[3, 8], [4, 6, 8], [4, 10], [5, 6], [5, 7, 9], [5, 10], [6, 9], [7, 8]]
	\end{align*}}
	The Smith normal form of the graph stabiliser is:
	\[\PI_{`G} `~ \faktor{\mathbb{Z}}{5\mathbb{Z}} `* \mathbb{Z}^{119}\]
	The values of the homology inclusion of the four arrangements are identical on the free part but differ on the torsion part:
	\begin{align*}
		B_1 &: \bar{1} & B_2 &: \bar{4} & \overline{B_1} &: \bar{3} & \overline{B_2} &: \bar{0}
	\end{align*}
	The automorphism group of the combinatorics is trivial, which means that the four arrangements form an unordered oriented Zariski quadruplet. The two pairs $(B_1,B_2)$ and $(B_1,\overline{B_2})$ and their respective conjugates are unordered unoriented (i.e. proper) Zariski pairs.
\end{example}
\begin{remark}
	For all the ordered Zariski pairs we obtain, the graph stabiliser $\PI_{`G}$ contains a torsion part (this is not always the case) and the difference of the homology inclusion values lies only in that torsion part. The loop-linking invariant had a similar behaviour in the sense that all Zariski pairs it could detect had a tensor-linking group $\TLG{\widehat{`G}}{G}$ which also was torsion.
\end{remark}
\printbibliography
\end{document}

%% file: MacLane_graph.pgf
\begin{tikzpicture}[scale=1.2,font=\footnotesize]
    \def\p{.2};
    \def\x{1}
    \def\y{.5}
    \begin{scope}[every node/.style={draw},scale=1.5,yscale=1.5]
        \node (L0) at (0,0) {$L_0$};
        \node (P012) at (\x,\y) {$P_{0,1,2}$};
        \node (P034) at (\x,0) {$P_{0,3,4}$};
        \node (P056) at (\x,-\y) {$P_{0,5,6}$};
        \foreach \i in {1,2,...,7}
        {
            \node (L\i) at (2*\x,1.75*\y-\i*\y/2) {$L_\i$};
        };
        \foreach[count=\i] \l in {{1,5,7},{1,4,6},{2,3,5},{2,4,7},{3,6,7}}
        {
            \node (P\i) at (4*\x,2*\y-.75*\i*\y) {$P_{\l}$};
        };
    \end{scope}
    \begin{scope}[nodes={circle,blue,fill=white,inner sep=1pt,font=\scriptsize}]
        \draw (L0.30) -- (P012.west) (L0.east) -- (P034.west) (L0.-30) -- (P056.west);
        \draw (P012) edge (L1.west) edge (L2.west);
        \draw (P034) edge (L3.west) edge (L4.west);
        \draw (P056) edge (L5.west) edge (L6.west);
        \draw (L0) edge[bend right=40] (L7.west);
        \draw (P1.160) -- (L1.30) (P1.west) -- (L5.30) (P1.-160) -- (L7.30);
        \draw (P2.160) -- (L1.east) (P2.west) -- (L4.30) (P2.-160) -- (L6.east);
        \draw (P3.160) -- (L2.30) (P3.west) -- (L3.east) (P3.-160) -- (L5.east);
        \draw (P4.160) -- (L2.east) (P4.west) -- (L4.east) (P4.-160) -- (L7.east);
        \draw (P5.160) -- (L3.-30) (P5.west) -- (L6.-30) (P5.-160) -- (L7.-30);
        \draw (L4) -- (L5);
        \draw (L2.-30) edge[bend left=50] (L6.30);
        \draw (L1.-30) edge[bend left=80] (L3.30);
    \end{scope}
\end{tikzpicture}

%% file: star_linear_order.pgf
\begin{tikzpicture}[scale=.7]
	\node[below=.1] (b0) at (0,0) {$b$};
	\fill[gray!20] (0,0) arc (-90:-112:10) -- ++ (-112:-2.5) node[shift=(-60:.45), black] {$D_W$} arc(-112:-68:7.5) -- ++ (-68:2.5) arc (-68:-90:10) ;
	\draw[dashed,->] (0,0) circle[radius=2];
	\node at (0,-1.6) {$U$};
	\draw[thick] (0,0) arc (-90:-112:10) node [below=.1] {$\partial^{\infty} D_W$};
	\draw[thick] (0,0) arc (-90:-68:10);
	\draw[red] (0,0) -- + (30:2.5) node[below right] {$\bar{\alpha}^{x}$};
	\draw[red,dotted] (30:2.5) -- (30:3.2);
	\draw[red] (0,0) -- + (60:2.5) node[right] {$\bar{\alpha}^{y}$};
	\draw[red,dotted] (60:2.5) -- (60:3.1);
	\draw[red] (0,0) -- + (150:2.5) node[below left] {$\bar{\alpha}^{z}$};
	\draw[red,dotted] (150:2.5) -- (150:3.2);
	\fill[black] (0,0) circle [radius=.1];
	\draw[blue,dotted,->] (70:1.2) arc (70:140:1.2) node[midway,above] {$\oplus$};
\end{tikzpicture}

%% file: star_circular_order.pgf
\begin{tikzpicture}[scale=.7]
	\def\r{2.5};
	\node[below=.1] (b0) at (0,0) {$b$};
	\draw[dashed,->] (0,0) circle[radius=.8*\r];
	\node at (-50:\r) {$U$};
	\begin{scope}[red,semithick]
		\draw (0,0) -- + (30:\r);
		\draw[dotted] (30:\r) -- (30:1.2*\r);
		\draw (0,0) -- + (140:\r) node[above right] {$\tilde{\alpha}^{x}$};
		\draw[dotted] (140:\r) -- (140:1.2*\r);
		\draw (0,0) -- + (180:\r) node[below] {$\tilde{\alpha}^{y}$};
		\draw[dotted] (180:\r) -- (180:1.2*\r);
		\draw (0,0) -- + (210:\r) node[below right] {$\tilde{\alpha}^{z}$};
		\draw[dotted] (210:\r) -- (210:1.2*\r);
		\draw (0,0) -- + (-30:\r);
		\draw[dotted] (-30:\r) -- (-30:1.2*\r);
		\draw[dotted] (50:.5*\r) arc (50:120:.5*\r);
		\draw[dotted] (230:.5*\r) arc (230:310:.5*\r);
	\end{scope}
	\fill[black] (0,0) circle [radius=.1];
	\draw[blue,dotted,->] (150:.95*\r) arc (150:170:.95*\r) node[midway,left] {$\oplus$};
\end{tikzpicture}

%% file: star_linear.pgf
\begin{tikzpicture}[scale=.55,font=\scriptsize]
	\def\r{.37};
	\fill[gray!20,draw=black,thick] (0,0) circle [radius=4];
	\node[label={[below]$b$},coordinate] (b0) at (0,-4) {};
	\node at (50:4.7) {$\partial^{\infty} D_W$};
	\foreach \x/\l/\t in {1/v/x,2/w/y,3/x/z,4/z/v,5/y/w}
	{
		\node[circle,draw,fill=white,thick] (\x) at ($(4*\x*\r,0)-(4*3*\r,0)$) {$\l$};
		\draw[red,semithick] (b0) -- +(\x*30:2) node[coordinate] (m\x) {};
	};
	\begin{scope}[red,semithick]
		\draw (m5) to[out=5*30,in=-145]  node[midway,right] {$\bar{\alpha}^{w}$} (2);
		\draw (m4) to[out=4*30,in=-70] (-2*\r,.5) to[out=110,in=90] node[midway,above] {$\bar{\alpha}^{v}$} (1);
		\draw (m3) to[out=3*30,in=-120] node[midway,right] {$\bar{\alpha}^{z}$} (4);
		\draw (m2) to[out=2*30,in=-120] node[midway,right] {$\bar{\alpha}^{y}$} (5);
		\draw (m1) to[out=1*30,in=-80] (10*\r,.5) to[out=100,in=90] node[midway,above] {$\bar{\alpha}^{x}$} (3);
	\end{scope}
	\fill[black] (0,-4) circle [radius=.1];
	\draw[dashed] (0,-4)+(-10:1.7) arc [start angle=-10,end angle=190,radius=1.7];
\end{tikzpicture}

%% file: star_circular.pgf
\newcommand\pgfmathsinandcos[3]{%
  \pgfmathsetmacro#1{sin(#3)}%
  \pgfmathsetmacro#2{cos(#3)}%
}
\newcommand\LongitudePlane[3][current plane]{%
  \pgfmathsinandcos\sinEl\cosEl{#2} 
  \pgfmathsinandcos\sint\cost{#3} 
  \tikzset{#1/.style={cm={\cost,\sint*\sinEl,0,\cosEl,(0,0)}}}
}
\newcommand\LatitudePlane[3][current plane]{%
  \pgfmathsinandcos\sinEl\cosEl{#2} 
  \pgfmathsinandcos\sint\cost{#3} 
  \pgfmathsetmacro\yshift{\cosEl*\sint}
  \tikzset{#1/.style={cm={\cost,0,0,\cost*\sinEl,(0,\yshift)}}} %
}
\newcommand\DrawLongitudeCircle[2][1]{
  \LongitudePlane{\angEl}{#2}
  \tikzset{current plane/.prefix style={scale=#1}}
  \pgfmathsetmacro\angVis{atan(sin(#2)*cos(\angEl)/sin(\angEl))} %
  \draw[current plane] (\angVis:1) arc (\angVis:\angVis+180:1);
  \draw[current plane,dotted] (\angVis-180:1) arc (\angVis-180:\angVis:1);
}
\newcommand\DrawLatitudeCircle[2][1]{
  \LatitudePlane{\angEl}{#2}
  \tikzset{current plane/.prefix style={scale=#1}}
  \pgfmathsetmacro\sinVis{sin(#2)/cos(#2)*sin(\angEl)/cos(\angEl)}
  \pgfmathsetmacro\angVis{asin(min(1,max(\sinVis,-1)))}
  \draw[current plane] (\angVis:1) arc (\angVis:-\angVis-180:1);
  \draw[current plane,dotted] (180-\angVis:1) arc (180-\angVis:\angVis:1);
}

\begin{tikzpicture}[scale=.6,font=\footnotesize]
	\def\R{1.5};
	\def\Rb{1.9};

	\def\angEl{30} 

    \filldraw[gray!20,draw=black] (0,0) circle (2*\Rb);
    \DrawLatitudeCircle[2*\Rb]{0}

	\node[above=.1] (b0) at (0,0) {$b$};
	\draw[dashed,->] (0,0) circle[radius=.7*\R];
	\node at (-50:\R) {$U$};
	\begin{scope}[nodes={circle,draw,fill=white,thick}]
		\foreach \i/\t/\r in {x/30/1,w/110/.7,z/0/.8,y/-60/.8,v/-180/.9}{
			\node (\i) at (\t:2*\r*\R) {${\i}$};
		}
	\end{scope}
	\begin{scope}[red,semithick]
		\draw (0,0) -- ++ (50:\R) node[right] {$\tilde{\alpha}^{v}$} to[out=50, in=0] (100:2*\R) to[out=180,in=90] (v);
		\draw (0,0) -- ++ (140:\R) node[below left] {$\tilde{\alpha}^{w}$} to[out=140, in=-135] (w);
		\draw (0,0) -- ++ (220:\R) node[left] {$\tilde{\alpha}^{x}$} to[out=220, in=180] (-80:2*1*\R) to[out=0, in=-120] (-30:2*1.1*\R)  to[out=60,in=-60] (x);
		\draw (0,0) -- ++ (-90:\R) node[left] {$\tilde{\alpha}^{y}$} to[out=-90, in=200] (y);
		\draw (0,0) -- ++ (-10:\R) node[above] {$\tilde{\alpha}^{z}$} to[out=-10, in=-150] (z);
	\end{scope}
	\fill[black] (0,0) circle [radius=.1];
	\node at (50:4.5) {$\Sigma_W$};
\end{tikzpicture}

%% file: star_linear_standard.pgf
\begin{tikzpicture}[scale=.7,font=\scriptsize]
	\def\r{3};
	\fill[gray!20,draw=black,thick] (0,0) circle [radius=\r];
	\node[below] (b0) at (0,-\r) {$-i$};
	\node at (50:1.2*\r) {$\Delta_m$};
	\draw[semithick,blue] (-\r,0) -- (\r,0);
	\node[above,blue] at (-.25*\r,0) {$\gamma_0$};
	\foreach \x in {-.75,.25,.75}
	{
		\draw[red,semithick] (\x*\r,0) -- (0,-\r);
		\fill[black] (\x*\r,0) circle [radius=.1];
	};
	\node[left] at (-\r,0) {$-1$};
	\node[right] at (\r,0) {$1$};
	\fill[black] (-\r,0) circle [radius=.1];
	\fill[black] (\r,0) circle [radius=.1];
	\fill[black] (0,-\r) circle [radius=.1];
	\node[above] at (-.75*\r,0) {$x^{m}$};
	\node[above] at (.25*\r,0) {$x^2$};
	\node[above] at (.75*\r,0) {$x^1$};
	\begin{scope}[red]
		\path (-.75*\r,0) -- (0,-\r) node[pos=.25,left] {$\bar{\alpha}_0^{m}$};
		\path (.25*\r,0) -- (0,-\r) node[pos=.25,shift=(0:.25)] {$\bar{\alpha}_0^2$};
		\path (.75*\r,0) -- (0,-\r) node[pos=.25,right] {$\bar{\alpha}_0^1$};
	\end{scope}
	\draw[dotted,thick,red] (-.5*\r,-.2*\r) -- (.1*\r,-.2*\r);
\end{tikzpicture}

%% file: star_linear_neigh.pgf
\begin{tikzpicture}[scale=.55,font=\scriptsize]
	\def\r{4};
	\def\t{7pt}
	\fill[gray!20,thick] (0,0) circle [radius=\r];
	\node[below] (b0) at (0,-1.05*\r) {$-i$};
	\node at (50:1.2*\r) {$\Delta_m$};
	\draw[dashed] (-\r,0) -- (\r,0);
	\draw[blue,double=white,double distance=\t,line cap=round] (0,-\r) -- ++(45:\r/4) to[out=45,in=-90] (.6*\r,0) node[above=\t/3,black] {$x^1$};
	\draw[blue,double=white,double distance=\t,line cap=round] (0,-\r) -- ++(90:\r/6) to[out=90,in=-90] (.4*\r,-.1*\r) -- (.4*\r,0*\r) to[out=90,in=90,looseness=1.5] (-.8*\r,0*\r) to[out=-90,in=-100,looseness=1.5] (0,0) node[above left,black] {$x^2$};
	\draw[blue,double=white,double distance=\t,line cap=round] (0,-\r) -- ++(135:\r/4) to[out=135,in=-90,looseness=1.5] (.2*\r,0) to[out=90,in=90,looseness=1.5] (-.6*\r,0) node[below right,black] {$x^3$};
	\foreach \x in {45,90,135}
	{
		\draw[white,line width=.95*\t] (0,-\r) -- ++(\x:\r/6);
	};
	\begin{scope}[red,semithick]
		\draw (0,-\r) -- ++(45:\r/4) to[out=45,in=-90] (.6*\r,0);
		\draw (0,-\r) -- ++(90:\r/6) to[out=90,in=-90] (.4*\r,-.1*\r) -- (.4*\r,0*\r) to[out=90,in=90,looseness=1.5] (-.8*\r,0*\r) to[out=-90,in=-100,looseness=1.5] (0,0);
		\draw (0,-\r) -- ++(135:\r/4) to[out=135,in=-90,looseness=1.5] (.2*\r,0) to[out=90,in=90,looseness=1.5] (-.6*\r,0);
	\end{scope}
	\draw[thick] (0,0) circle [radius=\r];
	\foreach \x in {-.6,0,.6}
	{
		\fill[black] (\x*\r,0) circle [radius=.1];
	};
	\node[left] at (-\r,0) {$-1$};
	\node[right] at (\r,0) {$1$};
	\fill[black] (-\r,0) circle [radius=.1] (\r,0) circle [radius=.1] (0,-\r) circle [radius=.1];
\end{tikzpicture}

%% file: star_neighbourhood.pgf
\begin{tikzpicture}[scale=.8,font=\scriptsize]
	\def\x{3.5};
	\def\y{1.5};
	\draw[thick] (-\x,0) -- (-\x,\y) (\x,\y) -- (\x,0);
	\draw[blue,semithick] (-\x,\y) -- (-.35*\x,\y) (-.15*\x,\y) -- (\x,\y);
	\draw[blue,semithick,dotted] (-.35*\x,\y) -- (-.15*\x,\y);
	\draw[red,semithick] (-\x,0) -- (-.35*\x,0) (-.15*\x,0) -- (\x,0);
	\draw[red,semithick,dotted] (-.35*\x,0) -- (-.15*\x,0);
	\foreach \k in {-.75,.25,.75}
	{
		\draw[blue,semithick] (\k*\x,0) -- (\k*\x-.05*\x,\y);
		\draw[blue,semithick] (\k*\x,0) -- (\k*\x+.05*\x,\y);
		\draw[white,semithick,dotted] (\k*\x-.05*\x,\y) -- (\k*\x+.05*\x,\y);
		\fill[black] (\k*\x,0) circle (.1);
	};
	\fill[black] (-\x,\y) circle (.1) (\x,\y) circle (.1);
	\foreach \k in {-1,-.5,0,.5,1}
	{
		\node[cross out,draw,inner sep=2pt] at (\k*\x,0) {};
		\node[below] at (\k*\x,0) {$-i$};
	};
	\node[left] at (-\x,\y) {$-1$};
	\node[right] at (\x,\y) {$1$};
	\node[below] at (-.75*\x,0) {$x^m$};
	\node[below] at (.25*\x,0) {$x^2$};
	\node[below] at (.75*\x,0) {$x^1$};
	\node[above,blue] at (0,\y) {$V(\bar{\alpha})$};
\end{tikzpicture}

%% file: mcg_action_push.pgf
\newcommand\pgfmathsinandcos[3]{%
  \pgfmathsetmacro#1{sin(#3)}%
  \pgfmathsetmacro#2{cos(#3)}%
}
\newcommand\LongitudePlane[3][current plane]{%
  \pgfmathsinandcos\sinEl\cosEl{#2} 
  \pgfmathsinandcos\sint\cost{#3} 
  \tikzset{#1/.style={cm={\cost,\sint*\sinEl,0,\cosEl,(0,0)}}}
}
\newcommand\LatitudePlane[3][current plane]{%
  \pgfmathsinandcos\sinEl\cosEl{#2} 
  \pgfmathsinandcos\sint\cost{#3} 
  \pgfmathsetmacro\yshift{\cosEl*\sint}
  \tikzset{#1/.style={cm={\cost,0,0,\cost*\sinEl,(0,\yshift)}}} %
}
\newcommand\DrawLongitudeCircle[2][1]{
  \LongitudePlane{\angEl}{#2}
  \tikzset{current plane/.prefix style={scale=#1}}
  \pgfmathsetmacro\angVis{atan(sin(#2)*cos(\angEl)/sin(\angEl))} %
  \draw[current plane] (\angVis:1) arc (\angVis:\angVis+180:1);
  \draw[current plane,dotted] (\angVis-180:1) arc (\angVis-180:\angVis:1);
}
\newcommand\DrawLatitudeCircle[2][1]{
  \LatitudePlane{\angEl}{#2}
  \tikzset{current plane/.prefix style={scale=#1}}
  \pgfmathsetmacro\sinVis{sin(#2)/cos(#2)*sin(\angEl)/cos(\angEl)}
  \pgfmathsetmacro\angVis{asin(min(1,max(\sinVis,-1)))}
  \draw[current plane] (\angVis:1) arc (\angVis:-\angVis-180:1);
  \draw[current plane,dotted] (180-\angVis:1) arc (180-\angVis:\angVis:1);
  \node[label={[below]$b$},coordinate] (b0) at (\angVis-90:1) {};
}

\begin{tikzpicture}[scale=.6,xscale=-1,font=\footnotesize]
  \def\R{1.5};
  \def\Rb{1.9};

  \def\angEl{15} 

  \filldraw[ball color=white] (0,0) circle (2*\Rb);
  \DrawLatitudeCircle[2*\Rb]{0}

  \begin{scope}[nodes={circle,draw,fill=white,thick}]
    \foreach \i/\t/\r in {w/140/.6,z/30/1.1,v/95/.6}{
        \node (\i) at ($(b0)+(\t:2*\r*\R)$) {${\i}$};
    }
  \end{scope}
  \begin{scope}[red,semithick]
    \draw (b0) -- (v) (b0) -- (w);
    \node[coordinate] at ($(b0)+(30:.58*2*\R)$) (z1) {};
    \node[coordinate] at ($(b0)+(30:.78*2*\R)$) (z2) {};
    \draw (b0) -- (z1);
    \draw[dotted] (z1) -- (z2);
    \draw (z2) -- (z);
  \end{scope}

  \begin{scope}[blue,dashed,rotate=30]
    \draw (b0)+(0,.5*\R) ellipse [x radius=1.1*\R,y radius=1.6*\R];
    \draw (b0)+(0,.5*\R) ellipse [x radius=1.7*\R,y radius=2.2*\R];
  \end{scope}
  \begin{scope}[red,semithick,rotate=30]
    \draw (z1) arc [x radius=1.3*\R,y radius=1.7*\R,start angle=-15,end angle=-95] node[coordinate] (zz1) {};
    \draw (z2) arc [x radius=1.6*\R,y radius=2.1*\R,start angle=-15,end angle=-95] node[coordinate] (zz2) {};
    \draw (zz1) to[out=-170,in=-170] (zz2);
  \end{scope}
  \node (p) at ($(b0)+(-60:1.4*\R)$) {};
  \node (pl) at ($(b0)+(-85:1.4*\R)$) {$b_{\infty}$};

  \fill[black] (b0) circle [radius=.1] (p) circle [radius=.08];
\end{tikzpicture}

%% file: mcg_action_push_end.pgf
\newcommand\pgfmathsinandcos[3]{%
  \pgfmathsetmacro#1{sin(#3)}%
  \pgfmathsetmacro#2{cos(#3)}%
}
\newcommand\LongitudePlane[3][current plane]{%
  \pgfmathsinandcos\sinEl\cosEl{#2} 
  \pgfmathsinandcos\sint\cost{#3} 
  \tikzset{#1/.style={cm={\cost,\sint*\sinEl,0,\cosEl,(0,0)}}}
}
\newcommand\LatitudePlane[3][current plane]{%
  \pgfmathsinandcos\sinEl\cosEl{#2} 
  \pgfmathsinandcos\sint\cost{#3} 
  \pgfmathsetmacro\yshift{\cosEl*\sint}
  \tikzset{#1/.style={cm={\cost,0,0,\cost*\sinEl,(0,\yshift)}}} %
}
\newcommand\DrawLongitudeCircle[2][1]{
  \LongitudePlane{\angEl}{#2}
  \tikzset{current plane/.prefix style={scale=#1}}
  \pgfmathsetmacro\angVis{atan(sin(#2)*cos(\angEl)/sin(\angEl))} %
  \draw[current plane] (\angVis:1) arc (\angVis:\angVis+180:1);
  \draw[current plane,dotted] (\angVis-180:1) arc (\angVis-180:\angVis:1);
}
\newcommand\DrawLatitudeCircle[2][1]{
  \LatitudePlane{\angEl}{#2}
  \tikzset{current plane/.prefix style={scale=#1}}
  \pgfmathsetmacro\sinVis{sin(#2)/cos(#2)*sin(\angEl)/cos(\angEl)}
  \pgfmathsetmacro\angVis{asin(min(1,max(\sinVis,-1)))}
  \draw[current plane] (\angVis:1) arc (\angVis:-\angVis-180:1);
  \draw[current plane,dotted] (180-\angVis:1) arc (180-\angVis:\angVis:1);
  \node[label={[below]$b$},coordinate] (b0) at (\angVis-90:1) {};
}

\begin{tikzpicture}[scale=.6,xscale=-1,font=\footnotesize]
  \def\R{1.5};
  \def\Rb{1.9};

  \def\angEl{15} 

  \filldraw[ball color=white] (0,0) circle (2*\Rb);
  \DrawLatitudeCircle[2*\Rb]{0}

  \begin{scope}[nodes={circle,draw,fill=white,thick}]
    \foreach \i/\t/\r in {w/140/.6,z/30/1.1,v/95/.6}{
        \node (\i) at ($(b0)+(\t:2*\r*\R)$) {${\i}$};
    }
  \end{scope}
  \begin{scope}[red,semithick]
    \draw (b0) -- (v) (b0) -- (w);
    \draw[dotted] (b0) -- (z);
    \draw (b0) -- ++(165:1.5*\R) to[out=165,in=200] ($(b0)+(90:2.3*\R)$) to[out=20,in=110] ($(b0)+(30:2.65*\R)$) to[out=-70,in=-80,looseness=2.5] (z.-30);
  \end{scope}

  \node (p) at ($(b0)+(-60:1.4*\R)$) {};
  \node (pl) at ($(b0)+(-85:1.4*\R)$) {$b_{\infty}$};

  \fill[black] (b0) circle [radius=.1] (p) circle [radius=.08];
\end{tikzpicture}

%% file: cycle_difference_map.pgf
\begin{tikzpicture}[font=\scriptsize,3d view={0}{115},scale=1,>=Stealth]
	\tikzset{-<-/.style={decoration={markings, mark=at position #1 with {\arrow{<}}},postaction={decorate}}}
	\tikzset{->-/.style={decoration={markings, mark=at position #1 with {\arrow{>}}},postaction={decorate}}}
	\def\t{1.7};
	\def\ar{.8};
    \def\xv{2.5};
    \def\hd{2.7};
    \def\vd{3};
    \clip (-2.1*\hd,-.7*\vd) rectangle (2.1*\hd,.7*\vd);
	\foreach \x/\s/\col/\colg/\colc/\cold [count=\i] in {-\hd/1/green/black/black/red,\hd/-1/red/black/black/green}
	{
		\begin{scope}[canvas is xy plane at z=0,shift={(\x,0,0)},xscale=\s]
			\draw[dashed] (-\xv,0) arc [radius=2*\xv,start angle=0, end angle=-20];
			\draw[dashed,semithick,\col] (.4*\xv,0) arc [radius=2.4*\xv,start angle=0, end angle=-19];
			\draw[dotted] (.8*\xv,0) arc [radius=2.8*\xv,start angle=0, end angle=-16];
		\end{scope}
		\begin{scope}[canvas is xz plane at y=0,shift={(\x,0,0)},xscale=\s]
				\fill[gray!15,semitransparent,even odd rule]
					(0,0) circle [radius=\xv]
					(-.6*\xv,0) circle [radius=.2*\xv]
					(0,0) circle [radius=.2*\xv]
					(.6*\xv,0) circle [radius=.2*\xv];
				\fill[white] (-1.5*\xv,2*\xv) rectangle (\xv,2.8*\xv);
				\draw
					(0,0) circle [radius=\xv]
					(-.6*\xv,0) circle [radius=.2*\xv]
					(0,0) circle [radius=.2*\xv];
				\begin{scope}
					\draw[\colg,-<-=\ar] (-.6*\xv,0) circle [radius=.2*\xv];
					\draw[\colc,-<-=\ar] (0,0) circle [radius=.2*\xv];
					\draw[dashed,\cold,-<-=\ar,thick] (.6*\xv,0) circle [radius=.2*\xv];
				\end{scope}
				\node[coordinate] (b\i) at (0,\xv) {\i};
				\node[coordinate] (t\i) at (.4*\xv,0) {};
		\end{scope}
		\begin{scope}[canvas is xy plane at z=0,shift={(\x,0,0)},xscale=\s]
			\draw[dashed,-<-=.7] (-\xv,0) arc [radius=2*\xv,start angle=0, end angle=20];
			\draw[dashed,semithick,\col,-<-=.7] (.4*\xv,0) arc [radius=2.4*\xv,start angle=0, end angle=19];
			\draw[dotted] (.8*\xv,0) arc [radius=2.8*\xv,start angle=0, end angle=16];
			\fill[white] (-1.5*\xv,\xv) rectangle (\xv,1.3*\xv);
		\end{scope}
	};
	\begin{scope}[canvas is xz plane at y=0,shift={(-\hd,0,0)},xscale=1]
		\draw[thick,blue,-<-=.6] (b1) cos (t1);
		\draw[thick,teal,-<-=.7] (b1) .. controls (-.6*\xv,.8*\xv) and (-.3*\xv,-1.5*\xv) .. (t1);

		\fill[yellow] (t1) circle [radius=.1];
		\fill[black] (b1) circle [radius=.1];
	\end{scope}
	\begin{scope}[canvas is xz plane at y=0,shift={(\hd,0,0)},xscale=-1]
		\draw[thick,teal,->-=.7] (b2) sin (t2);
		\draw[thick,blue] (b2) edge[out=-30,in=75,->-=.5] (t2);

		\fill[yellow] (t2) circle [radius=.1];
		\fill[black] (b2) circle [radius=.1];
	\end{scope}
	\node at (-1.6*\hd,.5*\vd) {$S_v$};
	\node at (1.6*\hd,.5*\vd) {$S_w$};
	\node[teal] at (-1.1*\hd,-.5*\vd) {$\bar{\alpha}_v^w$};
	\node[teal] at (1.1*\hd,-.5*\vd) {$\bar{\alpha}_w^v$};
	\node[blue] at (-.9*\hd,-.5*\vd) {$\bar{\alpha}_v^{w \prime}$};
	\node[blue] at (.9*\hd,-.5*\vd) {$\bar{\alpha}_w^{v \prime}$};
\end{tikzpicture}

%% file: mcg_action_initial.pgf
\begin{tikzpicture}[scale=1,xscale=-1,font=\small]
	\def\r{.5};
	\def\x{2.5};
	\def\s{.2};

	\node at (-\x,0) {$D^k$};
	\node at (\x,0) {$D^l$};

	\begin{scope}[semithick]
		\draw[blue] (-\x,-\r) -- ++(0,-8*\s) node[below] {$\bar{\alpha}^k$};
		\draw[green] (\x,-\r) -- ++(0,-8*\s) node[below] {$\bar{\alpha}^l$};
		\draw[red] (0,\r+8*\s) -- (0,-\r-8*\s) node[below] {$\bar{\alpha}^j$};
	\end{scope}

	\node at (-\x,\r+8*\s+0) {$\delta_{k,l}$};

	\draw[dashed] (-\x,-\r-6*\s) arc [start angle=-90,end angle=-270,radius=\r+6*\s];
	\draw[dashed] (\x,-\r-6*\s) arc [start angle=-90,end angle=90,radius=\r+6*\s];

	\draw[dashed] (-\x,-\r-6*\s) -- (\x,-\r-6*\s);
	\draw[dashed] (-\x,+\r+6*\s) -- (\x,+\r+6*\s);

	\foreach \xi in {-\x,\x}
	{
		\draw[black,thick] (\xi,0) circle [radius=\r];
		\fill[black] (\xi,-\r+0) circle [radius=.07];
	};

\end{tikzpicture}

%% file: mcg_action_pure_braid.pgf
\begin{tikzpicture}[scale=1,xscale=-1,font=\small]
	\def\r{.5};
	\def\x{2.5};
	\def\s{.2};

	\begin{scope}[semithick]
		\draw[blue] (-\x,-\r) -- ++(0,-\s) node[coordinate] (bllt) {};
		\draw[green] (\x,-\r) -- ++(0,-\s) node[coordinate] (bjjt) {};

		\draw[blue] (-\x,-\r-5*\s) node[coordinate] (bllb) {} -- ++(0,-3*\s) node[below] {$\bar{\alpha}^k \cdot a_{k,l}$};
		\draw[green] (\x,-\r-5*\s) node[coordinate] (bjjb) {} -- ++(0,-3*\s) node[below] {$\bar{\alpha}^l \cdot a_{k,l}$};
		\draw[red] (0,-\r-5*\s) node[coordinate] (bkkb) {} -- ++(0,-3*\s) node[below] {$\bar{\alpha}^j \cdot a_{k,l}$};

		\draw[red] (\x,-\r-2*\s) node[coordinate] (bjkt) {} to[out=180,in=0] (-\x,\r+2*\s) node[coordinate] (tlkb) {};

		\draw[green] (bjjb) -- (-\x,-\r-3*\s) node[coordinate] (blj) {};
		\draw[red] (bkkb) -- (-\x,-\r-4*\s) node[coordinate] (blkb) {};
		\draw[blue] (bllt) -- (\x,-\r-3*\s) node[coordinate] (bjl) {};
		\draw[red] (\x,-\r-4*\s) node[coordinate] (bjkb) {} -- (-\x,-\r-2*\s) node[coordinate] (blkt) {};

		\draw[red] (0,\r+5*\s) node[coordinate] (tkkt) {} -- ++(0,+3*\s);

		\draw[red] (tkkt) -- (\x,\r+4*\s) node[coordinate] (tjkt) {};
		\draw[blue] (-\x,\r+5*\s) node[coordinate] (tll) {} -- (\x,\r+3*\s) node[coordinate] (tjl) {};
		\draw[red] (-\x,\r+4*\s) node[coordinate] (tlkt) {} -- (\x,\r+2*\s) node[coordinate] (tjkb) {};
		\draw[green] (-\x,\r+3*\s) node[coordinate] (tlj) {} -- (\x,\r+\s) node[coordinate] (tjj) {};

		\draw[blue] (bllb) arc [start angle=-90,end angle=-270,radius=\r+5*\s];
		\draw[red] (blkb) arc [start angle=-90,end angle=-270,radius=\r+4*\s];
		\draw[green] (blj) arc [start angle=-90,end angle=-270,radius=\r+3*\s];
		\draw[red] (blkt) arc [start angle=-90,end angle=-270,radius=\r+2*\s];

		\draw[blue] (bjl) arc [start angle=-90,end angle=90,radius=\r+3*\s];
		\draw[red] (bjkb) arc [start angle=-90,end angle=90,radius=\r+4*\s];
		\draw[green] (bjjt) arc [start angle=-90,end angle=90,radius=\r+\s];
		\draw[red] (bjkt) arc [start angle=-90,end angle=90,radius=\r+2*\s];
	\end{scope}

	\node at (-\x,0) {$D^k$};
	\node at (\x,0) {$D^l$};

	\node at (-\x,\r+8*\s) {$\delta_{k,l}$};

	\draw[dashed] (-\x,-\r-6*\s) arc [start angle=-90,end angle=-270,radius=\r+6*\s];
	\draw[dashed] (\x,-\r-6*\s) arc [start angle=-90,end angle=90,radius=\r+6*\s];

	\draw[dashed] (-\x,-\r-6*\s) -- (\x,-\r-6*\s);
	\draw[dashed] (-\x,\r+6*\s) -- (\x,\r+6*\s);

	\foreach \xi in {-\x,\x}
	{
		\draw[black,thick] (\xi,0) circle [radius=\r];
		\fill[black] (\xi,-\r) circle [radius=.07];
	};
\end{tikzpicture}

%% file: mcg_action_initial_dt.pgf
\begin{tikzpicture}[scale=.9,xscale=-1,font=\small]

	\def\r{.5};

	\draw[semithick,blue] (0,-\r) -- ++(0,-2.1*\r) node[below] {$\bar{\alpha}^j$};

	\node at (0,0) {$D^j$};

	\node at (-1.2+0,1.2) {$\delta_{j}$};

	\draw[dashed] (0,0) circle [radius=\r+.8];

	\draw[black,thick] (0,0) circle [radius=\r];
	\fill[black] (0,-\r) circle [radius=.07];

\end{tikzpicture}

%% file: mcg_action_dehn_twist.pgf
\newcommand\spiral{} 

\begin{tikzpicture}[scale=.9,xscale=-1,font=\small]
	\def\spiral[#1](#2)(#3:#4)(#5:#6)[#7]{
	\pgfmathsetmacro{\domain}{#4+#7*360}
	\pgfmathsetmacro{\growth}{180*(#6-#5)/(pi*(\domain-#3))}
	\draw [#1,
		shift={(#2)},
		domain=#3*pi/180:\domain*pi/180,
		variable=\t,
		smooth,
		samples=int(\domain/5)] plot ({\t r}: {#5+\growth*\t-\growth*#3*pi/180})
	}
	\def\r{.5};
	\def\x{5};

	\draw[semithick,blue] (0,-\r) -- ++(0,-.2);
	\draw[semithick,blue] (0,-\r-.6) -- ++(0,-.5) node[below] {$\bar{\alpha}^j \cdot d_{j}$};

	\node at (0,0) {$D^j$};

	\node at (-1.2,1.2) {$\delta_{j}$};

	\draw[dashed] (0,0) circle [radius=\r+.8];

	\draw[black,thick] (0,0) circle [radius=\r];
	\fill[black] (0,-\r) circle [radius=.07];

	\spiral[blue,thick](0,0)(-90:-90)(\r+.2:\r+.2)[1];

\end{tikzpicture}

%% file: mcg_action_initial_bis.pgf
\begin{tikzpicture}[scale=1,xscale=-1,font=\small]
	\def\r{.5};
	\def\x{1.3};
	\def\s{.15};

	\begin{scope}[semithick]
		\draw[blue] (-\x,-\r) node[coordinate] (bll) {} -- ++(0,-8*\s) node[below] {$\bar{\alpha}_0^k$};
		\draw[green] (\x,-\r) node[coordinate] (bjj) {} -- ++(0,-8*\s) node[below] {$\bar{\alpha}_0^{k+1}$};
	\end{scope}

	\node at (-\x,0) {$D^k$};
	\node at (\x,0) {$D^{k+1}$};

	\node at (-\x,\r+8*\s) {$\delta_{k,k+1}$};

	\draw[dashed] (-\x,-\r-6*\s) arc [start angle=-90,end angle=-270,radius=\r+6*\s];
	\draw[dashed] (\x,-\r-6*\s) arc [start angle=-90,end angle=90,radius=\r+6*\s];

	\draw[dashed] (-\x,-\r-6*\s) -- (\x,-\r-6*\s);
	\draw[dashed] (-\x,\r+6*\s) -- (\x,\r+6*\s);

	\foreach \xi in {-\x,\x}
		{
			\draw[black,thick] (\xi,0) circle [radius=\r];
			\fill[black] (\xi,-\r) circle [radius=.07];
		};

\end{tikzpicture}

%% file: mcg_action_braid.pgf
\begin{tikzpicture}[scale=1,xscale=-1,font=\small]
	\def\r{.5};
	\def\x{1.3};
	\def\s{.15};

	\begin{scope}[semithick]
		\draw[green] (-\x,\r) -- ++(0,2*\s) node[coordinate] (tllt) {};
		\draw[blue] (\x,\r) -- ++(0,2*\s) node[coordinate] (tjjt) {};

		\draw[blue] (-\x,-\r-4*\s) node[coordinate] (bll) {} -- ++(0,-4*\s) node[below] {$\bar{\alpha}_0^k \cdot \sigma_{k,k+1}$};
		\draw[green] (\x,-\r-4*\s) node[coordinate] (bjj) {} -- ++(0,-4*\s) node[below] {$\bar{\alpha}_0^{k+1} \cdot \sigma_{k,k+1}$};

 		\draw[green] (bjj) -- (-\x,-\r-2*\s) node[coordinate] (blj) {};
 		\draw[blue] (-\x,\r+4*\s) node[coordinate] (tll) {} -- (\x,\r+2*\s) node[coordinate] (tjl) {};
 		\draw[blue] (bll) arc [start angle=-90,end angle=-270,radius=\r+4*\s];
 		\draw[green] (blj) arc [start angle=-90,end angle=-270,radius=\r+2*\s];
	\end{scope}

	\node at (-\x,0) {$D^k$};
	\node at (\x,0) {$D^{k+1}$};

	\node at (-\x,\r+8*\s) {$\delta_{k,k+1}$};

	\draw[dashed] (-\x,-\r-6*\s) arc [start angle=-90,end angle=-270,radius=\r+6*\s];
	\draw[dashed] (\x,-\r-6*\s) arc [start angle=-90,end angle=90,radius=\r+6*\s];

	\draw[dashed] (-\x,-\r-6*\s) -- (\x,-\r-6*\s);
	\draw[dashed] (-\x,\r+6*\s) -- (\x,\r+6*\s);

	\foreach \xi in {-\x,\x}
		{
			\draw[black,thick] (\xi,0) circle [radius=\r];
			\fill[black] (\xi,\r) circle [radius=.07];
		};

\end{tikzpicture}